\documentclass[a4paper]{amsart}
\usepackage[margin=1in,headheight=13.6pt]{geometry}
\usepackage{latexsym}\usepackage{ifthen}\usepackage[leqno]{amsmath}
\usepackage{enumerate}\usepackage{calc}\usepackage{hyphenat}
\usepackage{anyfontsize}
\usepackage{moresize}
\usepackage{mathtools}
\usepackage{bbm}
\usepackage{amstext,amsbsy,amsopn,amsthm,amsgen, amsmath}
\usepackage{amsfonts,amscd,amsxtra,upref}
\usepackage{graphicx,varwidth}
\graphicspath{{/Users/burii/Library/CloudStorage/Dropbox/SCDB/Hodge number of Kummer fibrations/}}
\usepackage{wrapfig}
\usepackage{float}
\usepackage{amssymb,pdflscape,changepage}

\usepackage{array,amsmath}

\usepackage{afterpage}
\usepackage{comment}
\usepackage{expl3,xparse}

\usepackage{multirow}
\usepackage[table]{xcolor}
\usepackage{array,booktabs}

\usepackage{tabularx}
\usepackage{stackrel}
\usepackage{caption, float}
\usepackage{hhline}
\usepackage{dirtytalk}

\usepackage{pdfpages}
\usepackage{placeins}

\usepackage{picture,slashbox}
\usepackage{dynkin-diagrams}

\usepackage{epstopdf}
\usepackage{dsfont}
\usepackage[hidelinks]{hyperref}
\usepackage{mathrsfs}
\usepackage{euscript, amssymb}
\theoremstyle{plain}

\newtheorem{Thm}{Theorem}[section]

\newtheorem{Lem}[Thm]{Lemma}

\newtheorem{Cor}[Thm]{Corollary}
\newtheorem{Pro}[Thm]{Proposition}

\theoremstyle{definition}

\newtheorem{Exm}[Thm]{Example}

\theoremstyle{remark}
\newtheorem{Rem}[Thm]{Remark}
\numberwithin{equation}{section}

\usepackage{rotating}

\usepackage{scalerel}
\usepackage{stackengine,wasysym}
\newcommand\tylda[1]{\ThisStyle{%
		\setbox0=\hbox{$\SavedStyle#1$}%
		\stackengine{-.1\LMpt}{$\SavedStyle#1$}{%
			\stretchto{\scaleto{\SavedStyle\mkern.2mu\AC}{.5150\wd0}}{.6\ht0}%
		}{O}{c}{F}{T}{S}%
}}

\newcommand\tyldaa[1]{\ThisStyle{%
		\setbox0=\hbox{$\SavedStyle#1$}%
		\stackengine{-.1\LMpt}{$\SavedStyle#1$}{%
			\stretchto{\scaleto{\SavedStyle\mkern -1.5mu\AC}{.5150\wd0}}{.5\ht0}%
		}{O}{c}{F}{T}{S}%
}}

\newcommand*\quot[2]{{^{\textstyle #1}\big/_{\textstyle #2}}}

\newcommand{\myData}[1][]{
\author[D.\ Burek]{Dominik Burek}
\address{D.\ Burek{}\\{}
Faculty of Mathematics and Computer Science\\{}Jagiellonian University,
\L{}ojasiewicza 6\\{}30-348 Krak\'{o}w\\{}Poland}
\email{dominik.burek@uj.edu.pl}
}

\newcommand{\bossData}[1][]{
\author[S.\ Cynk]{Sławomir Cynk}
\address{S.\ Cynk{}\\{}
Faculty of Mathematics and Computer Science\\{}Jagiellonian University,
\L{}ojasiewicza 6\\{}30-348 Krak\'{o}w\\{}Poland}
\email{slawomir.cynk@uj.edu.pl}
}

\newcommand{\ZZ}{\mathbb{Z}}
\newcommand{\CC}{\mathbb{C}}

\newcommand{\PP}{\mathbb{P}}

\newcommand{\FF}{\mathbb{F}}

\newcommand{\Orb}{\operatorname{Orb}}
\newcommand{\MW}{\operatorname{MW}}

\newcommand{\Fix}{\operatorname{Fix}}

\newcommand{\overbar}[1]{\mkern 1.5mu\overline{\mkern-1.5mu#1\mkern-1.5mu}\mkern 1.5mu}
\newcommand\restrict[1]{\raisebox{-.5ex}{$|$}_{#1}}
\renewcommand{\bar}{\overbar}

\def\restrict#1{\raise-.5ex\hbox{\ensuremath|}_{#1}}

\makeatletter
\DeclareRobustCommand{\iscircle}{\mathord{\mathpalette\is@circle\relax}}
\newcommand\is@circle[2]{%
  \begingroup
  \sbox\z@{\raisebox{\depth}{$\m@th#1\bigcirc$}}%
  \sbox\tw@{$#1\square$}%
  \resizebox{!}{\ht\tw@}{\usebox{\z@}}%
  \endgroup
}
\makeatother

\definecolor{ziel}{rgb}{0.0, 0.5, 0.0}
\usepackage{amssymb}
\usepackage{pifont}
\newcommand{\no}{\textsc{N}}

\newcommand{\fb}{\setlength{\fboxrule}{3pt}}
\newcommand{\szary}{\cellcolor{gray!30}}
\newcommand{\ww}[1]{\fcolorbox{white}{white}{#1}}
\newcommand{\wg}[1]{\fcolorbox{gray!30}{gray!30}{#1}}
\newcommand{\np}{\textsc{A}}
\newcommand{\p}{\textsc{P}}

\usepackage{tikz}

\usepackage{enumitem}
\usepackage{tikz-cd}
\usepackage{tikz}
\usetikzlibrary{calc,intersections}
\usepackage{xcolor}

\tikzset{
  fiber/.style={line cap=round, line join=round},
  fiberline/.style={fiber, line width=0.45pt},
  doubleline/.style={fiber, line width=1.15pt},
  ruling/.style={fiber, line width=0.28pt, draw=black!35},
}

\newcommand{\fibericon}[2][]{%
  \tikz[baseline=-0.45ex,scale=0.55,#1]{#2}%
}

\newcommand{\DoublePone}{%
  \fibericon{ \draw[doubleline] (0,-0.5) -- (0,0.5); }%
}

\newcommand{\Cross}{%
  \fibericon{%
    \draw[fiberline] (-0.45,-0.45) -- (0.45,0.45);
    \draw[fiberline] (-0.45,0.45) -- (0.45,-0.45);
  }%
}

\definecolor{kodblue}{RGB}{0,114,178}      
\definecolor{kodmagenta}{RGB}{204,121,167} 
\definecolor{kodviolet}{RGB}{117,112,179}  
\definecolor{kodorange}{RGB}{230,159,0}    
\definecolor{kodgreen}{RGB}{0,158,115}     
\definecolor{kodcyan}{RGB}{86,180,233}     
\definecolor{kodbrown}{RGB}{166,118,29}    
\definecolor{kodc1}{RGB}{0,114,178}
\definecolor{kodc2}{RGB}{213,94,0}
\definecolor{kodc3}{RGB}{117,112,179}
\definecolor{kodc4}{RGB}{230,159,0}
\definecolor{kodc5}{RGB}{0,158,115}
\definecolor{kodc6}{RGB}{86,180,233}
\definecolor{kodc7}{RGB}{166,118,29}
\definecolor{kodc8}{RGB}{228,26,28}
\definecolor{kodc9}{RGB}{140,140,140}
\definecolor{kodc10}{RGB}{204,121,167}
\definecolor{kodc11}{RGB}{102,166,30}
\definecolor{kodc12}{RGB}{153,102,204}
\definecolor{kodc13}{RGB}{240,228,66}
\definecolor{kodred}{RGB}{220,40,40}
\definecolor{kodii}{RGB}{200,40,40}

\tikzset{
  fiber/.style={line width=1.25pt, line cap=round},
  markA1/.style={circle, draw=kodblue, fill=kodblue!35, line width=0.9pt, minimum size=3.2mm, inner sep=0pt},    
  markA2/.style={circle, draw=kodorange, fill=kodorange!35, line width=0.9pt, minimum size=3.2mm, inner sep=0pt},
  markA3/.style={circle, draw=kodgreen, fill=kodgreen!35, line width=0.9pt, minimum size=3.2mm, inner sep=0pt},  
  markA4/.style={circle, draw=kodmagenta, fill=kodmagenta!35, line width=0.9pt, minimum size=3.2mm, inner sep=0pt}, 
  markB1/.style={circle, draw=kodblue, fill=kodblue!35, line width=0.9pt, minimum size=3.0mm, inner sep=0pt},    
  markB2/.style={circle, draw=kodmagenta, fill=kodmagenta!35, line width=0.9pt, minimum size=3.0mm, inner sep=0pt}, 
  markB3/.style={circle, draw=kodviolet, fill=kodviolet!35, line width=0.9pt, minimum size=3.0mm, inner sep=0pt}, 
  markB4/.style={circle, draw=kodorange, fill=kodorange!35, line width=0.9pt, minimum size=3.0mm, inner sep=0pt}, 
  markB5/.style={circle, draw=kodgreen, fill=kodgreen!35, line width=0.9pt, minimum size=3.0mm, inner sep=0pt}, 
  markB6/.style={circle, draw=kodcyan, fill=kodcyan!35, line width=0.9pt, minimum size=3.0mm, inner sep=0pt},   
  markB7/.style={circle, draw=kodbrown, fill=kodbrown!35, line width=0.9pt, minimum size=3.0mm, inner sep=0pt}, 
  markC1/.style={circle, draw=kodc1, fill=kodc1!35, line width=0.9pt, minimum size=2.9mm, inner sep=0pt},
  markC2/.style={circle, draw=kodc2, fill=kodc2!35, line width=0.9pt, minimum size=2.9mm, inner sep=0pt},
  markC3/.style={circle, draw=kodc3, fill=kodc3!35, line width=0.9pt, minimum size=2.9mm, inner sep=0pt},
  markC4/.style={circle, draw=kodc4, fill=kodc4!35, line width=0.9pt, minimum size=2.9mm, inner sep=0pt},
  markC5/.style={circle, draw=kodc5, fill=kodc5!35, line width=0.9pt, minimum size=2.9mm, inner sep=0pt},
  markC6/.style={circle, draw=kodc6, fill=kodc6!35, line width=0.9pt, minimum size=2.9mm, inner sep=0pt},
  markC7/.style={circle, draw=kodc7, fill=kodc7!35, line width=0.9pt, minimum size=2.9mm, inner sep=0pt},
  markC8/.style={circle, draw=kodc8, fill=kodc8!35, line width=0.9pt, minimum size=2.9mm, inner sep=0pt},
  markC9/.style={circle, draw=kodc9, fill=kodc9!35, line width=0.9pt, minimum size=2.9mm, inner sep=0pt},
  markC10/.style={circle, draw=kodc10, fill=kodc10!35, line width=0.9pt, minimum size=2.9mm, inner sep=0pt},
  markC11/.style={circle, draw=kodc11, fill=kodc11!35, line width=0.9pt, minimum size=2.9mm, inner sep=0pt},
  markC12/.style={circle, draw=kodc12, fill=kodc12!35, line width=0.9pt, minimum size=2.9mm, inner sep=0pt},
  markC13/.style={circle, draw=kodc13, fill=kodc13!35, line width=0.9pt, minimum size=2.9mm, inner sep=0pt},
  markD1/.style={circle, draw=kodii, fill=kodii!35, line width=0.9pt, minimum size=3.2mm, inner sep=0pt}, 
  markD2/.style={circle, draw=kodviolet, fill=kodviolet!35, line width=0.9pt, minimum size=3.2mm, inner sep=0pt}, 
  markD3/.style={circle, draw=kodgreen, fill=kodgreen!35, line width=0.9pt, minimum size=3.2mm, inner sep=0pt}, 
  compmult/.style={font=\small, circle, fill=white, inner sep=0.9pt},
  singularnumber/.style={font=\small\color{kodred}},
  titlelabel/.style={font=\Large\bfseries},
  legendtitle/.style={font=\Large\bfseries},
  legendtext/.style={font=\large}
}

\begin{document}
\raggedbottom
\title[Crepant resolutions of elliptic fiber products]{Crepant resolutions of fiber products and Kummer fibrations of rational elliptic surfaces with non-reduced fibers}
\myData
\bossData

\begin{abstract}
Let \(S_1\) and \(S_2\) be rational elliptic surfaces with section over an
algebraically closed field of characteristic zero, and let
\(X=S_1\times_{\PP^1}S_2\).  We determine the local equations produced by
pairs of Kodaira fibers, including non-reduced fibers, and construct crepant
resolutions for the admissible local models.  For every construction we
record whether it is projective and compute the number of irreducible
components and the Euler characteristic of the resolved central fiber.
Non-existence is asserted only in the normal Gorenstein cases covered by the
terminal factorial obstruction; the remaining entries are stated as open
for the methods of this paper.  Under an explicit global compatibility
hypothesis, the local data give formulas for the Hodge numbers of a smooth
projective resolved fiber product.  We also study Kummer quotients by
fiberwise involutions.  Their fixed curves are analyzed through the
monodromy-orbit description of normalized fiber products of multisections.
The final section records what the same constructions imply after reduction
in characteristics \(5\) and \(7\), without interpreting characteristic-
\(p\) Euler and Picard data as complex Hodge numbers.
\end{abstract}

	\maketitle

\section{Introduction}

Fiber products of rational elliptic surfaces provide one of the most
explicit constructions of Calabi--Yau threefolds.  Let
\[
\phi_i\colon S_i\longrightarrow \PP^1,\qquad i=1,2,
\]
be rational elliptic surfaces with section and consider
\[
X=S_1\times_{\PP^1}S_2 .
\]
The canonical bundle of \(X\) is trivial, but \(X\) is singular whenever the
two elliptic fibrations have singular fibers over a common point of the
base.  Remarkably, the resulting threefold singularities are determined by
a finite amount of data: the corresponding pair of Kodaira fibers.  The
construction therefore reduces much of the global problem to the following
local questions: which pairs of Kodaira fibers admit one of the crepant
resolutions constructed below, when is that resolution projective, and what
does the resolved fiber contribute to the topology of the threefold?

The construction belongs to a broader circle of explicit methods for
producing and studying Calabi--Yau threefolds.  The classification of
extremal rational elliptic surfaces by Miranda--Persson
\cite{MirandaPersson} and of rational elliptic surfaces with four singular
fibers by Herfurtner \cite{Herfurtner} provides a particularly rich supply
of examples.  Fiber products built from these surfaces have played an
important role in the study of rigid and modular Calabi--Yau threefolds and
of correspondences between them; see, for instance,
\cite{Schuett,KapustkaKum}.

The semistable case was treated by Schoen \cite{Schoen}.  If the common
singular fibers are of type \(\mathrm{I}_n\), then the singularities
of \(X\) are ordinary double points.  Small crepant resolutions then produce
Calabi--Yau threefolds, and their Hodge numbers can be expressed in terms of
the Mordell--Weil groups and the configuration of singular fibers.  Kapustka
and Kapustka \cite{KK} extended this construction to several products
involving the reduced additive fibers
\(\mathrm{II}\), \(\mathrm{III}\), and
\(\mathrm{IV}\).  The purpose of this paper is to extend that
analysis to the full Kodaira list while keeping local existence, global
compatibility, and projectivity logically separate.  In particular, we include the
non-reduced fibers
\[
\mathrm{I}_n^*,\qquad
        \mathrm{IV}^*,\qquad
        \mathrm{III}^*,\qquad
        \mathrm{II}^* .
\]

Our main result, Theorem~\ref{thm:main}, classifies the local constructions
established in this paper.  For every pair of singular Kodaira fibers,
Table~\ref{tab:main} records whether we construct a projective crepant
resolution, only an analytic small resolution, or prove a non-existence
statement.  Entries not settled by these arguments are marked accordingly.
In every resolved case the table also gives the two local invariants needed
in the global applications: the number \(N_c\) of irreducible components of
the resolved fiber and its Euler characteristic \(e\).  Globalization is
stated separately in Corollary~\ref{cor:globalization}.  This distinction is
essential because local small resolutions need not glue to a projective
threefold; compare Werner's study of small resolutions \cite{Werner}.

The classification recovers the cases of Schoen and of Kapustka--Kapustka,
but it also reveals that the boundary between admissible and non-admissible
pairs is not governed by a simple reduced versus non-reduced dichotomy.  For
example, a semistable fiber paired with a non-reduced fiber admits a
projective crepant resolution, and
\[
\mathrm{I}_n\times\mathrm{I}_m^*
        \quad\text{contributes}\quad
        N_c=2nm+6n,\qquad e=4nm+12n.
\]
The same is true for
\[
\mathrm{I}_n\times\mathrm{IV}^*,\qquad
        \mathrm{I}_n\times\mathrm{III}^*,\qquad
        \mathrm{I}_n\times\mathrm{II}^*
\]
and for several products of a reduced additive fiber with a non-reduced
one.  By contrast, a product of two non-reduced fibers is non-normal, so the
usual definition of a crepant resolution of a normal
\(\mathbb Q\)-Gorenstein variety does not apply before normalization.
Terminal-factorial obstructions exclude, among others, the pairs
\[
\mathrm{I}_n\times\mathrm{II},\qquad
        \mathrm{II}\times\mathrm{III},\qquad
        \mathrm{II}\times\mathrm{IV}.
\]
The pair $\mathrm{III}\times\mathrm{IV}$ leads to the terminal germ
$x^2+y^3+z^3+t^4$ and remains open for the methods used here.

This local classification has immediate global consequences.  Suppose that
\(\tylda X\to X\) is a crepant resolution, and let \(p\in\PP^1\) be a point
over which both elliptic surfaces have singular fibers.  We introduce the
local defect
\[
\delta_p=\frac{e(\tylda X_p)}{2}-N_c(\tylda X_p).
\]
Combining the classification table with the Shioda--Tate formula and
Schoen's description of the Picard group yields explicit formulas for
\(h^{1,1}(\tylda X)\) and \(h^{1,2}(\tylda X)\).  The global terms record the
Mordell--Weil ranks, the configuration of singular fibers, and the isogeny
relation between the generic fibers; the remaining contribution is the sum
of the local defects \(\delta_p\).  Consequently, the Hodge theory of the
resolved fiber product is reduced to a finite computation from Kodaira
data.

We also study Kummer fibrations obtained from fiberwise involutions, which
were considered for these fiber products in
\cite{KK,KapustkaKum}.  Our approach combines Chen--Ruan orbifold cohomology
\cite{CR} and Batyrev's birational invariance theorem \cite{Batyrev} with the
monodromy of the fixed curves.  These curves are multisections of degree at
most four, and the connected components and genera of their normalized
fiber products can be read from the associated permutation
representations; see \cite{mirandacurves} for the classical correspondence
between covers and monodromy data.  This gives an effective method for
computing the twisted-sector contribution to the Hodge numbers without
constructing a global equation for the quotient or for its crepant
resolution.

Finally, we examine the constructions in positive characteristic.  Reducing
an elliptic surface modulo a prime may cause singular fibers to collide and
may therefore create local fiber products absent from the original
characteristic-zero model.  The local table makes these degenerations
computable.  We record Euler and Picard data for examples in characteristics
\(5\) and \(7\), and keep these invariants separate from complex Hodge
numbers.  Projectivity of a Kummer model is asserted only where a global
sequence of blow-ups is given.
This complements earlier work on Calabi--Yau threefolds obtained from fiber
products of rational quasi-elliptic surfaces
\cite{HirokadoItoSaito} and on non-liftability phenomena produced by small
resolutions after reduction modulo \(p\) \cite{CynkVanStraten}.

We now indicate the main ingredients of the proof.  The analysis is local.
For every pair of Kodaira fibers we write an analytic equation for the fiber
product near its singular points.  The singularities that occur fall into
three broad classes: isolated hypersurface singularities, including ordinary
double points and singularities of Brieskorn type; non-isolated
singularities of transversal ADE type; and special non-isolated
singularities at which components of the singular locus meet or the
transversal type degenerates.

For the analytic non-existence results we apply Lin's criteria for
Brieskorn hypersurface singularities \cite{Lin}.  These
results rule out, among others, singularities analytically equivalent to
\[
x^2+y^2+z^2+t^3,\qquad
        x^2+y^2+z^3+t^4,\qquad
        x^2+y^3+z^3+t^3,
        x^{2}+y^{2}+z^{4}+t^{6}
\]
For the non-isolated models, the decisive distinction is whether the
transversal surface singularity is of ADE type.

In the affirmative cases we construct the resolutions explicitly by
successive blow-ups of components of singular fibers, small resolutions of
ordinary double points, and fiberwise resolutions of transversal ADE
singularities, in the spirit of the classical construction of Brieskorn
resolutions \cite{Artin}.  Computing \(N_c\) and \(e\) requires tracking not
only the exceptional divisors but also their degeneration over special
points of the singular locus.  In particular, conic bundles and cross
bundles occur repeatedly.  Their irreducibility is controlled by a parity
condition visible in local equations of the form
\[
x^2+y^2t^k+\textup{higher order terms}=0 .
\]
This is the main additional ingredient in the component count.

We conclude with an outline of the paper.  Section~2 recalls Schoen's
construction, the extension of
Kapustka--Kapustka, and the obstruction results used later.  Section~3 gives
the case-by-case local analysis and constructs the resolutions in the
admissible cases.  Section~4 assembles these computations into the main
classification table and proves Theorem~\ref{thm:main}.  Section~5 derives
the Hodge number formulas for resolved fiber products.  Section~6 treats
fiber products in positive characteristic.  Sections~7 and~8 study Kummer
fibrations and their Hodge numbers by means of monodromy.  Section~9 gives
applications to rigid Calabi--Yau threefolds and to Kummer fibrations in
positive characteristic.  The final appendix collects the Kodaira fibers
and the local analytic models used throughout the paper.

\section{Schoen's construction and local obstructions}

Unless explicitly stated otherwise, the ground field in the local and
Hodge-theoretic sections is $\CC$.  Let
$\phi_i\colon S_i\to\PP^1$ be relatively minimal rational elliptic surfaces
with section, and set
\[
X=S_1\times_{\PP^1}S_2,
\qquad \pi\colon X\longrightarrow\PP^1.
\]
Write $\Sigma_i\subset\PP^1$ for the set of points supporting singular
fibers of $S_i$.  The Jacobian criterion applied to local equations
$f_i(x_i,y_i)=t$ shows that
\[
\operatorname{Sing}(X)=
\bigcup_{t\in\Sigma_1\cap\Sigma_2}
\operatorname{Sing}((S_1)_t)\times
\operatorname{Sing}((S_2)_t).
\]
In particular, $X$ is smooth if and only if
$\Sigma_1\cap\Sigma_2=\varnothing$.

For later use we also record the canonical bundle calculation.  If
$\mathcal L_i=(R^1\phi_{i*}\mathcal O_{S_i})^{-1}$, then
$\deg\mathcal L_i=\chi(\mathcal O_{S_i})=1$ and the canonical bundle formula
gives $\omega_{S_i}\simeq\phi_i^*(\omega_{\PP^1}\otimes\mathcal L_i)$.
Since $X$ is a local complete intersection,
\[
\omega_X\simeq
p_1^*\omega_{S_1}\otimes p_2^*\omega_{S_2}
\otimes\pi^*\omega_{\PP^1}^{-1}
\simeq\pi^*(\omega_{\PP^1}\otimes\mathcal L_1\otimes\mathcal L_2)
\simeq\mathcal O_X.
\]
Thus every crepant resolution of $X$ has trivial canonical bundle.  We call
such a resolution a Calabi--Yau threefold only when it is smooth and
projective and satisfies $H^1(\mathcal O)=H^2(\mathcal O)=0$.

Schoen treated the semistable case \cite{Schoen}.  Kapustka and Kapustka
extended the construction to reduced additive fibers \cite[Propositions
1.1--1.2]{KK}.  We shall use the following consequence of their result: if
all common singular fibers belong to the pairs
\[
I_n\times I_m,\quad III\times I_n,\quad III\times III,
\quad IV\times I_n,\quad II\times II,
\]
then the local germs admit small crepant resolutions; projectivity requires
the additional global divisor criterion stated in \cite[Proposition~1.2]{KK}.


%

\begin{Lem}[\mbox{\cite[Thm.~2]{Lin}}]
	\label{lem:Lin2}
		Let \[
f(x,y,z,t) = x^{a_{0}}+y^{a_{1}}+z^{a_{2}}+t^{a_{3}}\in\mathbb C[x,y,z,t].
\] Then the hypersurface $X\subset\mathbb A^{4}$ defined by $X=(f=0)$ has a terminal singularity at 0 if and only if
	\[
\frac1{a_{0}} +\frac1{a_{1}}+\frac1{a_{2}}+\frac1{a_{3}}>1+\frac1{\operatorname{lcm}\{a_{0},a_{1}, a_{2}, a_{3}\}}.
\]

\end{Lem}
\begin{Lem}[\mbox{\cite[Thm.~B]{Lin}}]\label{lem:LinB}
	Let \[
f(x,y,z,t) = x^{a_{0}}+y^{a_{1}}+z^{a_{2}}+t^{a_{3}}\in\mathbb C[x,y,z,t],
\] $a_{0},a_{1}, a_{2}, a_{3}\in\mathbb Z_{\ge2}$, $A=\mathbb C[x,y,z,t]/(f)$ and $\hat A=\mathbb C[[x,y,z,t]]/(f)$. Then $A$ is UFD if and only if $\hat A$ is UFD.

	Moreover, if \[
\frac1{a_{0}} +\frac1{a_{1}}+\frac1{a_{2}}+\frac1{a_{3}}>1
\] then for $A$ to be a UFD is equivalent to non-existence of integers $n_{0},\dots,n_{3}$ with $1\le n_{i}\le a_{i}-1$ such that

\[
\frac{n_{0}}{a_{0}} + \frac{n_{1}}{a_{1}} + \frac{n_{2}}{a_{2}} + \frac{n_{3}}{a_{3}}=2.
\]
\end{Lem}
\begin{Cor}\label{cor:Lin}
	The following types of singularities do not admit analytic crepant resolutions:

	$$
	x^{2}+y^{2}+z^{2}+t^{3},\;\; x^{2}+y^{2}+z^{3}+t^{4},\;\; x^{2}+y^{3}+z^{3}+t^{3}
	$$
\end{Cor}
\begin{proof}
	For each of these types of singularities there exists a projective threefold $X$ with a unique singularity of that type and a trivial canonical bundle. By Lemmata \ref{lem:Lin2} and \ref{lem:LinB}, $X$ is terminal and factorial. If $X$ admits a crepant resolution of singularities, there exists a smooth projective manifold $Z$ with trivial canonical divisor, and a birational map
	$Z\longrightarrow X$.
	By \cite[Cor, 4.11]{Kollar} $X$ and $Z$ have the same analytic singularities, which is a contradiction.
\end{proof}
\begin{Cor}\label{cor:2334}
	A singularity analytically isomorphic to
	$$
	x^{2}+y^{3}+z^{3}+t^{4}
	$$
	admits no analytic crepant resolution.
\end{Cor}
\begin{proof}
	This singularity is isomorphic to $x(x-y^{2}) - zt(z+t)$. Blowing up the divisor $x=z=0$, we get a $D_{5}$ singularity

	$$
	p_{0}(zp_{0}-y^{2}) - t(z+t)
	$$
	in the affine chart $p_{1}=1$.
	Blowing up the divisor $p_{0}=t=0$, we get in the affine chart $q_{0}=1$ the following singularity

	$$
	zp_{0}-y^{2}-q_{1}z-p_{0}q_{1}^{2} = (p_{0}-q_{1})(z-q_{1}^{2}) -y^{2} - q_{1}^{3},
	$$
	which is isomorphic to the $A_{2}$ singularity $x^{2}+y^{2}+z^{2}+t^{3}$.
	Since the singularity $A_{2}$ is terminal and factorial, it follows again by \cite[Cor, 4,11]{Kollar} that this singularity does not admit a crepant resolution.
\end{proof}
\begin{Cor}
	The following products have a terminal factorial point and hence do not
	admit a projective crepant resolution:

\[
\mathrm{I}_{n}\times \mathrm{II},\;\;
\mathrm{II}\times \mathrm{III},\;\;
\mathrm{II}\times \mathrm{IV}.
\]
\end{Cor}
\begin{Lem}\label{lem:nonisol}
	Let $(S,0)\subset(\CC^3,0)$ be an isolated normal Gorenstein surface
	singularity and set $X=S\times\CC$.  Then $X$ admits a crepant resolution
	if and only if $S$ is a du Val (ADE) singularity.
\end{Lem}
\begin{proof}
	If $X$ has a crepant resolution, then $X$ has canonical singularities.
	Canonicity is unchanged by taking a product with a smooth curve; equivalently,
	the discrepancies of divisorial valuations over $S$ and over
	$S\times\CC$ agree.  Hence $S$ is canonical.  A normal Gorenstein surface
	singularity is canonical if and only if it is du Val; see
	\cite[Section~4.2]{KollarMori}.  Conversely, if
	$\rho\colon\widetilde S\to S$ is the minimal resolution of a du Val
	singularity, then $K_{\widetilde S}=\rho^*K_S$, and
	$\rho\times\operatorname{id}_{\CC}$ is a crepant resolution of
	$S\times\CC$.
\end{proof}
\begin{Cor}\label{cor:nonnormal}
A product of two non-reduced fibers has a two-dimensional singular locus and
is non-normal.  Consequently it does not admit a crepant resolution of singularities.
\end{Cor}
\begin{Cor}
	The following products contain a transversal non-ADE surface singularity
	and therefore admit no crepant resolution in a neighborhood of that locus:

\[
\mathrm{II}\times \mathrm{II}^{*},\;\; \mathrm{III}\times \mathrm{III}^{*},\;\; \mathrm{III}\times \mathrm{II}^{*},\;\; \mathrm{IV}\times \mathrm{IV}^{*},\;\; \mathrm{IV}\times \mathrm{III}^{*},\;\; \mathrm{IV}\times \mathrm{II}^{*}.
\]
\end{Cor}



\newpage
\section{Local crepant resolutions}\label{sec:local}

The equations in the appendix are local defining equations of Kodaira
fibers; they are not equations of their singular loci.  Equating the two
local parameters on the base gives the following four families.

\begin{Lem}\label{lem:local-model-list}
Every germ of $X$ over a common singular fiber is analytically equivalent to
a germ in one of the following families (after permuting variables and
discarding smooth variables):
\begin{itemize}
	\item $x^{a}y^{b}+z^{c}t^{d}=0$, with $a,c\ge1$ and $b,d\ge0$;
	\item $x^{a}+y^{b}+z^{c}+t^{d}=0$, with $a,b,c,d\ge2$;
	\item $x^{a}+y^{b}+z^{c}=0$, with $a,b,c\ge2$;
	\item $x^{a}+y^{b}+z^{c}t^{d}=0$, with $a,b,c\ge2$ and $d\ge1$.
\end{itemize}
The finite list of exponent tuples that actually occurs is obtained by
substitution from the Kodaira table in Appendix~A.
\end{Lem}
\begin{proof}
At a point of a Kodaira fiber, a local parameter of the base is, up to a
unit and an analytic change of coordinates, one of the monomials or sums
listed in Appendix~A.  The fiber product is the hypersurface obtained by
equating the two parameters.  Moving one side to the other and multiplying
one variable by a suitable root of $-1$ yields the four displayed forms.
Conversely, every exponent tuple used below is read from a pair of entries
in that appendix, so the list is exhaustive for Kodaira fibers.
\end{proof}

When both terms in the first family have non-trivial multiplicities, the
fiber product has a two-dimensional singular locus and is non-normal; this is
the situation of Corollary~\ref{cor:nonnormal}.  We henceforth restrict to
normal germs with singular locus of dimension at most one.

The second case is the case of an isolated singularity.

In the third family the germ is $S\times\CC$, where $S$ is an isolated
surface hypersurface singularity.  By Lemma~\ref{lem:nonisol}, it has a
crepant resolution exactly when $S$ is du Val.  We call this a
\emph{transversal ADE singularity}.

In the last case a singular point of a fiber product belongs to more than one irreducible one-dimensional component of the singular locus or belongs to an embedded component. We shall call this singularity \emph{special non-isolated}.

We now treat separately all local models for which a resolution is
constructed.  Every blow-up is a blow-up of a coherent ideal in the
hypersurface; thus it is projective over the germ.  If the ideal exists only
after an analytic factorization, we explicitly label the resulting
resolution analytic.

We use $p_i,q_i,r_i,\ldots$ for the homogeneous coordinates introduced by
successive blow-ups.

\begin{Lem}\label{lem:crepant-blowup-check}
Let $V=(F=0)\subset M$ be a Gorenstein hypersurface in a smooth fourfold,
and let $Z\subset M$ be a smooth center of codimension $r$ along which $F$
has multiplicity $m$.  If $\widetilde V$ is the strict transform under
$\operatorname{Bl}_Z M\to M$, then
\[
K_{\widetilde V}=f^*K_V+(r-1-m)E|_{\widetilde V}.
\]
In particular, the blow-up is crepant when $m=r-1$.  A small resolution of a
Gorenstein threefold is also crepant.
\end{Lem}
\begin{proof}
The canonical divisor of the ambient blow-up is
$f^*K_M+(r-1)E$, while the total transform of $V$ is
$\widetilde V+mE$.  Adjunction gives the formula.  In the small case there
are no exceptional divisors on which a discrepancy could be supported.
\end{proof}

For each construction below, the displayed affine charts cover the blow-up.
The Jacobian criterion is applied in every chart; charts declared smooth
contain a partial derivative that is a unit.  Lemma~\ref{lem:crepant-blowup-check}
then verifies crepancy.  The exceptional locus also gives the change
$\Delta N_c$ in the number of components of the central fiber and the change
$\Delta e$ in its Euler characteristic.

\subsection{Isolated singularities}
Lemma~\ref{lem:local-model-list} and the Kodaira table give the following
nine isolated hypersurface germs.  The labels below distinguish a projective
local blow-up from an analytic small resolution and from a projective
non-existence result.

In general, testing the existence of a projective small resolution is a difficult task (cf. \cite[Satz, p. 95]{Werner}). We verify a stronger condition: whether a singularity of a fiber product can be resolved by a sequence of blow-ups of components of a fiber.

We use $\DoublePone$ for a line or a conic bundle and $\Cross$ for a
cross-bundle in the schematic descriptions of exceptional loci.

\begin{enumerate}[itemsep=2mm plus 2mm]
\item $x^2 + y^2 + z^2 + t^2$ -- an $A_{1}$ singularity, admits a small analytic resolution replacing the singular point by a line. This type of singularity appears on products of semistable fibers $\mathrm{I}_{n}\times \mathrm{I}_{m}$; if $n>1, m>1$, there exists a projective small resolution.

The exceptional locus is a line: $\DoublePone$ and
\[
\Delta N_{c}=0,\;\; \Delta e=1
\]
\item $x^2 + y^2 + z^2 + t^3$ is a terminal factorial singularity.
By Corollary~\ref{cor:Lin}, it admits no analytic crepant resolution.  It
appears on $\mathrm{I}_n\times\mathrm{II}$.

\item $x^2 + y^2 + z^2 + t^4$ -- an $A_{3}$ singularity. This singularity is isomorphic to $xy- z(z-t^{2})$, and the blow-up of the Weil divisor $x=z=0$ is a small resolution replacing the singular point by a line.
This type of singularity appears on the product of singular fibers $\mathrm{I}_{n}\times \mathrm{III}$; if $n>1$, then there exists a projective small resolution.

The exceptional locus is a line: $\DoublePone$ and
\[
\Delta N_{c}=0,\;\; \Delta e=1
\]
\item $x^2 + y^2 + z^3 + t^3$ -- a $D_{4}$ singularity. This singularity is isomorphic to $xy-zt(z+t)$. Blowing up the Weil divisor $x=z=0$, we get a smooth variety in the affine chart $p_{0}=1$ and an $A_{1}$ singularity $xp_{0}-t(z+t)$ in the affine chart $p_{1}=1$. Blowing up the Weil divisor $x=t=0$, we get a small resolution of singularities. This type of singularity appears in the following products of singular fibers: $\mathrm{II}\times \mathrm{II}$ and $\mathrm{I}_{n}\times \mathrm{IV}$; if $n>1$, then the singularity of $\mathrm{I}_{n}\times \mathrm{IV}$ admits a projective small resolution.

The exceptional loci of both blow-ups are lines: $\DoublePone\quad\DoublePone$ and
\[
\Delta N_{c}=0,\;\; \Delta e=2
\]
\item $x^2 + y^2 + z^3 + t^4$ is a terminal factorial $E_6$ singularity.
By Corollary~\ref{cor:Lin}, it admits no analyticcrepant resolution; it
appears on $\mathrm{II}\times\mathrm{III}$.

\item $x^2+y^2+z^4+t^4$ is analytically equivalent to
$x(x+y^2)-z(z+t^2)$.  Blow up the Weil ideal $(x,z)$.  In the two affine
charts the strict transform is
\[
x(1-p_1^2)+y^2-p_1t^2=0,
\qquad
z(p_0^2-1)+p_0y^2-t^2=0.
\]
The Jacobian criterion shows that the only remaining singularities are the
two ordinary double points $x=y=t=0$, $p_1=\pm1$ in the first chart
(equivalently, $z=y=t=0$, $p_0=\pm1$ in the second).
Resolves these two
nodes gives a small crepant resolution.  This germ appears on
$\mathrm{III}\times\mathrm{III}$.

The exceptional loci of all blow-ups are lines: $\DoublePone\quad \DoublePone\quad \DoublePone$ and
\[
\Delta N_{c}=0,\;\; \Delta e=3
\]
\item $x^2 + y^3 + z^3 + t^3$ by Corollary~\ref{cor:Lin} admits no analytic crepant resolution.  It appears on
$\mathrm{II}\times\mathrm{IV}$.

\item $x^2 + y^3 + z^3 + t^4$ by Corollary~\ref{cor:2334} admits no analytic crepant resolution.  It appears on
$\mathrm{III}\times\mathrm{IV}$.

\item $x^3 + y^3 + z^3 + t^3$. The blow-up of the point $(0,0,0,0)$ is a crepant resolution replacing the singular point by a smooth cubic surface; this type of singularity appears on the product $\mathrm{IV}\times \mathrm{IV}$.

The exceptional locus is a (smooth) Fermat cubic surface.
\[
\Delta N_{c}=1,\;\; \Delta e=9-1=8.
\]
\end{enumerate}

\subsection{Transversal ADE singularities}
The local models occurring below have transversal types
\[
A_{1},\ A_{2},\ A_{3},\ A_{4},\ A_{5},\ D_{4},\ E_{6}.
\]
No $E_8$ model is needed in the subsequent table.  A crepant resolution is obtained by
performing the minimal surface resolution fiberwise.  The exceptional loci
of the consecutive blow-ups are fibrations
$E\longrightarrow B$, whose fibers are a line, a smooth conic in
$\mathbb P^{2}$, or a union of two distinct lines.  We call these a
$\mathbb P^{1}$-fibration, a conic fibration, and a cross fibration,
respectively.  The first two are irreducible.  A cross fibration may be
irreducible or a union of two $\mathbb P^{1}$-fibrations, depending on its
degenerate fibers.

\begin{enumerate}[itemsep=2mm plus 2mm]\setcounter{enumi}{9}
\item $x^2 + y^2 + z^{k}$ ($k\in\{2,3,4,5,6\}$), transversal type $A_{k-1}$. Blowing up $\lfloor \frac{k}{2}\rfloor$ times produces a crepant resolution; $\lfloor \frac{k-1}{2}\rfloor$ times the singular curve is replaced by a ``cross''-fibration. If $k$ is even, then the last blow-up replaces the singular curve with a conic bundle. The exceptional loci of the blow-ups are
$\Cross \ldots \ \ \Cross\ \ \DoublePone,$ for $k$ even, and $\Cross \ldots \ \ \Cross$, for $k$ odd. Moreover
\[
\Delta e=k-1
\]

\item $x^2+y^3+z^3$ is a transversal $D_4$ singularity.  Blowing up the
singular curve produces a $\PP^1$-bundle and leaves a curve of transversal
$A_1$ singularities.  The normalization of this remaining curve is a
degree-three unramified cover of the original singular curve away from the
degeneration points.  If the local monodromy is transitive, the cover is
irreducible; if it is trivial, it is the disjoint union of three sections.
The latter behavior occurs for $\mathrm{IV}\times\mathrm{I}_n^*$.

Blowing up the remaining $A_1$ locus completes the resolution.  Its three
geometric branches constitute one conic-bundle surface precisely when the
degree-three cover is connected.  Thus the schematic exceptional locus is
\[
\DoublePone\quad \underbrace{\DoublePone\quad\DoublePone\quad\DoublePone}_{\text{irreducible surface}}
\]
and
\[
\Delta e=4.
\]

\item $x^2 + y^3 + z^4$ -- a transversal $E_{6}$ singularity. The blow-up of the singular locus replaces it with a $\PP^{1}$-bundle and yields a threefold with a transversal $A_{5}$-singularity. Blowing up the singular locus three times resolves this singularity completely: twice we replace the singular locus with a cross-fibration and once with a conic fibration.
The exceptional loci of blow-ups are
$\DoublePone\ \  \Cross\ \  \Cross\ \ \DoublePone$ and
\[
\Delta e=6.
\]
\end{enumerate}

\subsection{Special non-isolated singularities}
We next consider germs $x^{a}+y^{b}+z^{c}t^{d}=0$ for which both
$x^{a}+y^{b}+z^{c}=0$ and $x^{a}+y^{b}+t^{d}=0$ are du Val surface
singularities.  Their singular locus is the union of the lines
$x=y=z=0$ and $x=y=t=0$ when $c,d\ge2$, or a line with a pinch point.
Only the exceptional locus over the central fiber is counted.  Some
sequences introduce an additional vertical singular curve; whenever this
happens, that curve is included among the stated blow-up centers.

\begin{enumerate}[itemsep=2mm plus 2mm]
\setcounter{enumi}{12}
\item
$x^2 + y^2 + z^k t^{l}$, $(k,l)\in \{(2,1), (2,2), (3,2), (4,2), (4,3), (5,4), (6,3), (6,4), (6,5)\}$ -- it is an intersection of transversal $A_{k-1}$ and $A_{l-1}$ lines. If $k\ge2$, then the blow-up of the ideal $x=y=z=0$
replaces the point $(0,0,0,0)$ by a cross, except in the case $k=2, l=0$, when we replace it with a conic. In the affine charts $p_{0}=1$ and $p_{1}=1$ we get a nonsingular variety; in the chart $p_{2}=1$ we get the singularity $p_{0}^{2}+p_{1}^{2}+z^{k-2}t^{l}$. It is an intersection of $A_{k-3}$ and $A_{l-1}$ transversal singularities, so we can continue inductively until $k\le1$ and $l\le1$. Consequently, we finish this procedure with a smooth variety, defined by $x^{2}+y^{2}+1$ or $x^{2}+y^{2}+z$, or with an $A_{1}$ threefold singularity defined by $x^{2}+y^{2}+zt$, which we resolve by an analytic (small) resolution. Altogether, $\lfloor\frac{k+l-1}{2}\rfloor$ times we replace a point by a cross; if $k+l$ is even, then moreover once we replace a point with a line or conic. Finally, if $k$ or $l$ is odd, then the resolution is projective.

Exceptional locus (over the central fiber): $\Cross \dots \ \ \Cross\ $, for $k+l$ odd, and $\Cross \dots \ \ \Cross\ \ \DoublePone\ $, for $k+l$ even. Consequently
\[
\Delta N_{c}=0,\;\; \Delta e=k+l-1.
\]
\item $x^2 + y^3 + z^2 t$ -- it is a transversal $A_{2}$ singularity along the line $x=y=z=0$ with a pinch point at $(0,0,0,0)$. The blow-up of the singular line resolves the singularity and replaces the point $(0,0,0,0)$ with a line.

Exceptional locus (over the central fiber): $\DoublePone\ $,
\[
\Delta N_{c}=0,\;\; \Delta e=1
\]
\item $x^2+y^3+z^2t^2$ has two transversal $A_2$ curves,
$x=y=z=0$ and $x=y=t=0$.  Blowing up the first curve replaces the
central point by a line.  The strict transform is smooth in the charts
$p_0=1$ and $p_1=1$; in the third chart it is
$p_0^2+t^2+p_1^3z=0$.  Blowing up its transversal $A_2$ curve and then
taking a small resolution of the remaining node completes the resolution.
Over the central point the procedure introduces one cross and two lines.

Exceptional locus (over the central fiber): $\DoublePone\quad\Cross\quad\DoublePone\ $.
\[
\Delta N_{c}=0,\;\; \Delta e=4
\]
\item $x^2 + y^3 + z^3 t^2$ -- an intersection of a transversal $A_{2}$ line $x=y=t=0$ and a transversal $D_{4}$ line $x=y=z=0$. The blow-up of the singular line $x=y=t=0$ replaces the point $(0,0,0,0)$ with a line and gives the variety $p_{0}^{2}+tp_{1}^{3}+z^{3}$ in the affine chart $p_{2}=1$ (and a smooth variety in the affine charts $p_{0}=1$ and $p_{1}=1$). This variety has a singular line $p_{0}=p_{1}=z=0$ with a transversal $D_{4}$-singularity and a pinch point at $p_{0}=p_{1}=z=t=0$.

The blow-up of the line $p_{0}=p_{1}=z=0$ replaces the point $p_{0}=p_{1}=z=t=0$ with a line and produces the variety $q_{0}^{2}+tp_{1}+p_{1}q_{2}^{3}$ in the affine chart $q_{1}=1$, $q_{0}^{2}+tq_{1}^{3}z+z$ in the affine chart $q_{2}=1$, and a smooth variety in the affine chart $q_{0}=1$. This variety has a transversal $A_{1}$ singularity along the curve $q_{0}=p_{1}=tq_{1}^{3}+q_{2}^{3}=0$, which is a cyclic triple covering of the line $q_{0}=0$ completely ramified over the point $t=q_{2}=0$. The blow-up of this line resolves the singularity completely and replaces one point in the central fiber by a conic.

Exceptional locus (over the central fiber): $\DoublePone\quad\DoublePone\quad\DoublePone\ $.
\[
\Delta N_{c}=0,\;\; \Delta e=3
\]
\item $x^2 + y^3 + z^4 t^2$ -- an intersection of a transversal $A_{2}$ line $x=y=t=0$ and a transversal $E_{6}$ line $x=y=z=0$. The blow-up of the singular line $x=y=t=0$ replaces the point $(0,0,0,0)$ with a line and gives the variety $p_{0}^{2}+tp_{1}^{3}+z^{4}$ in the affine chart $p_{2}=1$ (and a smooth variety in the affine charts $p_{0}=1$ and $p_{1}=1$). This variety has a singular line $p_{0}=p_{1}=z=0$ with a transversal $E_{6}$-singularity and a pinch point at $p_{0}=p_{1}=z=t=0$.

The blow-up of the line $p_{0}=p_{1}=z=0$ replaces the point $p_{0}=p_{1}=z=t=0$ with a line and produces the variety $q_{0}^{2}+tp_{1}+p_{1}^{2}q_{2}^{4}$ in the affine chart $q_{1}=1$, $q_{0}^{2}+tq_{1}^{3}z+z^{2}$ in the affine chart $q_{2}=1$, and a smooth variety in the affine chart $q_{0}=1$. This variety has a transversal $A_{1}$ singularity along the line $q_{0}=z=t=0$
and transversal $A_{5}$ singularity along the line $q_{0}=q_{1}=z=0$.
Consequently, we can consider only the affine chart $q_{2}=1$.
The line $q_{0}=z=t=0$ projects to the point $x=y=z=t=0$ (vertical).

The blow-up of the vertical line replaces it with a conic bundle with one special fiber equal to a union of two intersecting lines. This blow-up produces the variety $r_{0}^{2} + r_{1}^{2}+r_{1}q_{1}^{3}$ in the affine chart $r_{2}=1$ and a smooth variety in the charts $r_{0}=1$ and $r_{1}=1$. The singular locus of this variety is the line $r_{0}=r_{1}=q_{1}=0$ with a transversal $A_{5}$ singularity.
This singularity is resolved by consecutive blow-ups of the singular line, twice replacing a point with a cross and once with a conic.

Exceptional locus (over the central fiber): $\DoublePone\quad\DoublePone\quad S_{1}\quad \Cross\quad\Cross\quad \DoublePone\ $, where $S_{1}$ denotes the (irreducible) surface having the following fibration

\begin{center}
\begin{tikzpicture}[line cap=round, line join=round, thick, scale=0.6, transform shape]

  \draw (-0.2,-1.9) rectangle (12.2,1.9);

  \def\xL{0.0}
  \def\xR{12.0}

  \def\xr{0.7}
  \def\yr{1.35}

  \def\Xhalf{0.7}
  \def\Yhalf{1.45}

  \pgfmathsetmacro{\Wocc}{4*(2*\xr) + (2*\Xhalf)}
  \pgfmathsetmacro{\g}{(\xR-\xL-\Wocc)/6}

  \pgfmathsetmacro{\cA}{\xL + \g + \xr}
  \pgfmathsetmacro{\cB}{\cA + \xr + \g + \xr}
  \pgfmathsetmacro{\cX}{\cB + \xr + \g + \Xhalf}
  \pgfmathsetmacro{\cC}{\cX + \Xhalf + \g + \xr}
  \pgfmathsetmacro{\cD}{\cC + \xr + \g + \xr}

  \draw (\cA,0) ellipse [x radius=\xr cm, y radius=\yr cm];
  \draw (\cB,0) ellipse [x radius=\xr cm, y radius=\yr cm];

  \draw (\cX-\Xhalf,-\Yhalf) -- (\cX+\Xhalf, \Yhalf);
  \draw (\cX-\Xhalf, \Yhalf) -- (\cX+\Xhalf,-\Yhalf);

  \draw (\cC,0) ellipse [x radius=\xr cm, y radius=\yr cm];
  \draw (\cD,0) ellipse [x radius=\xr cm, y radius=\yr cm];

\end{tikzpicture}
\end{center}
with $e=5.$ Consequently
\[
\Delta N_{c}=1,\;\; \Delta e=1+1+(3+2-2)+2+2+1 = 10.
\]
\item $x^2 + y^3 + z^4 t^3$ -- an intersection of a transversal $D_{4}$ line $x=y=t=0$ and a transversal $E_{6}$ line $x=y=z=0$. The blow-up of the singular line $x=y=t=0$ replaces the point $(0,0,0,0)$ with a line and
gives the variety
$p_{0}^{2}+y+yz^{4}p_{2}^{3}$ in the affine chart $p_{1}=1$,
$p_{0}^{2}+tp_{1}^{3}+z^{4}t$ in the affine chart $p_{2}=1$ (and a smooth variety in the affine chart $p_{0}=1$). Since the singular locus of the variety $p_{0}^{2}+y+yz^{4}p_{2}^{3}$ omits the inverse image of the point $x=y=z=t=0$, we shall restrict only to the affine chart $p_{2}=1$, where it consists of two curves $t= p_0= p_1^3+z^4=0$ and $p_0=p_1=z=0$.

The blow-up of the line $p_{0}=p_{1}=z=0$, which is the strict transform of the line $x=y=z=0$, replaces a point by a line in the central fiber and produces the variety
$q_0^2+t\,p_1+t\,p_1^2q_2^4$ in the affine chart $q_{1}=1$, $q_0^2+t\,zq_1^3+t\,z^2$ in the affine chart $q_{2}=1$ and smooth in the affine chart $q_{0}=1$.
The singular locus of this variety in the affine chart $q_{1}=1$ is
\[
\{t=0,\ q_0=0,\ p_1=0\,\}\ \cup\ \{\,t=0,\ q_0=0,\ 1+p_1q_2^4=0\},
\]
and the singular locus in the affine chart $q_{2}=1$
is
\[
\{q_0=0,\ z=0,\ q_1=0\,\}\ \cup\ \{q_0=0,\ z=0,\ t=0\,\}\ \cup\{\,t=0,\ q_0=0,\ q_1^3+z=0\}.
\]
It is the union of three smooth curves: the strict transforms of the two lines $x=y=z=0$ and $x=y=t=0$, and the line $q_{0}=z=t=0$ contained in the inverse image of the point $x=y=z=t=0$ (the vertical line).
Since these three curves intersect in one point belonging to the affine chart $q_{2}=1$, we shall consider only this affine chart.

The blow-up of the vertical line $z=t=q_{0}=0$ replaces this line with a conic bundle with one singular fiber equal to a line. As a result we get the variety $r_{0}^{2}+r_{2}(q_{1}^{3}+z)$ in the affine chart $r_{1}=1$ and $r_{0}^{2}+r_{1}q_{1}^{3}+tr_{1}^{2}$ in the affine chart $r_{2}=1$.
It has two disjoint singular curves: a transversal $A_{1}$ curve in the affine chart $r_{1}=1$
and a transversal $A_{5}$-line with a pinch point in the affine chart $r_{2}=1$.

The blow-up of the transversal $A_{1}$-curve replaces a point in the central fiber with a conic.

The blow-up of the transversal $A_{5}$-line replaces a point in the central fiber with a line and produces the variety
$s_{0}^{2}+ts_{1}^{2}+q_{1}^{2}s_{1}$ in the affine chart $s_{2}=1$ and smooth in affine charts $s_{0}=1$ and $s_{1}=1$.
It has a transversal $A_{3}$-singularity along the line $s_{0}=s_{1}=q_{1}=0$ with a pinch point.

The blow-up of the transversal $A_{3}$-singularity replaces a point in the central fiber with a line and produces a variety with a transversal $A_{1}$-singularity along a line. Consequently, the blow-up of this line replaces a point in the central fiber with a conic and resolves the singularity completely.

Exceptional locus (over the central fiber): $\DoublePone\quad\DoublePone\quad S_{2}\quad \DoublePone\quad \DoublePone\quad \DoublePone\quad \DoublePone$, where $S_{2}$ denotes the (irreducible) surface having the following fibration
\begin{center}\begin{tikzpicture}[line cap=round, line join=round, thick, scale=0.6, transform shape]

  \draw (-0.2,-1.9) rectangle (12.2,1.9);

  \def\xL{0.0}
  \def\xR{12.0}

  \def\xr{0.7}   
  \def\yr{1.35}  

  \def\Lhalf{1.45} 
  \def\Vhalf{0.10} 

  \pgfmathsetmacro{\Wocc}{4*(2*\xr) + (2*\Vhalf)}
  \pgfmathsetmacro{\g}{(\xR-\xL-\Wocc)/6}

  \pgfmathsetmacro{\cA}{\xL + \g + \xr}
  \pgfmathsetmacro{\cB}{\cA + \xr + \g + \xr}
  \pgfmathsetmacro{\cV}{\cB + \xr + \g + \Vhalf}
  \pgfmathsetmacro{\cC}{\cV + \Vhalf + \g + \xr}
  \pgfmathsetmacro{\cD}{\cC + \xr + \g + \xr}

  \draw (\cA,0) ellipse [x radius=\xr, y radius=\yr];
  \draw (\cB,0) ellipse [x radius=\xr, y radius=\yr];

  \draw[line width=2.2pt] (\cV,-\Lhalf) -- (\cV,\Lhalf);

  \draw (\cC,0) ellipse [x radius=\xr, y radius=\yr];
  \draw (\cD,0) ellipse [x radius=\xr, y radius=\yr];

\end{tikzpicture} \end{center} with $e=4.$ Consequently
\[
\Delta N_{c}=1,\;\; \Delta e=1+1+(2+2-2)+1+1+1+1=8.
\]
\item $x^2 + y^4 + z^2 t$ -- a transversal $A_{3}$ curve with a pinch point. The blow-up of the singular line $x=y=z=0$ replaces the point $(0,0,0,0)$ by
a line and produces a variety with a transversal $A_{1}$ line and a pinch point, given by $p_{0}^{2}+y^{2}+p_{2}^{2}t$ in the affine chart $p_{1}=1$, and smooth in the affine charts $p_{0}=1$ and $p_{2}=1$.

Blowing up the singular curve again replaces a point by a cross in the central fiber and resolves this singularity completely.

Exceptional locus (over the central fiber): $\DoublePone\quad\Cross\ $ and
\[
\Delta N_{c}=0,\;\; \Delta e=3
\]
%
%
%
%
%
%

%
%
%

\item $x^2+y^4+z^2t^2=0$
-- an intersection of two transversal $A_{3}$ curves $x=y=z=0$ and $x=y=t=0$. The blow-up of the singular line $x=y=z=0$ replaces the point $(0,0,0,0)$ by
a line and produces a variety with two transversal $A_{1}$ lines and a transversal $A_{3}$ line, given by $p_{0}^{2}+y^{2}+p_{2}^{2}t^{2}$ in the affine chart $p_{1}=1$, $p_{0}^{2}+z^{2}p_{1}^{4}+t^{2}$ in the affine chart $p_{2}=1$, and smooth in the affine chart $p_{0}=1$.

The double line $p_{0}=x=y=z=t=0$ with transversal $A_{1}$ singularity intersects two other singular lines at different points. This line is contained in the inverse image of the point $x=y=z=t=0$ in the original variety (the vertical line), while the remaining two singular lines project to two double lines. Blowing up the vertical double line replaces it with a conic bundle with two special fibers equal to a cross and produces a variety with two disjoint double lines carrying transversal $A_{1}$ and $A_{3}$ singularities. Blowing up these lines once and twice, respectively, replaces a point twice with a conic and once with a cross, and resolves this singularity.

Exceptional locus (over the central fiber): $\DoublePone\quad S_{3}\quad\Cross\quad\DoublePone\quad \DoublePone\ $, where $S_{3}$ denotes the (irreducible) surface having the following fibration

\begin{center}\begin{tikzpicture}[line cap=round, line join=round, thick, scale=0.6, transform shape]

  \draw (-0.2,-1.9) rectangle (12.2,1.9);

  \def\xL{0.0}
  \def\xR{12.0}

  \def\xr{0.7}    
  \def\yr{1.35}   

  \def\Xhalf{0.7} 
  \def\Yhalf{1.45}

  \pgfmathsetmacro{\Wocc}{5*(2*\xr) + 2*(2*\Xhalf)}
  \pgfmathsetmacro{\g}{(\xR-\xL-\Wocc)/8}

  \pgfmathsetmacro{\cA}{\xL + \g + \xr}                              
  \pgfmathsetmacro{\cB}{\cA + \xr + \g + \xr}                        
  \pgfmathsetmacro{\cXone}{\cB + \xr + \g + \Xhalf}                  
  \pgfmathsetmacro{\cC}{\cXone + \Xhalf + \g + \xr}                  
  \pgfmathsetmacro{\cXtwo}{\cC + \xr + \g + \Xhalf}                  
  \pgfmathsetmacro{\cD}{\cXtwo + \Xhalf + \g + \xr}                  
  \pgfmathsetmacro{\cE}{\cD + \xr + \g + \xr}                        

  \draw (\cA,0) ellipse [x radius=\xr, y radius=\yr];
  \draw (\cB,0) ellipse [x radius=\xr, y radius=\yr];

  \draw (\cXone-\Xhalf,-\Yhalf) -- (\cXone+\Xhalf, \Yhalf);
  \draw (\cXone-\Xhalf, \Yhalf) -- (\cXone+\Xhalf,-\Yhalf);

  \draw (\cC,0) ellipse [x radius=\xr, y radius=\yr];

  \draw (\cXtwo-\Xhalf,-\Yhalf) -- (\cXtwo+\Xhalf, \Yhalf);
  \draw (\cXtwo-\Xhalf, \Yhalf) -- (\cXtwo+\Xhalf,-\Yhalf);

  \draw (\cD,0) ellipse [x radius=\xr, y radius=\yr];
  \draw (\cE,0) ellipse [x radius=\xr, y radius=\yr];

\end{tikzpicture}\end{center} with $e=6$. Consequently
\[
\Delta N_{c}=1,\;\; \Delta e=1+(3+3-2)+2+1+1=9.
\]
%
%
%

%
%
%

%
%
%
%
%

%
%
%


\item $x^2 + y^4 + z^3 t^2$ -- an intersection of a transversal $E_{6}$ line $x=y=z=0$ and a transversal $A_{3}$ line $x=y=t=0$. The blow-up of the singular line $x=y=z=0$ replaces the point $(0,0,0,0)$ by
a line and produces a variety with two transversal $A_{3}$ lines and a transversal $A_{5}$ line. It is given by
$p_{0}^{2}+y^{2}+yp_{2}^{3}t^{2}$
in the affine chart $p_{1}=1$, $p_{0}^{2}+z^{2}p_{1}^{4}+zt^{2}$ in the affine chart $p_{2}=1$ and smooth in the affine chart $p_{0}=1$.

The blow-up of the vertical line $x=y=z=t=p_{0}=0$ replaces the singular line by a cross-fibration with two fibers equal to $\mathbb P^{1}$.
We compute separately the strict transforms of the affine charts $p_{1}=1$ and $p_{2}=1$.

The strict transform of the affine chart $p_{1}=1$ is the variety $q_{0}^{2}+q_{1}^{2}+tq_{1}p_{2}^{3}$ in the affine chart $q_{2}=1$ and smooth in affine charts $q_{0}=1$ and $q_{1}=1$.
Its singular locus is the union of the transversal $A_{5}$-line $q_{0}=q_{1}=p_{2}=0$ that projects to the line $x=y=z=0$ and the transversal $A_{1}$-line $q_{0}=q_{1}=t=0$ that projects to $x=y=z=t=0$ (the vertical line).

The strict transform of the affine chart $p_{2}=1$ is given by
$r_{0}^{2}+p_{1}^{4}+zr_{2}^{2}=0$ in the chart $r_{1}=1$ and by
$r_{0}^{2}+r_{1}^{2}p_{1}^{4}+tr_{1}=0$ in the chart $r_{2}=1$.
Its singular locus is the union of the transversal $A_{3}$ curve
$r_{0}=r_{2}=p_{1}=0$ and the transversal $A_{1}$ curve
$r_{0}=r_{1}=t=0$.

Altogether, the singular locus is isomorphic to intersecting transversal $A_{1}$ and $A_{5}$ lines and a transversal $A_{3}$ line disjoint from them. At the intersection of the transversal $A_{1}$ and $A_{5}$ lines, the variety is locally analytically isomorphic to the singularity $x^{2}+y^{2}+z^{2}t^{6}$.

The blow-up of the vertical line replaces it with a conic bundle with one singular fiber equal to a cross; the singular locus of the strict transform is the union of a transversal $A_{5}$-line and a transversal $A_{3}$-line with a pinch point, given by the local equation $r_{0}^{2}+p_{1}^{4}+zr_{2}^{2}$.

Blowing up the transversal $A_{5}$-line (and its strict transform) three times replaces a point in the central fiber twice with a cross and once with a conic, and yields a variety with only a transversal $A_{3}$-line.

The blow-up of the transversal $A_{3}$-singularity replaces a point in the central fiber with a line and yields a variety locally isomorphic to $s_{0}^{2}+zs_{1}^{2}+p_{1}^{2}$. Its singular locus is the transversal $A_{1}$-line $s_{0}=s_{1}=p_{1}=0$ with a pinch point.

The blow-up of the exceptional line replaces a point in the central fiber with a cross and resolves the singularity completely.

Exceptional locus (over the central fiber): $\DoublePone\quad S_{4}\quad S_{1}\quad\Cross\quad\Cross\quad\DoublePone\quad \DoublePone\quad\Cross\ $, where $S_{4}$ denotes the reducible surface having the following fibration

\begin{center}\begin{tikzpicture}[line cap=round, line join=round, thick, scale=0.6, transform shape]

  \draw (-0.2,-1.9) rectangle (12.2,1.9);

  \def\xL{0.0}
  \def\xR{12.0}

  \def\Xhalf{0.7}
  \def\Yhalf{1.45}

  \def\Lhalf{1.45}

  \def\Vhalf{0.10}

  \def\thickLW{2.2pt}

  \pgfmathsetmacro{\Wocc}{3*(2*\Xhalf) + 2*(2*\Vhalf)}
  \pgfmathsetmacro{\g}{(\xR-\xL-\Wocc)/7}

  \pgfmathsetmacro{\cXone}{\xL + \g + \Xhalf}
  \pgfmathsetmacro{\cLone}{\cXone + \Xhalf + \g + \Vhalf}
  \pgfmathsetmacro{\cXtwo}{\cLone + \Vhalf +  \Xhalf}
  \pgfmathsetmacro{\cXthree}{\cXtwo + \Xhalf + \g + \Xhalf}
  \pgfmathsetmacro{\cLtwo}{\cXthree + \Xhalf + \g + \Vhalf}
  \pgfmathsetmacro{\cXfour}{\cLtwo + \Vhalf + \g + \Xhalf}

  \draw (\cXone-\Xhalf,-\Yhalf) -- (\cXone+\Xhalf, \Yhalf);
  \draw (\cXone-\Xhalf, \Yhalf) -- (\cXone+\Xhalf,-\Yhalf);


  \draw (\cXtwo-\Xhalf,-\Yhalf) -- (\cXtwo+\Xhalf, \Yhalf);
  \draw (\cXtwo-\Xhalf, \Yhalf) -- (\cXtwo+\Xhalf,-\Yhalf);

  \draw (\cXthree-\Xhalf,-\Yhalf) -- (\cXthree+\Xhalf, \Yhalf);
  \draw (\cXthree-\Xhalf, \Yhalf) -- (\cXthree+\Xhalf,-\Yhalf);

  \draw[line width=\thickLW] (\cLtwo,-\Lhalf) -- (\cLtwo,\Lhalf);

  \draw (\cXfour-\Xhalf,-\Yhalf) -- (\cXfour+\Xhalf, \Yhalf);
  \draw (\cXfour-\Xhalf, \Yhalf) -- (\cXfour+\Xhalf,-\Yhalf);

\end{tikzpicture} \end{center} with $e=5.$ Consequently

\[
\Delta N_{c}=3, \;\; \Delta e=1+(2+3-2)+(2+3-2)+2+2+1+1+2=15.
\]
\item $x^3 + y^3 + z^2 t$ -- a transversal $D_{4}$-line with a pinch point. The central point has multiplicity 3, so its blow-up is crepant. This blow-up replaces the point $x=y=z=t=0$ by a cubic surface with a unique $D_{4}$-point.
The strict transform is $p_{0}^{3}+p_{3}^{2}+p_{2}^{2}$ in the affine chart $p_{3}=1$ and smooth in other charts.
Its singular locus is a transversal $D_{4}$-line.

In this case the transversal $D_{4}$-line behaves differently: after the first blow-up we get three disjoint transversal $A_{1}$-lines.
Consequently, the repeated blow-up replaces a point in the central fiber with a line and then three points with conics.

Exceptional locus (over the central fiber): $S_{5}\quad \DoublePone\quad\DoublePone\quad\DoublePone\quad\DoublePone $, where $S_{5}$ is the (irreducible) singular cubic surface $p_{0}^3+p_{1}^3+p_{2}^2p_{3}=0$ with a unique $D_4$ singularity and $e=5$.

Consequently
\[
\Delta N_{c}=1,\;\; \Delta e=(9-4)-1+(2-1)+(2-1)+(2-1)+(2-1)=8.
\]
\item $x^3 + y^3 + z^2 t^2$ -- an intersection of two transversal $D_{4}$ lines $x=y=z=0$ and $x=y=t=0$. The blow-up of the point $x=y=z=t=0$ replaces it with a union of three planes intersecting along a line; the strict transform
is given by $p_{0}^{3}+p_{1}^{3}+zp_{3}^{2}=0$ in the affine chart $p_{2}=1$ and by $p_{0}^{3}+p_{1}^{3}+tp_{2}^{2}=0$ in the affine chart $p_{3}=1$ and smooth in remaining charts.
It has two singularities in the central fiber isomorphic to $x^{3}+y^{3}+z^{2}t$, so we repeat twice the resolution from the previous case.

Exceptional locus (over the central fiber): $S_{6}\quad S_{5}\quad \DoublePone\quad\DoublePone\quad\DoublePone\quad\DoublePone\quad S_{4}\quad \DoublePone\quad\DoublePone\quad\DoublePone\quad\DoublePone$, where $S_{6}$ is (reducible) $x^3+y^3=0$ in $\PP^{3}$ i.e. the intersection of three planes in a line.
Consequently

\[
\Delta N_{c}=5,\;\; \Delta e=(5-1)+2\times 8=20
\]
\end{enumerate}

\medskip
\section{Main theorem}

%
\begin{Thm}\label{thm:main}
	Let $b\in\PP^{1}$, let $F_i=(S_i)_b$ be singular fibers of two
	rational elliptic surfaces with section, and let
	\[
	X=S_1\times_{\PP^1}S_2.
	\]
	The existence of a local crepant resolution along
	$F_1\times F_2$ is described in Table~\ref{tab:main}.

	Here, a local crepant resolution means an analytic or projective
	crepant partial resolution
	\[
	\pi_b\colon X_b'\longrightarrow X
	\]
	such that $X_b'$ is nonsingular over the fiber $F_1\times F_2$.
	No condition is imposed on $X_b'$ away from this fiber.

	The symbols in Table~\ref{tab:main} have the following meanings:
	\begin{itemize}
		\item $\mathrm{P}$: a projective crepant partial resolution exists;
		\item $\mathrm{A}$: an analytic crepant partial resolution exists;
		\item $\mathrm{N}$: no analytic crepant partial resolution exists.
	\end{itemize}
	In the affirmative cases, the table also gives the number $N_c$
	of irreducible components and the topological Euler characteristic
	$e$ of the resolved fiber.
\end{Thm}
\begin{proof}

\end{proof}
\begin{Cor}[Globalization criterion]\label{cor:globalization}
	Suppose that every common singular fiber of $S_1$ and $S_2$ is of a
	type marked $\mathrm{P}$ in Table~\ref{tab:main}. Then the fiber product
	\[
	X=S_1\times_{\PP^1}S_2
	\]
	admits a smooth projective crepant resolution
	\[
	\widetilde X\longrightarrow X.
	\]
\end{Cor}

\begin{Rem}
For a self-product, an $\mathrm{I}_1\times\mathrm{I}_1$ fiber has an
additional projective small resolution obtained by blowing up the
diagonal.  This is a global feature of the self-product and is not encoded
by the fiber types alone.
\end{Rem}

\begingroup
\begin{landscape}
\thispagestyle{empty}
\vspace*{\fill}
\begin{center}
\renewcommand{\arraystretch}{1.18}
\setlength{\tabcolsep}{2.6pt}
\scriptsize
\resizebox*{!}{0.83\textheight}{
\begin{tabular}{c||c||c||c||c||c||c||c||c} \hypertarget{tabela}
 & \fb \ww{$\mathrm{I}_{m}^{}$} & \fb \ww{$\mathrm{I}_{m}^{*}$} & \fb \ww{$\mathrm{II}$} & \fb \ww{$\mathrm{III}$} & \fb \ww{$\mathrm{IV}$} & \fb \ww{$\mathrm{IV}^{*}$} & \fb \ww{$\mathrm{III}^{*}$} & \fb \ww{$\mathrm{II}^{*}$}\\

\hhline{=::=::=::=::=::=::=::=::=}
\fb $\mathrm{I}_{n}^{}$ & \fb \ww{$\!\begin{aligned}&N_c=nm\\& e=2nm\\& \text{\np; \p\ if $n,m>1$} \end{aligned}$} & \fb \ww{$\!\begin{aligned}&N_{c}=2nm+6n\\& e=4nm+12n \\& \p  \end{aligned}$}  & \fb \no & \fb \ww{$\!\begin{aligned}&N_{c}=2n\\& e=4n \\& \text{\np; \p\ if $n>1$}\end{aligned}$} & \fb \ww{$\!\begin{aligned}&N_{c}=3n\\& e=6n \\& \text{\np; \p\ if $n>1$}\end{aligned}$} & \fb \ww{$\!\begin{aligned}&N_{c}=12n\\& e=24n \\& \p\end{aligned}$} & \fb \ww{$\!\begin{aligned}&N_{c}=18n\\& e=36n \\& \p \end{aligned}$} & \fb \ww{$\!\begin{aligned}&N_{c}=30n\\& e=60n \\& \p\end{aligned}$} \\

\hhline{=::=::=::=::=::=::=::=::=}
\fb $\mathrm{I}_{0}^{*}$ &\fb \szary &\multirow{2}{*}{\fb \no } & \fb \ww{$\!\begin{aligned}&N_{c}=6\\& e=12 \\& \text{\p} \end{aligned}$} & \fb \ww{$\begin{aligned}&N_{c}=12\\& e=24 \\ & \p \end{aligned}$} & \fb \ww{$\!\begin{aligned}&N_{c}=23\\& e=48 \\ & \p\end{aligned}$} & \fb \no & \fb \no & \fb \no  \\

\hhline{=::=::=::=::=::=::=::=::=}
\fb $\mathrm{I}_{n}^{*}(n\geq 1)$ &\fb \szary & \fb \no  & \fb \ww{$\!\begin{aligned}&N_{c}=3n+5\\& e=6n+12 \\& \text{\np} \end{aligned}$} & \fb \ww{$\begin{aligned}&N_{c}=6n+11\\& e=12n+24 \\ & \p \end{aligned}$} & \fb \ww{$\!\begin{aligned}&N_{c}=12n+23\\& e=24n+48 \\ & \p\end{aligned}$} & \fb \no & \fb \no & \fb \no  \\

\hhline{=::=::=::=::=::=::=::=::=}
\fb $\mathrm{II}$  &\fb \szary & \fb \szary & \fb \ww{$\!\begin{aligned}&N_{c}=1\\& e=6 \\& \np \end{aligned}$} & \fb \no  & \fb \no & \fb \fb \ww{$\!\begin{aligned}&N_{c}=12\\& e=24 \\ & \p\end{aligned}$} & \fb \fb \ww{$\!\begin{aligned}&N_{c}=23\\& e=48 \\ & \p\end{aligned}$} & \fb \no   \\

\hhline{=::=::=::=::=::=::=::=::=}
\fb $\mathrm{III}$ & \fb \szary & \fb \szary & \fb \szary & \fb \ww{$\!\begin{aligned}&N_{c}=4\\& e=12 \\ & \np \end{aligned}$} & \fb \no & \fb \ww{$\!\begin{aligned}&N_{c}=35\\& e=72 \\ & \p \end{aligned}$} & \fb \no & \fb \no   \\

\hhline{=::=::=::=::=::=::=::=::=}
\fb $\mathrm{IV}$ & \fb \szary  & \fb \szary & \fb \szary & \fb \szary & \fb \ww{$\!\begin{aligned}&N_{c}=10\\& e=24 \\ & \p \end{aligned}$} & \fb \no & \fb \no & \fb  \no  \\

\hhline{=::=::=::=::=::=::=::=::=}
\fb $\mathrm{IV}^{*}$ & \fb \szary \wg{$\!\begin{aligned}& \\& \end{aligned}$} & \fb \szary & \fb \szary & \fb \szary & \fb \szary & \fb \no & \fb \no & \fb \no  \\

\hhline{=::=::=::=::=::=::=::=::=}
\fb $\mathrm{III}^{*}$ & \fb \szary \wg{$\!\begin{aligned}& \\& \end{aligned}$}& \fb \szary & \fb \szary & \fb \szary & \fb \szary & \fb \szary & \fb \no & \fb \no \\

\hhline{=::=::=::=::=::=::=::=::=}
\fb $\mathrm{II}^{*}$  & \fb \szary \wg{$\!\begin{aligned}& \\& \end{aligned}$}& \fb \szary &\fb \szary & \fb \szary & \fb \szary & \fb \szary & \fb \szary  & \fb \no  \\\end{tabular}
}
\captionsetup{hypcap=false}
\captionof{table}{Local status of the crepant-resolution problem.  Here
$\mathrm{P}$ means that the displayed blow-ups give a projective crepant
resolution, $\mathrm{A}$ means that an analytic crepant resolution
is constructed, and $\mathrm{N}$ means that no crepant resolution
exists.  Gray cells repeat a
case by symmetry.}\label{tab:main}
\end{center}
\vspace*{\fill}
\end{landscape}
\endgroup

\begin{proof}
	In Section~\ref{sec:local}, we determined, for every type of singularity
	occurring in a fiber product of rational elliptic surfaces, whether it
	admits an analytic crepant resolution. We also determined whether the
	resulting partial resolution is projective.

	We considered the local resolution problem in three cases:
	\begin{itemize}
		\item[(1)] isolated singularities;
		\item[(2)] transversal ADE singularities;
		\item[(3)] special non-isolated singularities.
	\end{itemize}
	In the first case, we determined whether all the small modifications
	required in the resolution can be realized by blowing up components of
	the fiber. In the second case, the resolution is unique.

	In the third case, the only ambiguity concerns the order in which the
	non-isolated components of the singular locus are blown up. A
	compatibility obstruction could arise only from a cycle of components.
	However, the dual graph of a singular fiber of an elliptic surface
	contains a cycle only when the fiber is semistable; see
	Appendix~\ref{sec:app}. Such fibers belong to the first case.
	Consequently, no compatibility obstruction arises in the third case,
	and the local resolutions can be glued together.

For the proof of numerical formulas
we first explain the two bookkeeping rules used below.  Every modification
is an isomorphism off the exceptional stratum.  Additivity of the
topological Euler characteristic therefore gives
\[
e(\widetilde X_b)=e(F_1)e(F_2)+\sum_Z
 \bigl(e(E_Z)-e(Z)\bigr),
\]
where the sum runs over the disjoint strata modified at the relevant
stage and $E_Z$ is the exceptional locus over $Z$.  These summands are the
numbers denoted by $\Delta e$ in Section~3.  The same stratification,
together with the incidence relations between strict transforms, counts
the irreducible components and gives the recorded $\Delta N_c$.

There is one monodromy issue in that component count.  A cross bundle over a projective line may
be irreducible or the union of two $\PP^1$-bundles.  Such bundles arise
from transversal $A_1$ singularities with finitely many special points.
Locally analytically the equation has the form
\[
x^{2}+y^{2}t^{k}+\left(\text{terms vanishing to order at least 3 along the line } x=y=z=0 \right) =0, \;\;\textup{where }\; k\geq 1.
\]
The integer $k$ is the order of degeneration.  After normalizing the
exceptional divisor, its two rulings are defined over the double cover
obtained by adjoining a square root of the transverse coefficient.
A small loop around the special point acts on that square root by
$(-1)^k$.  Hence the two rulings are interchanged exactly when $k$ is
odd.  It follows that the cross bundle splits if and only if all local
orders of degeneration are even; otherwise the two local rulings form a
single irreducible surface by monodromy.

For a transversal $D_4$ singularity, the first exceptional divisor and
the exceptional divisor of the residual resolution are two distinct
components.  The normalization of the latter is governed by the analogous
degree-three cover; this is the source of the conic-bundle count used
below.

Cross bundles appear in the resolution of transversal $A_{k}$ and $E_{6}$ lines. Explicitly, we have to consider the following singularities:
\begin{align*}
x^{2}+z^{3}+y^{2}t&=0, &
x^{2}+z^{3}+y^{2}t^{2}&=0, &
x^{2}+z^{3}+y^{2}t^{3}&=0, \\
x^{2}+z^{3}+y^{2}t^{4}&=0, &
x^{2}+z^{3}+y^{4}t^{2}&=0, &
x^{2}+z^{3}+y^{4}t^{3}&=0, \\
x^{2}+z^{4}+y^{2}t&=0, &
x^{2}+z^{4}+y^{2}t^{2}&=0, &
x^{2}+z^{4}+y^{2}t^{3}&=0, \\
x^{2}+z^{4}+y^{3}t^{2}&=0.
\end{align*}
In each displayed equation, completion at the generic point of the
singular line and one step of the resolution reduce the quadratic part to
$uv=t^k$ times a unit.  Units have square roots in the completed local
ring after an unramified extension, so the preceding monodromy test shows
that the order of degeneration is congruent modulo $2$ to the displayed
exponent of $t$.

We now carry out the computation of \(N_c\) and \(e\) for every affirmative
entry of Table~\ref{tab:main}.  The negative entries are proved in
Section~2, while the open entry is Corollary~\ref{cor:2334}.
\medskip

\medskip
\noindent\emph{$\mathrm{I}_n \times \mathrm{I}_{m}$}.
The fiber of type $\mathrm{I}_n \times \mathrm{I}_{m}$ has $nm$ $A_{1}$ singularities. Consequently,
\[
e(\widetilde{X_{b}}) = e(\mathrm{I}_n)e(\mathrm{I}_{m}) + n\times m\times 1=2nm
\]
and
\[
N_{c}(\widetilde{X_{b}}) = N_{c}(\mathrm{I}_n)N_{c}(\mathrm{I}_{m}) = n\times m\times 1=nm.
\]
\medskip
\noindent\emph{$\mathrm{II} \times \mathrm{II}$}.
The fiber of type $\mathrm{II} \times \mathrm{II}$ has a unique $D_{4}$  singularity. Consequently,
\[
e(\widetilde{X_{b}}) = e(\mathrm{II})e(\mathrm{II}) + 2=2\times 2+2=6
\]
and
\[
N_{c}(\widetilde{X_{b}}) = N_{c}(\mathrm{II})N_{c}(\mathrm{II}) = 1.
\]
\medskip
\noindent\emph{$\mathrm{III} \times \mathrm{III}$}.
The fiber of type $\mathrm{III} \times \mathrm{III}$ has a unique singularity with local equation $x^{2}+y^{2}+z^{4}+t^{4}=0$. Consequently,
\[
e(\widetilde{X_{b}}) = e(\mathrm{III})e(\mathrm{III}) + 3=3\times 3+3=12
\]
and
\[
N_{c}(\widetilde{X_{b}}) = N_{c}(\mathrm{III})N_{c}(\mathrm{III}) = 2\times 2=4.
\]
\medskip
\noindent\emph{$\mathrm{IV} \times \mathrm{IV}$}.
The fiber of type $\mathrm{IV} \times \mathrm{IV}$ has a unique
ordinary triple point.
Consequently,
\[
e(\widetilde{X_{b}}) = e(\mathrm{IV})e(\mathrm{IV}) + 8=4\times 4+8=24
\]
and
\[
N_{c}(\widetilde{X_{b}}) = N_{c}(\mathrm{IV})N_{c}(\mathrm{IV}) = 3\times 3 +1 =10.
\]
\medskip
\noindent\emph{$\mathrm{I}_n \times \mathrm{III}$}.
The fiber of type $\mathrm{I}_n \times \mathrm{III}$ has $n$ singular points of type $A_{3}$. Consequently,
\[
e(\widetilde{X_{b}}) = e(\mathrm{I}_{n})e(\mathrm{III}) + n\times 1=n\times 3 + n\times 1=4n
\]
and
\[
N_{c}(\widetilde{X_{b}}) = N_{c}(\mathrm{I}_{n})N_{c}(\mathrm{III}) = n\times 2 + 0 = 2n.
\]
\medskip
\noindent\emph{$\mathrm{I}_n \times \mathrm{IV}$}.
The fiber of type $\mathrm{I}_n \times \mathrm{IV}$ has $n$ singular points of type $D_{4}$. Consequently,
\[
e(\widetilde{X_{b}}) = e(\mathrm{I}_{n})e(\mathrm{IV}) + n\times 2=n\times 4 + 2n=6n
\]
and
\[
N_{c}(\widetilde{X_{b}}) = N_{c}(\mathrm{I}_{n})N_{c}(\mathrm{IV}) + n\times 0  = n\times 3 + 0 = 3n.
\]
\medskip
\noindent\emph{$\mathrm{I}_n \times \mathrm{I}_{m}^{*}$}.
We separately consider the case $m=0$ and $m>0$.

\medskip
\noindent
\emph{Case 1: $m=0$.}

The fiber of type $\mathrm{I}_n \times \mathrm{I}_{0}^{*}$ has the following local singularities
\begin{itemize}
	\item [] $4n$ points with local equation: $x^{2}+y^{2}+z^{2}t$,
	\item [] $n$ transversal $A_{1}$ lines with four  special points with degeneration order 1.
\end{itemize}
Consequently,
\[
e(\widetilde{X_{b}}) = e(\mathrm{I}_n)e(\mathrm{I}_{0}^{*}) + 4n\times2 + n\times((-2) \times 1) = n\times 6 + 6n = 12n
\]
and
\[
N_{c}(\widetilde{X_{b}}) = N_{c}(\mathrm{I}_{n})N_{c}(\mathrm{I}_{0}^{*}) = n\times 5 + n \times 1 =5n + n = 6n.
\]
\medskip
\noindent
\emph{Case 2: $m>0$.}

The fiber of type $\mathrm{I}_n \times \mathrm{I}_{m}^{*}$, $m>0$, has the following local singularities
\begin{itemize}
	\item [] $4n$ points with local equation: $x^{2}+y^{2}+z^{2}t$,
	\item [] $nm$ points with local equation: $x^{2}+y^{2}+z^{2}t^{2}$,
	\item [] $2n$ transversal $A_{1}$ lines with three  special points,
	\item [] $n(m-1)$ transversal $A_{1}$ lines with two special points.
\end{itemize}
Consequently,
\[
e(\widetilde{X_{b}}) = e(\mathrm{I}_n)e(\mathrm{I}_{m}^{*}) + 4n\times2 + nm\times3 +2n \times (-1) + n(m-1)\times 0 = n\times (6+m) + 6n + 3nm = 4 mn + 12n
\]
and
\[
N_{c}(\widetilde{X_{b}}) = N_{c}(\mathrm{I}_{n})N_{c}(\mathrm{I}_{m}^{*}) +  2n +  n \times (m - 1) = n\times (5 + m) + nm + n = 2nm + 6n.
\]
\medskip
\noindent\emph{$\mathrm{I}_n \times \mathrm{IV}^{*}$}.
The fiber of type $\mathrm{I}_n \times \mathrm{IV}^{*}$ has the following local types of singularities
\begin{itemize}
	\item [] $3n$ points with local equation: $x^{2}+y^{2}+z^{2}t$,
	\item [] $3n$ points with local equation: $x^{2}+y^{2}+z^{3}t^{2}$,
	\item [] $3n$ transversal $A_{1}$ lines with two special points,
	\item [] $n$ transversal $A_{2}$ lines with three special points with degeneration order 2.
\end{itemize}
Consequently,
\[
e(\widetilde{X_{b}}) = e(\mathrm{I}_n)e(\mathrm{IV}^{*}) + 3n\times2 +3n \times 4 + 3n\times 0+n\times(-1)\times 2 =
n\times 8 + 16n = 24n
\]
and
\[
N_{c}(\widetilde{X_{b}}) = N_{c}(\mathrm{I}_{n})N_{c}(\mathrm{IV}^{*}) +  3n\times1 +  n \times 2 = n\times 7+5n=12n.
\]
\medskip
\noindent\emph{$\mathrm{I}_n \times \mathrm{III}^{*}$}.
The fiber of type $\mathrm{I}_n \times \mathrm{III}^{*}$ has the following local types of singularities
\begin{itemize}
	\item [] $2n$ points with local equation: $x^{2}+y^{2}+z^{2}t$,
	\item [] $2n$ points with local equation: $x^{2}+y^{2}+z^{3}t^{2}$,
	\item [] $n$ points with local equation: $x^{2}+y^{2}+z^{4}t^{2}$,
	\item [] $2n$ points with local equation: $x^{2}+y^{2}+z^{4}t^{3}$,
	\item [] $n$ transversal $A_{1}$ lines with one special point,
	\item [] $2n$ transversal $A_{1}$ lines with two special points,
	\item [] $2n$ transversal $A_{2}$ lines with two special points with degeneration order 2 and 4,
	\item [] $n$ transversal $A_{3}$ lines with three special points with degeneration order 2, 3 and 3.
\end{itemize}
Consequently,
\[
e(\widetilde{X_{b}}) = e(\mathrm{I}_n)e(\mathrm{III}^{*}) + 2n\times2 +2n \times 4 + n\times 5+2n\times 6 + n\times 1 +n\times 3\times (-1) =
n\times 9 + 27n = 36n
\]
and
\[
N_{c}(\widetilde{X_{b}}) = N_{c}(\mathrm{I}_{n})N_{c}(\mathrm{III}^{*}) +  n\times1 +  2n \times 1 +  2n \times 2 + n\times 3 = n\times 8+10n=18n.
\]
\medskip
\noindent\emph{$\mathrm{I}_n \times \mathrm{II}^*$}.

The fiber of type $\mathrm{I}_n \times \mathrm{II}^{*}$ has the following local types of singularities

\begin{itemize}
	\item [] $n$ points with local equation: $x^{2}+y^{2}+z^{2}t$,
	\item [] $n$ points with local equation: $x^{2}+y^{2}+z^{3}t^{2}$,
	\item [] $n$ points with local equation: $x^{2}+y^{2}+z^{4}t^{2}$,
	\item [] $n$ points with local equation: $x^{2}+y^{2}+z^{4}t^{3}$,
	\item [] $n$ points with local equation: $x^{2}+y^{2}+z^{5}t^{4}$,
	\item [] $n$ points with local equation: $x^{2}+y^{2}+z^{6}t^{3}$,
	\item [] $n$ points with local equation: $x^{2}+y^{2}+z^{6}t^{4}$,
	\item [] $n$ points with local equation: $x^{2}+y^{2}+z^{6}t^{5}$,
	\item [] $n$ transversal $A_{1}$ lines with one special point,
	\item [] $n$ transversal $A_{1}$ lines with two special points,
	\item [] $n$ transversal $A_{2}$ lines with one special point,
	\item [] $n$ transversal $A_{2}$ lines with two special points,
	\item [] $2n$ transversal $A_{3}$ lines with two special points,
	\item [] $n$ transversal $A_{4}$ lines with two special points,
	\item [] $n$ transversal $A_{5}$ lines with three special points.
\end{itemize}
Consequently,
\begin{multline*}
	e(\widetilde{X_{b}}) = e(\mathrm{I}_n)e(\mathrm{II}^{*}) + n\times2 +n\times 4 + n\times 5+n\times 6 + n\times 8 +n\times 8 + n\times 9 + n\times 10 + n\times 1 + n\times 2 - n \times 5 =\\
	=n\times 10 + 50n = 60n
\end{multline*}

and
\[
N_{c}(\widetilde{X_{b}}) = N_{c}(\mathrm{I}_{n})N_{c}(\mathrm{II}^{*}) + n\times1 + n\times1 + n \times 2+ n \times 2 +  2n \times 3 + n\times 4 + n\times 5 = n\times 9 + 21n = 30n.
\]
\medskip
\noindent\emph{$\mathrm{II} \times \mathrm{I}_{n}^{*}$}.
We separately consider the case $n=0$ and $n>0$.

\medskip
\noindent
\emph{Case 1: $n=0$.}

The fiber of type $\mathrm{II} \times \mathrm{I}_{0}^{*}$ has the following local singularities
\begin{itemize}
	\item [] $4$ points with local equation: $x^{2}+y^{3}+z^{2}t$,
	\item [] $1$ transversal $A_{2}$ lines with four  special points with degeneration order 1.
\end{itemize}
Consequently,
\[
e(\widetilde{X_{b}}) = e(\mathrm{II})e(\mathrm{I}_{0}^{*}) + 4\times1 + 1\times(-2)\times 2 = 2\times 6 + 0 = 12
\]
and
\[
N_{c}(\widetilde{X_{b}}) = N_{c}(\mathrm{II})N_{c}(\mathrm{I}_{0}^{*}) +1\times 1  = 1\times 5 + 1  = 6.
\]
\medskip
\noindent
\emph{Case 2: $n>0$.}

The fiber of type $\mathrm{II} \times \mathrm{I}_{n}^{*}$, $n>0$, has the following local singularities
\begin{itemize}
	\item [] $4$ points with local equation: $x^{2}+y^{3}+z^{2}t$,
	\item [] $n$ points with local equation: $x^{2}+y^{3}+z^{2}t^{2}$,
	\item [] $2$ transversal $A_{2}$ lines with three  special points with degeneration order 1,1 and 2,
	\item [] $(n-1)$ transversal $A_{2}$ lines with two special points with degeneration order 2 and 2.
\end{itemize}
Consequently,
\[
e(\widetilde{X_{b}}) = e(\mathrm{II})e(\mathrm{I}_{n}^{*}) + 4 + n\times4 +2 \times 2\times (-1) + (n-1)\times 0 = 2\times (6+n) + 4n =6n + 12
\]
and
\[
N_{c}(\widetilde{X_{b}}) = N_{c}(\mathrm{II})N_{c}(\mathrm{I}_{n}^{*}) +  2\times 1 + (n-1) \times 2 = 1\times (5 + n) + 2n = 3n+5.
\]
\medskip
\noindent\emph{$\mathrm{II} \times \mathrm{IV}^{*}$}.
The fiber of type $\mathrm{II} \times \mathrm{IV}^{*}$ has the following local types of singularities
\begin{itemize}
	\item [] $3$ points with local equation: $x^{2}+y^{3}+z^{2}t$,
	\item [] $3$ points with local equation: $x^{2}+y^{3}+z^{3}t^{2}$,
	\item [] $3$ transversal $A_{2}$ lines with two special points,
	\item [] $1$ transversal $D_{4}$ lines with three special points with degeneration order 2.
\end{itemize}
Consequently,
\[
e(\widetilde{X_{b}}) = e(\mathrm{II})e(\mathrm{IV}^{*}) + 3\times1 + 3 \times 3 + 3\times 0 + 1\times 4 \times (-1) =
2\times 8 + 8 = 24
\]
and
\[
N_{c}(\widetilde{X_{b}}) = N_{c}(\mathrm{II})N_{c}(\mathrm{IV}^{*}) +  3\times 1 +  1 \times 2 = 1\times 7 + 5 = 12.
\]
\medskip
\noindent\emph{$\mathrm{II} \times \mathrm{III}^*$}.
The fiber of type $\mathrm{II} \times \mathrm{III}^{*}$ has the following local types of singularities
\begin{itemize}
	\item [] $2$ points with local equation: $x^{2}+y^{3}+z^{2}t$,
	\item [] $2$ points with local equation: $x^{2}+y^{3}+z^{3}t^{2}$,
	\item [] $1$ points with local equation: $x^{2}+y^{3}+z^{4}t^{2}$,
	\item [] $2$ points with local equation: $x^{2}+y^{3}+z^{4}t^{3}$,
	\item [] $1$ transversal $A_{2}$ lines with one special point with degeneration order 4,
	\item [] $2$ transversal $A_{2}$ lines with two special points with degeneration order 1, 3,
	\item [] $2$ transversal $D_{4}$ lines with two special points with degeneration order 2 and 4,
	\item [] $1$ transversal $E_{6}$ lines with three special points with degeneration order 2, 3 and 3.
\end{itemize}
Consequently,
\[
e(\widetilde{X_{b}}) = e(\mathrm{II})e(\mathrm{III}^{*}) + 2\times1 +2 \times 3 + 1\times 10 + 2\times 8 + 2 + 1\times (-1) \times 6  =
2\times 9 + 30 = 48
\]
and
\[
N_{c}(\widetilde{X_{b}}) = N_{c}(\mathrm{II})N_{c}(\mathrm{III}^{*}) +  1\times1 +  2 \times 1 +  1 \times 2 + 2\times 1 + 2 \times 2 + 1 \times 4 = 1\times 8 + 15 = 23.
\]
\medskip
\noindent\emph{$\mathrm{III} \times \mathrm{I}_{n}^{*}$}.
We separately consider the case $n=0$ and $n>0$.

\medskip
\noindent
\emph{Case 1: $n=0$.}

The fiber of type $\mathrm{III} \times \mathrm{I}_{0}^{*}$ has the following local singularities
\begin{itemize}
	\item [] $4$ points with local equation: $x^{2}+y^{4}+z^{2}t$,
	\item [] $1$ transversal $A_{3}$ lines with four  special points with degeneration order 1.
\end{itemize}
Consequently,
\[
e(\widetilde{X_{b}}) = e(\mathrm{III})e(\mathrm{I}_{0}^{*}) + 4\times3 + 1\times(-2)\times 3 = 3\times 6 + 6 = 24
\]
and
\[
N_{c}(\widetilde{X_{b}}) = N_{c}(\mathrm{III})N_{c}(\mathrm{I}_{0}^{*}) +1\times 2  = 2\times 5 + 2  = 12.
\]
\medskip
\noindent
\emph{Case 2: $n>0$.}

The fiber of type $\mathrm{III} \times \mathrm{I}_{n}^{*}$, $n>0$, has the following local singularities
\begin{itemize}
	\item [] $4$ points with local equation: $x^{2}+y^{4}+z^{2}t$,
	\item [] $n$ points with local equation: $x^{2}+y^{4}+z^{2}t^{2}$,
	\item [] $2$ transversal $A_{3}$ lines with three  special points with degeneration order 1,1 and 2,
	\item [] $(n-1)$ transversal $A_{3}$ lines with two special points with degeneration order 2 and 2.
\end{itemize}
Consequently,
\[
e(\widetilde{X_{b}}) = e(\mathrm{III})e(\mathrm{I}_{n}^{*}) + 4\times 3 + n\times9 +2 \times 3\times (-1) + (n-1)\times 0 = 3\times (6+n) + 9n+6=12n+24
\]
and
\[
N_{c}(\widetilde{X_{b}}) = N_{c}(\mathrm{III})N_{c}(\mathrm{I}_{n}^{*}) + n\times 1 + 2 \times 2 + (n-1) \times 3 = 2\times (5 + n) + 4n +1 = 6n+11.
\]
\medskip
\noindent\emph{$\mathrm{III} \times \mathrm{IV}^{*}$}.
The fiber of type $\mathrm{III} \times \mathrm{IV}^{*}$ has the following local types of singularities
\begin{itemize}
	\item [] $3$ points with local equation: $x^{2}+y^{4}+z^{2}t$,
	\item [] $3$ points with local equation: $x^{2}+y^{4}+z^{3}t^{2}$,
	\item [] $3$ transversal $A_{3}$ lines with two special points with degeneration orders 1 and 3,
	\item [] $1$ transversal $E_{6}$ lines with three special points with degeneration order 2.
\end{itemize}
Consequently,
\[
e(\widetilde{X_{b}}) = e(\mathrm{III})e(\mathrm{IV}^{*}) + 3\times3 + 3 \times 15 + 1\times 6  \times (-1) =
3\times 8 + 48 = 72
\]
and
\[
N_{c}(\widetilde{X_{b}}) = N_{c}(\mathrm{III})N_{c}(\mathrm{IV}^{*}) +  3\times 3 +  3 \times 2 +1\times 6= 2\times 7 + 21 = 35.
\]
\medskip
\noindent\emph{$\mathrm{IV} \times \mathrm{I}_{n}^{*}$}.
We separately consider the case $n=0$ and $n>0$.

\medskip
\noindent
\emph{Case 1: $n=0$.}

The fiber of type $\mathrm{IV} \times \mathrm{I}_{0}^{*}$ has the following local singularities
\begin{itemize}
	\item [] $4$ points with local equation: $x^{3}+y^{3}+z^{2}t$,
	\item [] $1$ transversal $D_{4}$ lines with four  special points with degeneration order 1.
\end{itemize}
Consequently,
\[
e(\widetilde{X_{b}}) = e(\mathrm{IV})e(\mathrm{I}_{0}^{*}) + 4\times8 + 1\times(-2)\times 4 = 4\times 6 + 24 = 48
\]
and
\[
N_{c}(\widetilde{X_{b}}) = N_{c}(\mathrm{IV})N_{c}(\mathrm{I}_{0}^{*}) + 4\times 1+ 1\times 4  = 3\times 5 + 8  = 23.
\]
\medskip
\noindent
\emph{Case 2: $n>0$.}

The fiber of type $\mathrm{IV} \times \mathrm{I}_{n}^{*}$, $n>0$, has the following local singularities
\begin{itemize}
	\item [] $4$ points with local equation: $x^{3}+y^{3}+z^{2}t$,
	\item [] $n$ points with local equation: $x^{3}+y^{3}+z^{2}t^{2}$,
	\item [] $2$ transversal $D_{4}$ lines with three  special points with degeneration order 1,1 and 2,
	\item [] $(n-1)$ transversal $D_{4}$ lines with two special points with degeneration order 2 and 2.
\end{itemize}
Consequently,
\[
e(\widetilde{X_{b}}) = e(\mathrm{IV})e(\mathrm{I}_{n}^{*}) + 4\times 8 + n\times20 +2 \times 4\times (-1) + (n-1)\times 0 = 4\times (6+n) + 20n + 24 = 24n + 48
\]
and
\[
N_{c}(\widetilde{X_{b}}) = N_{c}(\mathrm{IV})N_{c}(\mathrm{I}_{n}^{*}) + n\times 5 + 4\times 1 + 2 \times 4 + (n-1) \times 4 = 3\times (n + 5) + 9n + 8 = 12n + 23.
\]
\end{proof}

\section{Hodge numbers of the resolved fiber product}

Let $S_{1}\to \PP^{1}$ and $S_{2}\to \PP^{1}$ be rational elliptic surfaces with section, and let
$
X=S_{1}\times_{\PP^{1}}S_{2}.
$
Assume that $X$ admits a crepant resolution
$
\tylda X \longrightarrow X.
$
Denote by $\Sigma_i\subset \PP^{1}$ the set of points over which $S_i$ has singular fiber. For each point $p\in \Sigma_1\cap \Sigma_2$, let $\tylda X_p$ be the fiber of $\tylda X$ over $p$, and let $N_c(\tylda X_p)$
denote the number of irreducible components of $\tylda X_p$. We also write $e(\tylda X_p)$ for the Euler characteristic of this fiber. Following the local computations of the previous section, it is convenient to introduce the local defect
\[
\delta_p:=\frac{e(\tylda X_p)}{2}-N_c(\tylda X_p).
\]
For a rational elliptic surface with section, the Shioda--Tate formula gives
\[
\operatorname{rk}\MW(S_i)=8 - \sum_{p\in\Sigma_i}\bigl(N_{p}(S_{i})-1\bigr),
\]
where $N_{p}(S_{i})$ is the number of irreducible components of the fiber $(S_{i})_{p}$.
Over $\CC$, the Euler characteristic of a Kodaira fiber equals
\[
e((S_{i})_{p}) = \begin{cases}
	N_{p}(S_{i}), &\text{if $(S_{i})_{p}$ is multiplicative},\\
	N_{p}(S_{i})+1, &\text{if $(S_{i})_{p}$ is additive}.
\end{cases}
\]
The additive fibers are precisely the non-semistable Kodaira fibers
$
\mathrm{II}$, $\mathrm{III}$, $\mathrm{IV}$, $\mathrm{I}_n^*$, $\mathrm{IV}^*$, $\mathrm{III}^*$, $\mathrm{II}^*
$ (see, for example, \cite{Miranda}).  Since $e(S_i)=12$, these two
equalities give the following convenient form of Shioda--Tate.
\begin{Lem}\label{lem:mw-rank}
Let $a_i$ be the number of additive fibers of $S_i$.  Then
\[
\operatorname{rk}\MW(S_i)=\#\Sigma_i-4+a_i.
\]
\end{Lem}
To simplify further formulas we shall use the notation
\[
\Sigma_{i}^{\text{mul}} \quad\text{ and } \quad \Sigma_{i}^{\text{add}}
\]
for the sets of multiplicative (semistable) and additive singular fibers.

Let $E_i/\CC(t)$ be the generic fiber of $S_i$ and put
\[
d:=\operatorname{rk}_{\ZZ}\operatorname{Hom}_{\CC(t)}(E_1,E_2).
\]
For non-isotrivial generic fibers this is $1$ when they are isogenous and
$0$ otherwise.  This formulation also covers the possible extra
endomorphisms in an isotrivial situation.

For $p\in (\Sigma_{1}\cup \Sigma_{2}) \setminus (\Sigma_{1}\cap \Sigma_{2})$ the fiber products $S_1\times_{\PP^1}S_2$ is non-singular over $p$ but the fiber $X_{p}$ is singular. The number of components of $X_{p}$ equals $N_{p}(S_{1})$ when $p\in\Sigma_{1}\setminus\Sigma_{2}$ and $N_{p}(S_{2})$ when $p\in\Sigma_{2}\setminus\Sigma_{1}$. Denote
\[R:=\sum_{p\in \Sigma_{1}\setminus \Sigma_{2}}(N_{p}(S_{1})-1) + \sum_{p\in \Sigma_{2}\setminus \Sigma_{1}}(N_{p}(S_{2})-1)\]
\begin{Pro}\label{prop:hodge}
Let $S_1,S_2$ be rational elliptic surfaces with section over $\CC$, and
suppose that $X=S_1\times_{\PP^1}S_2$ admits a \emph{smooth projective}
crepant resolution $\widetilde X$.  Write
$a_i=\#\Sigma_i^{\mathrm{add}}$.  Then
\begin{align*}
e(\widetilde X)
 &=\sum_{p\in\Sigma_1\cap\Sigma_2}e(\widetilde X_p),\\
h^{1,1}(\widetilde X)
 &=d+\#(\Sigma_1\cup\Sigma_2)-5+a_1+a_2 + R
   +\sum_{p\in\Sigma_1\cap\Sigma_2}N_c(\widetilde X_p),\\
h^{1,2}(\widetilde X)
 &=d+\#(\Sigma_1\cup\Sigma_2)-5+a_1+a_2 +R
   -\sum_{p\in\Sigma_1\cap\Sigma_2}\delta_p.
\end{align*}
\end{Pro}
\begin{proof}
For $p\notin\Sigma_1\cap\Sigma_2$, at least one factor of the fiber is a
smooth elliptic curve and therefore has Euler characteristic zero.  The
resolution is an isomorphism over such a fiber, and additivity over the
base proves the first formula.

We recall the divisor calculation, following Schoen's restriction
argument \cite[Section~3]{Schoen}.  The fiber class, the two zero-section
divisors, the Mordell--Weil divisors from the two elliptic surfaces, and
the graphs of a basis of $\operatorname{Hom}(E_1,E_2)$ are independent.
Over a singular value of at least one surface one further obtains $N_c(\widetilde X_p)-1$
independent vertical components.  Conversely, restricting a divisor to
the generic abelian surface and then subtracting these classes leaves a
vertical divisor; its coefficients are constant on every fiber and hence
it is a multiple of the fiber class.  Thus
\begin{align*}
\rho(\widetilde X)=3+d+\operatorname{rk}\MW(S_1)
+\operatorname{rk}\MW(S_2)
+\sum_{p\in\Sigma_1\cup\Sigma_2}\bigl(N_c(\widetilde X_p)-1\bigr)\\
=3+d+\operatorname{rk}\MW(S_1)
+\operatorname{rk}\MW(S_2) + R
+\sum_{p\in\Sigma_1\cap\Sigma_2}\bigl(N_c(\widetilde X_p)-1\bigr).
\end{align*}
Because $\widetilde X$ is smooth and projective and has
$H^1(\widetilde X,\mathcal O)=H^2(\widetilde X,\mathcal O)=0$, the
exponential sequence and Lefschetz $(1,1)$ theorem identify this rank with
$h^{1,1}$.  Substitution of Lemma~\ref{lem:mw-rank} gives the stated
formula.  Finally, $K_{\widetilde X}\simeq\mathcal O_{\widetilde X}$ and
the same vanishing give
\[
e(\widetilde X)=2\bigl(h^{1,1}(\widetilde X)-h^{1,2}(\widetilde X)\bigr).
\]
Combining this identity with the first two formulas and the definition of
$\delta_p$ proves the formula for $h^{1,2}$.
\end{proof}
\begin{Rem}
The previous formula shows that the Hodge numbers are controlled by two kinds of data:
\begin{itemize}
\item the global configuration of singular fibers, encoded by the sets $\Sigma_1,\Sigma_2$, the number of additive fibers, and the parameter $d$;
\item the local geometry of the singularities of the fiber product, encoded by the numbers $N_c(\tylda X_p)$, $e(\tylda X_p)$, or equivalently by the defects $\delta_p$ and $R$.
\end{itemize}
Hence the local classification obtained in the previous section feeds directly into the global computation of the Hodge numbers.
\end{Rem}
\begin{table}[ht]
	\centering
	\small
	\renewcommand{\arraystretch}{1.15}
	\setlength{\tabcolsep}{4pt}
	\begin{tabular}{>{\centering\arraybackslash}m{0.30\linewidth}||>{\raggedright\arraybackslash}m{0.60\linewidth}}
		\fb \ww{$\alpha=\#\{\textup{additive fibers}\}-\delta$}
		&
		\fb \ww{Pairs of fibers with crepant resolution} \\
		\hhline{=::=}
		\fb \ww{$\alpha=0$} &
		\ww{\begin{tabular}[t]{@{}l@{}}
				$\mathrm{I}_n\times \mathrm{I}_m,\;
				\mathrm{II}\times \mathrm{II},\;
				\mathrm{III}\times \mathrm{III},\;
				\mathrm{IV}\times \mathrm{IV}$
		\end{tabular}} \\
		\hhline{=::=}
		\fb \ww{$\alpha=1$} &
		\ww{\begin{tabular}[t]{@{}l@{}}
				$\mathrm{I}_n\times \mathrm{I}_m^{*},\;
				\mathrm{I}_n\times \mathrm{III},\;
				\mathrm{I}_n\times \mathrm{IV},\;
				\mathrm{I}_n\times \mathrm{IV}^{*},\;
				\mathrm{I}_n\times \mathrm{III}^{*},\;
				\mathrm{I}_n\times \mathrm{II}^{*},$ \\
				$\mathrm{I}_n^{*}\times \mathrm{II},\;
				\mathrm{I}_n^{*}\times \mathrm{III},\;
				\mathrm{I}_n^{*}\times \mathrm{IV},\;
				\mathrm{II}\times \mathrm{III}^{*},\;
				\mathrm{III}\times \mathrm{IV}^{*}$
		\end{tabular}} \\
		\hhline{=::=}
		\fb \ww{$\alpha=2$} &
		\ww{\begin{tabular}[t]{@{}l@{}}
				$\mathrm{I}_0^{*}\times \mathrm{II},\;
				\mathrm{I}_0^{*}\times \mathrm{III},\;
				\mathrm{II}\times \mathrm{IV}^{*}$
		\end{tabular}} \\
	\end{tabular}
\caption{Classification of admissible fiber pairs according to
		$\alpha=\#\{\textup{additive fibers}\}-\delta$, where
		$\delta=\frac{e}{2}-N_c$.}
\end{table}

\section{Numerical behavior in positive characteristic}
\label{sec6}

Proposition~\ref{prop:hodge} is a statement over $\CC$ and does not
extend verbatim to positive characteristic: its proof uses the complex
exponential sequence, the Lefschetz $(1,1)$ theorem, and Hodge
decomposition.  The local component counts and the additivity of the
$\ell$-adic Euler characteristic ($\ell\ne p$) remain meaningful when the
displayed resolutions are defined in characteristic $p$.  The examples
below record those numerical invariants without calling them complex
Hodge numbers.  In particular, a negative value obtained by formally
substituting into Proposition~\ref{prop:hodge} signals that the
substitution is invalid.

\begin{Exm} \label{ex:61}
 Let $S$ be the rational elliptic surface defined by a Weierstrass equation
\[
y^{2}=x^{3}+324t(5t-8)x-432t(32t^3-105t^2+60t+40).
\]
The discriminant of $S$ equals
\[
2^{18}\cdot 3^9\cdot t^2(t-1)^5(16t-25).
\]
From the Tate algorithm it follows that the surface $S$ has the following singular fiber types:

\begin{table}[H]
	\centering
	\renewcommand{\arraystretch}{1.2}
	\begin{tabular}{lcccc}
		\toprule
		Singular fiber
		& $\mathrm{IV}$
		& $\mathrm{II}$
		& $\mathrm{I}_{5}$ & $\mathrm{I}_{1}$ \\
		\midrule
		Point on the base & $t=\infty$ & $t=0$ & $t=1$ & $t=\frac{25}{16}$ \\
		\bottomrule
	\end{tabular}
\end{table}

The local constructions give a compact analytic crepant resolution of
$S\times_{\PP^{1}}S$.  Since projectivity (and even the K\"ahler property)
is not established for that model, Proposition~\ref{prop:hodge} is not
used to assign it Hodge numbers.

In characteristic $5$, the fibers $\mathrm{I}_1$ and $\mathrm{II}$ collide
and produce a fiber of type $\mathrm{III}$.  Consequently $S_5$ has the
Weierstrass equation
\[
y^2=x^3-2tx+t^4
\]
over $\FF_5$, with standard discriminant
\[
3t^3(t-1)^5.
\]
From the Tate algorithm it follows that the surface $S_5$ has the following singular fiber types:

\begin{table}[H]
	\centering
	\renewcommand{\arraystretch}{1.2}
	\begin{tabular}{lccc}
		\toprule
		Singular fiber
		& $\mathrm{IV}$
		& $\mathrm{III}$
		& $\mathrm{I}_{5}$ \\
		\midrule
		Point on the base & $t=\infty$ & $t=0$ & $t=1$ \\
		\bottomrule
	\end{tabular}
\end{table}

For the self-fiber product $S_{5}\times_{\PP^{1}}S_{5}$ the local table
gives
\[
\sum_pN_c(\widetilde X_p)=25+10+4=39,
\qquad e(\widetilde X)=50+24+12=86.
\]
The divisor count used over $\CC$ would formally give $42$, and the
complex identity would then give $42-86/2=-1$.  This contradiction is
precisely why these numbers must not be labelled $h^{1,1}$ and $h^{1,2}$
in characteristic $5$.
\end{Exm}

\begin{Exm}\label{ex:62} Let $S$ be the rational elliptic surface defined by a Weierstrass equation
\[
y^{2}=x^{3}-3(t^4-12t^3+14t^2+12t+1)x-2(t^2+1)(t^4-18t^3+74t^2+18t+1).
\]
The discriminant of $S$ equals
\[
2^12\cdot 3^6\cdot t^5\cdot (t^2-11t-1).
\]

From the Tate algorithm it follows that the surface $S$ has the following singular fiber types:

\begin{table}[H]
	\centering
	\renewcommand{\arraystretch}{1.2}
	\begin{tabular}{lcccc}
		\toprule
		Singular fiber
		& $\mathrm{I}_{5}$
		& $\mathrm{I}_{5}$
		& $\mathrm{I}_{1}$
		& $\mathrm{I}_{1}$ \\
		\midrule
		Point on the base
		& $t=\infty$
		& $t=0$
		& $t=\frac{11+5\sqrt{5}}{2}$
		& $t=\frac{11-5\sqrt{5}}{2}$ \\
		\bottomrule
	\end{tabular}
\end{table}

The fiber product $S\times_{\PP^{1}}S$ has a projective crepant resolution
$\widetilde X$ which is a rigid Calabi--Yau threefold with
$h^{1,1}=52$ \cite{Schoen}.  The surface is the modular surface for
$\Gamma_1(5)$, and the associated weight-$4$, level-$5$ modular form is
discussed in \cite{NorikoUpdate}.

In characteristic $5$, the two $\mathbf{I}_1$ fibers collide at $t=3$ and
produce a fiber of type $\mathbf{II}$. The surface $S_{5}$
obtained from $S$ by reduction modulo $5$ has, after the change of variable $x=(1-t')^{2}x', y=(1-t')^{3}y', t'=\frac{1+t}{1-t}$,
the following Weierstrass equation
\[
y^{2}=x^{3}+2(t+2)^{4}x+(t+2)(t+3)^{5}
\]
with discriminant
\[
(t-1)^5(t+1)^5(t+2)^2
\]
and singular fiber
types:
\begin{table}[H]
	\centering
	\renewcommand{\arraystretch}{1.2}
	\begin{tabular}{lccc}
		\toprule
		Singular fiber
		& $\mathrm{I}_{5}$
		& $\mathrm{I}_{5}$
		& $\mathrm{II}$ \\
		\midrule
		Point on the base & $t=1$ & $t=-1$ & $t=-2$ \\
		\bottomrule
	\end{tabular}
\end{table}
It is, up to reparametrization, the extremal rational elliptic surface with
three singular fibers from \cite[Theorem~4.1]{LANG}.  For its self-fiber
product the local data give
\[
\sum_pN_c(\widetilde X_p)=25+25+1=51,
\qquad e(\widetilde X)=50+50+6=106.
\]
The formal complex divisor formula gives $52$ and hence the meaningless
value $52-106/2=-1$; it does not compute Hodge numbers in characteristic
$5$.
\end{Exm}

%
\begin{Exm}\label{ex:63}
Let $S$ be the rational elliptic surface defined by the Weierstrass equation
\[
y^2=x^3-3t(t-1)(729t^2-945t+280)x
-2t(t-1)^2(19683t^3-28431t^2+11988t-1000).
\]
The discriminant of $S$ equals
\[
2^{12}\cdot 3^3\cdot t^2(t-1)^3(189t-125)^2.
\]
From the Tate algorithm it follows that the surface $S$ has the following
singular fiber types:
\begin{table}[H]
	\centering
	\renewcommand{\arraystretch}{1.2}
	\begin{tabular}{lcccc}
		\toprule
		Singular fiber
		& $\mathrm{I}_{5}$
		& $\mathrm{II}$
		& $\mathrm{III}$
		& $\mathrm{I}_{2}$ \\
		\midrule
		Point on the base
		& $t=\infty$
		& $t=0$
		& $t=1$
		& $t=\frac{125}{189}$ \\
		\bottomrule
	\end{tabular}
\end{table}

In characteristic $7$, the fibers $\mathbf{I_{5}}$ and $\mathbf{I_{2}}$
collide at $t=\infty$ and produce a fiber of type $\mathbf{I_{7}}$.
Consequently, the surface $S_7$ obtained from $S$ by reduction modulo $7$
has the Weierstrass equation
\[
y^2=x^3+4t^3(t-1)x+2t(t-1)^5
\]
over $\FF_7$, with discriminant
\[
6t^2(t-1)^3.
\]
From the Tate algorithm it follows that the surface $S_7$ has the following
singular fiber types:
\begin{table}[H]
	\centering
	\renewcommand{\arraystretch}{1.2}
	\begin{tabular}{lccc}
		\toprule
		Singular fiber
		& $\mathrm{I}_{7}$
		& $\mathrm{II}$
		& $\mathrm{III}$ \\
		\midrule
		Point on the base & $t=\infty$ & $t=0$ & $t=1$ \\
		\bottomrule
	\end{tabular}
\end{table}

For the self-fiber product the local data give
\[
\sum_pN_c(\widetilde X_p)=49+1+4=54,
\qquad e(\widetilde X)=98+6+12=116.
\]
The formal complex divisor formula gives $57$, again producing the
meaningless value $57-116/2=-1$.  No negative Hodge number is asserted.
\end{Exm}

Section~\ref{sec10} returns to Example~\ref{ex:61} through its Kummer
quotients.
\newpage

\section{Kummer fibrations}

Every birational map $\alpha_{k}\colon S_{k}\to S_{k}$ preserving the fibration
$\phi_{k}\colon S_{k}\to \PP^{1}$ is an automorphism
(\cite[Thm.~1, p.~197]{Szafarewicz}). For a pair of automorphisms
$\alpha_k\colon S_{k}\to S_{k}$ acting fiberwise $(k=1,2)$ we can consider the map
\[
\alpha:=\alpha_{1}\times_{\PP^{1}} \alpha_{2}\colon
S_{1}\times_{\PP^{1}} S_{2} \to S_{1}\times_{\PP^{1}} S_{2}.
\]
If the fiber product $X = S_{1}\times_{\PP^{1}} S_{2}$ admits a crepant resolution $\tylda X\to X$ and $\alpha$ lifts to a morphism $\tylda\alpha:\tylda X\to \tylda X$, then in fact $\tylda\alpha$ is an automorphism.
In the cases considered below, the lifted involution acts transversally to
each fixed curve as $\operatorname{diag}(-1,-1)$ and preserves the
canonical form.  The quotient therefore has a transversal $A_1$ locus,
whose blow-up is crepant; the additional isolated quotient germs are
treated case by case.  When the resulting model is smooth and projective,
it is again a Calabi--Yau threefold.

Schoen \cite{Schoen} studied the most natural case of this construction, namely
when the singular fibers $(S_{i})_{p}$, for $p\in\Sigma_{1}\cap\Sigma_{2}$, are of types $\mathrm{I}_{n_{i}}$ and the involution $\alpha$ comes from inversions $\alpha_{i}(x)= - x$ on the fibers.

Kapustka and Kapustka \cite{KK} considered fiberwise involutions $i_1,i_2$
when $S_i$ is a rational elliptic surface with section
$b_{i}\colon \PP^{1}\to S_{i}$ for $i=1,2$. Suppose that $S_{1}$ and $S_{2}$ admit only reduced (i.e. of type $\mathrm{I}_{n}$, $\mathrm{II}$, $\mathrm{III}$ or $\mathrm{IV}$)
singular fibers.
Then the fiber product $X:=S_{1}\times_{\PP^{1}} S_{2}$ carries the induced section
$b=(b_{1},b_{2})\colon \PP^{1}\to X$.

Now consider the involutions
\[
i_{k}\colon S_{k}\ni x\mapsto b_{k}(\phi_{k}(x))-x\in S_{k},
\]
where $\phi_{k}\colon S_{k}\to \PP^{1}$ is the elliptic fibration, for $k=1,2$.
This gives an involution
\[
i\colon X\ni (x_{1},x_{2})\mapsto \bigl(i_{1}(x_{1}),i_{2}(x_{2})\bigr)\in X.
\]
This is a well-defined involutive isomorphism preserving the canonical form
on the smooth locus of $X$.

We begin with the action of $i$ on singular fibers.  The following
description is \cite[Lemma~2.2]{KK}.

\begin{Lem} Assume that $S$ is a rational elliptic surface with a chosen zero section and reduced fibers. Let $i$ be the involution of the form $i\colon x\mapsto b-x$, where $b$ is a section of $S$. Let $F$ be a singular fiber of $S$. Then the following possibilities occur:
	\begin{enumerate}\leftskip=-9mm
		\begin{minipage}{.54\textwidth}
			\item[\textup{$(1)$}] The fiber $F$ is of type \textup{$\mathrm{I}_{1}$}. Then $i$ acts on $F$ by symmetry with three fixed points.
			\vspace{5mm}
		\end{minipage}
		\quad
		\begin{minipage}{.38\textwidth}
			\centering
			\includegraphics[width=0.4\textwidth]{node.1}
			\vspace{5mm}
		\end{minipage}

		\begin{minipage}{.54\textwidth}
			\item[\textup{$(2)$}] The fiber $F$ is of type \textup{$\mathrm{I}_{2k+1},$} where $k\geq 1$. Then $i$ acts on one of the components of $F$ by symmetry, with two fixed points away from the singularities of $F$, and interchanges the corresponding pairs of the remaining components. The singular point opposite to the component on which $i$ acts is fixed.
			\vspace{5mm}
		\end{minipage}
		\quad
		\begin{minipage}{.38\textwidth}
			\centering
			\includegraphics[width=0.4\textwidth]{nkaty.1}
			\vspace{5mm}
		\end{minipage}

\begin{minipage}{.94\linewidth}
			\item[\textup{$(3)$}] The fiber $F$ is of type \textup{$\mathrm{I}_{2k},$} where $k\geq 1$. Then we have one of the following cases:
			\vspace{5mm}
			\begin{enumerate}\leftskip=-9mm
				\begin{minipage}{.52\textwidth}
					\item[\textup{$(a)$}] The involution $i$ has two fixed points, namely two opposite singular points of $F$, and interchanges pairs of components of $F$.
					\vspace{5mm}
				\end{minipage}
				\quad
				\begin{minipage}{.38\textwidth}
					\centering
					\includegraphics[width=0.4\textwidth]{nkaty.2}
					\vspace{5mm}
				\end{minipage}

				\begin{minipage}{.52\textwidth}
					\item[\textup{$(b)$}] The involution $i$ acts by symmetry on two opposite components of $F$ and interchanges the corresponding remaining pairs of components. In this case $i$ has $4$ fixed points away from the singularities of $F$.
					\vspace{5mm}
				\end{minipage}
				\quad
				\begin{minipage}{.38\textwidth}
					\centering
					\includegraphics[width=0.4\textwidth]{nkaty.3}
					\vspace{5mm}
				\end{minipage}
			\end{enumerate}
		\end{minipage}

		\begin{minipage}{.54\textwidth}
			\item[\textup{$(4)$}] The fiber $F$ is of type \textup{$\mathrm{II}.$} Then $i$ has two fixed points, one of which is the singular point of $F$.
			\vspace{5mm}
		\end{minipage}
		\quad
		\begin{minipage}{.38\textwidth}
			\centering
			\includegraphics[width=0.3\textwidth]{cusp.1}
			\vspace{5mm}
		\end{minipage}

\begin{minipage}{.94\linewidth}
			\item[\textup{$(5)$}] The fiber $F$ is of type \textup{$\mathrm{III}.$} Then we have one of the following cases:
			\vspace{5mm}
			\begin{enumerate}\leftskip=-9mm
				\begin{minipage}{.52\textwidth}
					\item[\textup{$(a)$}] The involution $i$ fixes only the singular point and interchanges the two components of $F$.
					\vspace{5mm}
				\end{minipage}
				\quad
				\begin{minipage}{.38\textwidth}
					\centering
					\includegraphics[width=0.4\textwidth]{II.1}
					\vspace{5mm}
				\end{minipage}

				\begin{minipage}{.52\textwidth}
					\item[\textup{$(b)$}] The involution $i$ acts on both components, fixing the singular point and one additional point on each component, for a total of three fixed points.
					\vspace{5mm}
				\end{minipage}
				\quad
				\begin{minipage}{.38\textwidth}
					\centering
					\includegraphics[width=0.4\textwidth]{II.2}
					\vspace{5mm}
				\end{minipage}
			\end{enumerate}
		\end{minipage}

		\begin{minipage}{.54\textwidth}
			\item[\textup{$(6)$}] The fiber $F$ is of type \textup{$\mathrm{IV}.$} Then $i$ fixes the triple point, interchanges two components of the fiber, and acts on the third component with one additional fixed point. Thus there are two fixed points in total.
		\end{minipage}
		\quad
		\begin{minipage}{.38\textwidth}
			\centering
			\includegraphics[width=0.3\textwidth]{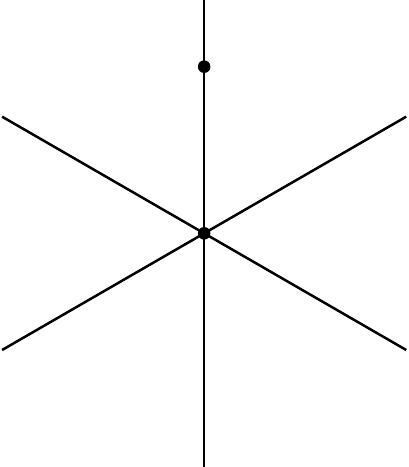}
		\end{minipage}

	\end{enumerate}
\end{Lem}

\begin{Rem}
	To simplify notation, we shall denote by $\mathrm{I}_{2k}^{a}$ and $\mathrm{I}_{2k}^{b}$ (resp. $\mathrm{III}^{a}$ and $\mathrm{III}^{b}$) a fiber of type $\mathrm{I}_{2k}$ (resp. $\mathrm{III}$) with the involution described in points $(3-a)$ and $(3-b)$ (resp. $(5-a)$ and $(5-b)$) of the previous lemma.
\end{Rem}

\begin{Cor}
	We can perform a local analytic change of coordinates around $P$ in such
	a way that the fibration is preserved and the surface $S$ is given in
	$\mathbb{C}^3$ with coordinates $(x,t,u)$ by the equation:
	\begin{enumerate}
		\item [$\mathrm{I}_{0}$:] $u^2 = x - t$ and the fibration is given by $t$,
		\item [$\mathrm{I}_{n}$:] $u^2 = x^2 		- t$ and the fibration is given by $t$,
		\item [$\mathrm{II}$:] $u^2 = x^3 - t$ and the fibration is given by $t$,
		\item [$\mathrm{III}^{a}$] $u^2 = x^4 - t$ and the fibration is given by $t$,
		\item [$\mathrm{III}^{b}$:] $u^2 = x - t$ and the fibration is given by $xt$,
		\item [$\mathrm{IV}$:] $u^2 = x^2 - t$ and the fibration is given by $xt$,
	\end{enumerate}
	where the involution is given by $u \mapsto -u$.
\end{Cor}

\begin{Thm}[Kummer resolutions]\label{thm:kummer}
Assume that the fiber product $X$ has singular fibers only of the following types:

\noindent
$\mathrm{I}_{n}\times \mathrm{I}_{m}$, $\mathrm{I}_{n}\times \mathrm{III}$, $\mathrm{III}\times \mathrm{I}_{n}$, $\mathrm{I}_{n}\times \mathrm{IV}$, $\mathrm{IV}\times \mathrm{I}_{n}$, $\mathrm{II}\times \mathrm{II}$, $\mathrm{II}\times \mathrm{IV}$, $\mathrm{IV}\times \mathrm{II}$, $\mathrm{III}\times \mathrm{III}$, $\mathrm{III}^{b}\times \mathrm{IV}$, $\mathrm{IV}\times \mathrm{III}^{b}$, $\mathrm{IV}\times \mathrm{IV}$,
\noindent
then there exists a crepant resolution of singularities of $\quot{X}{i}$ which is a smooth
compact complex threefold with trivial canonical bundle, not necessarily
projective.  It is called Calabi--Yau here only in the projective cases
with the required cohomology vanishing.

Moreover, if $X$ has only fibers of the following types:

\noindent
$\mathrm{I}_{n}\times \mathrm{I}_{m}$, $n>1, m>1$ except $\mathrm{I}_{2}^{b}\times \mathrm{I}_{2k+1}$,
$\mathrm{I}_{2}^{b}\times \mathrm{I}_{2k}^{a}$,  $\mathrm{I}_{2k+1}\times \mathrm{I}_{2}^{b}$,
$\mathrm{I}_{2k}^{a}\times \mathrm{I}_{2}^{b}$.

\noindent
$\mathrm{I}_{n}\times \mathrm{III}$ and  $\mathrm{III}\times \mathrm{I}_{n}$, except
$\mathrm{I}_{2}^{b}\times \mathrm{III}^{a}$,  $\mathrm{III}^{a}\times \mathrm{I}_{2}^{b}$,
$\mathrm{I}_{2k}^{a}\times \mathrm{III}^{b}$,  $\mathrm{III}^{b}\times \mathrm{I}_{2k}^{a}$,
$\mathrm{I}_{2k+1}\times \mathrm{III}^{b}$,  $\mathrm{III}^{b}\times \mathrm{I}_{2k+1}$,

\noindent
$\mathrm{I}_{n}\times \mathrm{IV}$,   $\mathrm{IV}\times \mathrm{I}_{n}$, $n>1$,
$\mathrm{I}_{2}^{a}\times \mathrm{IV}$,   $\mathrm{IV}\times \mathrm{I}_{2}^{a}$,

\noindent
$\mathrm{III}^{b}\times \mathrm{III}^{b}$, $\mathrm{III}^{b}\times \mathrm{IV}$, $\mathrm{IV}\times \mathrm{III}^{b}$, $\mathrm{IV}\times \mathrm{IV}$,

\noindent
then  $\quot{X}{i}$
admits a projective crepant resolution of singularities.
\end{Thm}

\begin{proof}
The semistable part of the argument is the equivariant construction of
Kapustka--Kapustka \cite{KK}; we include the local verification because the
additive cases require additional models.  Whenever the construction
below gives an $i$-equivariant crepant resolution $\widetilde X$, the
quotient $\widetilde X/i$ has only the displayed Gorenstein quotient
singularities.  Resolving their transversal ADE strata and the explicitly
listed isolated germs gives a crepant resolution of $X/i$.  The cases in
which the fiber product itself has no suitable equivariant resolution are
handled directly on the quotient.

We now consider, case by case, the possible types of the singular fiber $\pi^{-1}(s)=(S_{1})_{s}\times (S_{2})_{s}$.

\vspace{2mm}
\medskip
\noindent\emph{$\mathrm{I}_n \times \mathrm{I}_{m}$}.
We shall follow the arguments from the proof of \cite[Lemma 6.1]{Schoen}.
The fiber product $X=S_{1}\times S_{2}$ has $nm$ singularities of type $A_{1}$ contained in the fiber $\pi^{-1}(s)$.
If $n,m>1$, then these singularities have a crepant resolution by a blow-up of components of the fiber. We shall consider a sequence $D_{1}$, $i(D_{1}),\dots,D_{k},i(D_{k})$, where $D_{i}$ is a component of $\pi^{-1}(s)$. The involution $i$ lifts to the successive blow-ups of strict transforms of $D_{1}, i(D_{1}), \dots$ when the blow-up of the strict transform of $D_{j}$ followed by the blow-up of the strict transform of $i(D_{j})$ is isomorphic to the blow-up of the strict transform of $i(D_{j})$ followed by the blow-up of the strict transform of $D_{j}$, $j=1,\dots,k$. This condition is satisfied when the blow-ups of $D_{j}$ and $i(D_{j})$ are isomorphic.

A node $P\in\pi^{-1}(S)$ has two small blow-ups. Two components $D'$ and $D''$ of the fiber $\pi^{-1}(s)$ define the same blow-up at $P\in D'\cap D''$
if and only if $P$ is an isolated point of $D'\cap D''$.

A component $D$ equals $L^{1}\times L^{2}$, the intersection
$D\cap i(D) = (L^{1}\cap i(L^{1})) \times (L^{2}\cap i(L^{2}))$ has dimension 1 when $L^{1}$ is invariant under $i_{1}$ and $L^{2}$ and $i_{2}(L_{2})$ intersect, or when $L^{2}$ is invariant under $i_{2}$ and $L^{1}$ and $i_{1}(L_{1})$ intersect. The only case when a singular point of a semistable fiber belongs to two fixed components is $\mathrm{I}_{2}^{b}$, while a singular point belonging to two components interchanged by the involution exists in fibers of type $\mathrm{I}_{2n+1}$ and $\mathrm{I}_{2n}^{a}$.

Consequently, if $n,m>1$, then the Kummer fibration admits a crepant resolution, except for the cases
\[
\mathrm{I}_{2}^{b}\times \mathrm{I}_{2k+1}, \quad \mathrm{I}_{2}^{b}\times \mathrm{I}_{2k}^{a}, \quad \mathrm{I}_{2k+1}\times \mathrm{I}_{2}^{b}, \quad \mathrm{I}_{2k}^{a}\times \mathrm{I}_{2}^{b}
\]
%
\vspace{2mm}
\medskip
\noindent\emph{$\mathrm{I}_n \times \mathrm{II}$}.
The quotient $\quot{X}{i}$ has a local equation
\[
u^{2}=(x^{3}-t)(y^{2}-t).
\]
Blowing up the singular curve gives an isolated terminal factorial
singularity of the same type as for the fiber product.  Corollary
\ref{cor:Lin} therefore excludes an analytic crepant
resolution.

\vspace{2mm}
\medskip
\noindent\emph{$\mathrm{I}_n \times \mathrm{III}$}.
A singular fiber of type $\mathrm{I}_{n}\times\mathrm{III}$ contains
$n$ singular points, each with local equation
\[
xy-z(z-t^{2})=0.
\]
These are singularities of type $A_{3}$.

Each such singular point lies on four irreducible components of the
fiber, and blowing up any one of these components resolves the
singularity. If a singular point is isolated in the intersection of two
such components, then blowing up either component produces isomorphic
local resolutions.

Consequently, the construction of a projective crepant resolution is
exactly the same as in the case of a fiber of type
$\mathrm{I}_{n}\times\mathrm{I}_{2}$. Such a resolution exists except
in the following cases:
\[
\mathrm{I}_{2}^{b}\times\mathrm{III}^{a},\qquad
\mathrm{I}_{2k+1}\times\mathrm{III}^{b},\qquad
\mathrm{I}_{2k}^{a}\times\mathrm{III}^{b}.
\]
Then the Kummer fibration also admits a projective crepant resolution.

\vspace{2mm}
\medskip
\noindent\emph{$\mathrm{I}_n \times \mathrm{IV}$}.
A singular fiber of that type contains $n$ singular points of type $D_{4}$;
it has local equation $xy-zt(z+t)$. A projective crepant resolution can be performed by successive blow-ups of two components of $\pi^{-1}(s)$ intersecting in one point. We shall prove that there exists a crepant resolution of this fiber unless $n=2$ and the involution acts on each component separately.

Denote by $m_{0}, m_{1},m_{1}'$ components of the fiber $\mathrm{IV}$, where $m_{0}$ is fixed by the involution and $m_{1}, m_{1}'$ exchanged. Similarly, denote the
components of the fiber $\mathrm{I}_{n}$ by
\begin{itemize}
	\item [$\mathrm{I}_{2k}^{a}$:] $l_{1}, l_{1}',\dots,l_{k}, l_{k}'$, $n=2k$, $l_{1}\cap l_{1}'\ne\emptyset$,  $l_{i}\cap l_{i+1}\ne\emptyset$,   $l_{i}'\cap l_{i+1}'\ne\emptyset$ for $i=1,\dots,k-1$,  $l_{k}\cap l_{k}'\ne\emptyset$
	\item[$\mathrm{I}_{2k+1}$:] $l_{0}, l_{1}, l_{1}',\dots,l_{k}, l_{k}'$, $n=2k+1$, $l_{0}\cap l_{1}\ne\emptyset$, $l_{0}\cap l_{1}'\ne\emptyset$,  $l_{i}\cap l_{i+1}\ne\emptyset$,   $l_{i}'\cap l_{i+1}'\ne\emptyset$ for $i=1,\dots,k-1$,  $l_{k}\cap l_{k}'\ne\emptyset$
	\item[$\mathrm{I}_{2k+2}^{b}$:] $l_{0}, l_{1}, l_{1}',\dots,l_{k}, l_{k}',l_{k+1}$, $n=2k+2$, $l_{0}\cap l_{1}\ne\emptyset$, $l_{0}\cap l_{1}'\ne\emptyset$,  $l_{i}\cap l_{i+1}\ne\emptyset$,   $l_{i}'\cap l_{i+1}'\ne\emptyset$ for $i=1,\dots,k-1$,  $l_{k}\cap l_{k+1}\ne\emptyset$,  $l_{k+1}\cap l_{k}'\ne\emptyset$
\end{itemize}
 where $l_{0}$ and $l_{k+1}$ are fixed by the involution and $l_{i}$ and $l_{i}'$, $i=1,\dots,k$, are exchanged by the involution.

 Then blowing up in the following order
 \begin{itemize}
 	\item $l_{0}\times m_{0}$, in cases (2), (3),
	\item $l_{i}\times m_{1}$ and $l_{i}'\times m_{1}'$ for
	$i=1,\ldots,k$, in all three cases,
 	\item  $l_{k+1}\times m_{0}$, in case (3),
 \end{itemize}
 yields a projective, crepant resolution of the fiber to which the involution lifts.
 Consequently, there exists a projective crepant resolution of the Kummer fibration except in the case of fiber $I_{2}^{b}$.

\vspace{2mm}
\medskip
\noindent\emph{$\mathrm{II} \times \mathrm{II}$}.
The local equation of the singularity of the fiber product is again given by $xy-zt(z+t)$.
The involution, which is given this time by $(x,y,z,t)\longmapsto (-x,-y,z,t)$, preserves the ideal $\langle xy,xz, yt, zt\rangle$ of the union of planes $x=t=0$ and $y=z=0$. Consequently, the blow-up of this ideal is a crepant resolution of the fiber product and the involution lifts to this crepant resolution.
Hence the Kummer fibration admits an analytic crepant resolution.

\vspace{2mm}
\medskip
\noindent\emph{$\mathrm{II} \times \mathrm{III}$}.
The fiber product has an isolated terminal factorial point of type $E_6$,
so Corollary~\ref{cor:Lin} excludes an analytic
crepant resolution. We have two possible equations for the quotient $\quot{X}{i}$:
	\[
u^{2} = (x^{3}-t)(y^{4}-t)\qquad\text{and}\qquad u^{2}=(x^{3}-yz)(y-z).
\]
Blowing up the singular curve produces the same terminal factorial
$E_6$ germ in both cases.  Thus the quotient has no analytic crepant
resolution by Corollary~\ref{cor:Lin}.

\vspace{2mm}
\medskip
\noindent\emph{$\mathrm{II} \times \mathrm{IV}$}.
The factorial obstruction excludes a projective crepant resolution of the
fiber product.  The quotient $\quot{X}{i}$ has the local equation $u^{2}-(x^{3}-yz)(y^{2}-z)$. The blow-up at the singular curve has (in the affine chart $p_{1}=1$) equation \[
p_{2}^{2}y^{2}-p_{2}^{2}z-x^{3}+yz = \left(\mathit{p_{2}}^{4}+\mathit{p_{2}}^{2} y +z \right) \left(-\mathit{p_{2}}^{2}+y \right)+\left(\mathit{p_{2}}^{2}-x \right) \left(\mathit{p_{2}}^{4}+\mathit{p_{2}}^{2} x +x^{2}\right)
\]
(type $J_{2,0}$).

Blowing up the surface $x-p_{2}^{2}=y-p_{2}^{2}=0$, we get an $A_{3}$-singularity with equation
\[
p_{2}^{4} q_{1}^{2}-2 p_{2}^{4} q_{1} -p_{2}^{2} q_{1}^{2} x +p_{2}^{4}+p_{2}^{2} x -q_{1} z +x^{2}.
\] Consequently, there exists an analytic crepant resolution of singularities.

\vspace{2mm}
\medskip
\noindent\emph{$\mathrm{III} \times \mathrm{III}$}.
The fiber product has an analytic (non-projective) crepant resolution, and the quotient $\quot{X}{i}$ has a local equation of one of the following three forms:
\def\labelwidth{5cm}
\begin{itemize}\def\itemindent{15mm}
	\item [$\mathrm{III}^{a}\times \mathrm{III}^{a}$:] $u^{2}-(x^{4}-t)(y^{4}-t)=0$,
	\item [$\mathrm{III}^{a}\times \mathrm{III}^{b}$:] $u^2 - (x^4-yt)(y-t)=0$,
	\item [$\mathrm{III}^{b}\times \mathrm{III}^{b}$:] $u^2-(x-t)(y-z) = xt-yz=0.$
\end{itemize}
The blow-up of the quotient $\quot{X}{i}$ at the singular curve has the following singularities:
\begin{itemize}
	\item two singular points locally analytically isomorphic to the singularity $x^{2}+y^{2}+z^{4}+t^{4}=0$ (type $X_{9}$) in the first case,
	\item one singular point of type $X_{9}$ in the second case,
	\item a transversal $A_{1}$-singularity in the third case.
\end{itemize}
In all cases there exists a crepant resolution of singularities; in the third case it is projective.
\vspace{2mm}

\medskip
\noindent\emph{$\mathrm{III} \times \mathrm{IV}$}.

There is no analytic crepant resolution for the fiber product,
whereas the quotient $\quot{X}{i}$
has a local equation of one of the following two forms:
\def\labelwidth{5cm}
\begin{itemize}\def\itemindent{15mm}
	\item [$\mathrm{III}^{a}\times \mathrm{IV}^{}$:] $u^{2}-(x^{4}-yt)(y^{2}-t)$,
	\item [$\mathrm{III}^{b}\times \mathrm{IV}^{}$:] $u^{2}-(x-t)(y^{2}-z) = xt-yz =0$.
\end{itemize}
In the first case, blowing up the singular curve produces a singularity
analytically isomorphic to
\[
x^{2}+y^{2}+z^{4}+t^{6}=0,
\]
which is of type $W_{1,0}$. After an analytic change of coordinates, its
equation can be written as
\[
x(x+y^{2})-z(z+t^{3})=0.
\]
Blowing up the Weil divisor
\[
D=\{x=z=0\}
\]
gives, in the affine chart $p_{0}=1$, the strict transform
\[
(1-p_{1}^{2})x+y^{2}-p_{1}t^{3}=0.
\]
It has two singular points,
\[
x=y=t=0,\qquad p_{1}=\pm1,
\]
both of type $A_{2}$.

Suppose that
\[
Z\longrightarrow \quot{X}{i}
\]
were a crepant resolution. Then there would be a crepant birational map,
over $\quot{X}{i}$, between $Z$ and the partial resolution constructed
above. By \cite[Cor.~4.11]{Kollar}, this map factors as a sequence of
flops. Since the flops occurring in this situation preserve the local
analytic types of the singularities, the partial resolution would have
to be smooth. This contradicts the presence of the two $A_{2}$
singularities described above.

In the second case, blowing up curves of multiplicity 2 twice gives a projective crepant resolution.

\vspace{2mm}
\medskip
\noindent\emph{$\mathrm{IV} \times \mathrm{IV}$}.
The fiber product has an ordinary triple point, this singularity is resolved by a (big) blow-up of the singular point. The action of the involution lifts to the resolution of the fiber product  and the quotient $\quot{X}{i}$ has a projective, crepant resolution of singularities.
\end{proof}

\begin{Rem}
	The results of this section remain true for elliptic surfaces defined over any perfect field of characteristic $p\ne2,3$.
\end{Rem}

\section{Hodge numbers of a Kummer fibration Calabi--Yau threefold}

Kapustka and Kapustka gave a method for computing Hodge numbers by realizing $S_{1}$ and $S_{2}$ as double quartics (see \cite{KK}). We now present a different approach, better suited to explicit computations.

Our approach is based on the monodromy description of the curves in the fiber product fixed by the involution. Since the computation uses orbifold Chen-Ruan cohomology (see \cite{CR}), the method does not require an explicit description of the quotient of the fiber product by the involution, nor of a crepant resolution of that quotient. It applies, however, only when the fiber product $X$ admits a crepant resolution $\tylda{X}$ and the involution lifts to an involution of $\tylda{X}$, which by abuse of notation we again denote by $i$. Direct computations based on \cite[Lemma 2.2]{KK} show that this condition is satisfied for fiber products with singular fibers of types \textup{$F\times \mathrm{I}_{0},$} \textup{$\mathrm{I}_{n}\times \mathrm{I}_m,$} \textup{$\mathrm{III}\times \mathrm{I}_n,$} \textup{$\mathrm{III}\times \mathrm{III},$} \textup{$\mathrm{IV}\times \mathrm{I}_n,$} \textup{$\mathrm{II}\times \mathrm{II},$} and \textup{$\mathrm{IV}\times \mathrm{IV}$}.

In the special case of an involution, orbifold cohomology and the strong
McKay correspondence \cite{CR,BatyrevDais} yield the following formula for
the Hodge numbers of any projective crepant resolution $Y$ of
$\quot{X}{i}$:
\[
H^{1,2}(Y)=H^{1,2}(\tylda{X})^{\langle i\rangle}+\bigoplus_{C\in  \Lambda(\Fix(i))}H^{0,1}(C),
\]
where $\Lambda(\Fix(i))$ denotes the set of connected components of the fixed locus of $i$. All components of $\Fix(i)$ contained in fibers are smooth rational curves; this is immediate in the case of a small resolution. For the fiber \textup{$\mathrm{IV}\times \mathrm{IV}$}, which is the only case requiring a non-small crepant resolution, the local analytic equation of the singular fiber product is $x(x^2-y^2)-z(z^2-t^2)=0$, and the corresponding involution is given by $i\colon (x,y,z,t)\longmapsto (x, -y, z, -t)$. An explicit computation in local analytic coordinates shows that $\Fix(i)$ is a disjoint union of four curves.

Therefore
\[
h^{1,2}(Y)=
\dim H^{1,2}\left(\widetilde X\right)^{\langle i\rangle}+\sum_{C\in  \Lambda(\tylda{C_{1}\times_{\PP^1} C_{2}})}g(C) + g(C_{0}),
\]
where
$C_{1}:=\Fix(i_{1}),$ $C_{2}:=\Fix(i_{2})$, and $\Lambda(\tylda{C_{1}\times_{\PP^{1}} C_{2}})$ denotes the set of connected components of the normalization $\tylda{C_{1}\times_{\PP^{1}} C_{2}}$ of the fiber product $C_{1}\times_{\PP^{1}} C_{2}$ and $C_{0}$ is the unique fixed curve of the lifted involution on the exceptional cubic surface.

Since $C_{0}$ is a projective line, $g(C_{0}) = 0$.
By the Hurwitz formula,
\begin{equation}\label{hurhur}
\sum_{C\in  \Lambda(\tylda{C_{1}\times_{\PP^1} C_{2}})}\left(2-2g(C)\right)=2\cdot 16-\sum_{P\in \tylda{C_1\times_{\PP^{1}} C_{2}}}(e_P-1),
\end{equation}
where $e_P$ denotes the ramification index at $P$. Therefore, in order to determine $\displaystyle \sum_{C\in  \Lambda(\tylda{C_{1}\times_{\PP^1} C_{2}})}g(C)$, it suffices to compute the following quantities:
\begin{itemize}
	\item the number of components of $\tylda{C_{1}\times_{\PP^1} C_{2}},$
	\item the ramification indices of points $P\in \tylda{C_1\times_{\PP^1} C_{2}}.$
\end{itemize}

\subsection{Monodromy approach}

The curves $C_1$ and $C_2$ can be described in terms of their monodromy representations:
\begin{Thm}[\cite{mirandacurves}] Let $B=\{b_{1}, b_{2}, \ldots, b_{n}\}\subset \PP^{1}$ be a finite set of points. Then there is a $1$-$1$ correspondence
\[
\left\{\begin{array}{lr}
\textup{Isomorphism classes of holomorphic}\\
\textup{maps $F\colon X\to \PP^{1}$ of degree $d$}\\
\textup{whose branch points lie in $B$}
\end{array}\right\}\leftrightarrow
\left\{\begin{array}{lr}
\textup{Conjugacy classes of $n$-tuples $(\sigma_{1}, \ldots, \sigma_{n})$}\\
\textup{of permutations in $\Sigma_{d}$ such that $\sigma_{1}\ldots\sigma_{n}=1$}\\
\textup{and the subgroup generated by the $\sigma_{i}$'s}\\
\textup{is transitive}
\end{array}\right\}.
\]
\end{Thm}

Forming the product representation $\sigma_1=\sigma_1^{(1)}\times \sigma_1^{(2)}, \ldots, \sigma_n=\sigma_n^{(1)}\times \sigma_n^{(2)}$, we obtain the monodromy representation of $\tylda{C_1\times_{\PP^{1}} C_2}$. This allows us to determine the connected components of $\tylda{C_1\times_{\PP^{1}} C_2}$, their genera, and hence the total sum of genera.

\subsection{Ramification index of points of the normalized fiber product}

If $i_1$ and $i_2$ are involutions, then the curves $C_1$ and $C_2$ are fourfold covers of $\PP^1$. If, in addition, they are reducible, which in particular happens for a standard Kummer fibration, then the sum of genera of the irreducible components of $\tylda{C_1\times_{\PP^{1}} C_2}$ can also be computed from the behavior of the involution on the singular fibers.

We now give a more direct description of the components of $\Fix(i)$ in the case where $i_1$ and $i_2$ are non-symplectic involutions. In this situation both curves $C_{1}$ and $C_{2}$, possibly reducible, define $4:1$ coverings of $\PP^{1}$.
\begin{displaymath}
	\begin{tikzcd}[column sep=35pt, row sep=35pt]
		C_{1} \ar[dr,"4:1"'] & & C_{2} \ar[dl,"4:1"] \\
		& \PP^{1} &
	\end{tikzcd}
\end{displaymath}

Let us compute the number of irreducible components of $C_{1}\times_{\PP^{1}} C_{2}$. Assume that $D_{1}$ and $D_{2}$ are irreducible components of $C_{1}$ and $C_{2}$, respectively. Then for each irreducible connected component $C\subseteq D_{1}\times_{\PP^{1}} D_{2}$ we have the following diagram:
\begin{displaymath}
	\begin{tikzcd}[column sep=35pt, row sep=35pt]
		& C \ar[dr,""'] \ar[dl,""]\\
		D_{1} \ar[dr,""'] & & D_{2} \ar[dl,""] \\
		& \PP^{1} &
	\end{tikzcd}
\end{displaymath}
where $D_{k}\to \PP^{1}$ is a covering of degree at most $4$ for $k=1,2$.

\begin{Thm}[Orbit criterion for a normalized fiber product]\label{fund}
Let $D_i\to\PP^1$ be connected covers of degrees at most $4$.  The
connected components of the normalization of
$D_1\times_{\PP^1}D_2$ are the orbits of the product monodromy on the
product of the two geometric fibers.  In particular, the following
low-degree patterns occur.  The degree-four list includes the patterns
needed below; monodromy, rather than ramification indices alone, decides
which pattern occurs.
\begin{enumerate}\leftskip=-5mm
\item[\textup{$(1)$}]\hypertarget{11} If $D_{1}\to \PP^{1}$, $D_{2}\to \PP^{1}$ are $a:1$ and $b:1$ coverings, where $\gcd(a,b)=1$
\begin{displaymath}
	\begin{tikzcd}[column sep=35pt, row sep=35pt]
		& C \ar[dr,"a:1"] \ar[dl,"b:1"']\\
		D_{1} \ar[dr,"a:1"'] & & D_{2} \ar[dl,"b:1"] \\
		& \PP^{1} &
	\end{tikzcd}
	\hspace{20mm}
	\begin{tikzcd}[column sep=35pt, row sep=35pt]
		& C \ar[dr,"b:1"] \ar[dl,"a:1"']\\
		D_{1} \ar[dr,"b:1"'] & & D_{2} \ar[dl,"a:1"] \\
		& \PP^{1} &
	\end{tikzcd}
\end{displaymath}
then $C\simeq \operatorname{Norm}(D_{1}\times_{\PP^{1}} D_{2})$ and hence we get exactly \textbf{one} component.

\item[\textup{$(2)$}]\hypertarget{22} If $D_{1}\to \PP^{1}$, $D_{2}\to \PP^{1}$ are $2:1$ coverings, we have two possible diagrams
\begin{displaymath}
	\begin{tikzcd}[column sep=35pt, row sep=35pt]
		& C \ar[dr,"1:1"] \ar[dd,"2:1",dashed] \ar[dl,"1:1"']\\
		D_{1} \ar[dr,"2:1"'] & & D_{2} \ar[dl,"2:1"] \\
		& \PP^{1} &
	\end{tikzcd}
\hspace{20mm}
\begin{tikzcd}[column sep=35pt, row sep=35pt]
	& C \ar[dr,"2:1"] \ar[dd,"4:1",dashed] \ar[dl,"2:1"']\\
	D_{1} \ar[dr,"2:1"'] & & D_{2} \ar[dl,"2:1"] \\
	& \PP^{1} &
\end{tikzcd}
\end{displaymath}
then one of the following possibilities holds:
\begin{itemize}
	\item[\textup{$(a)$}] $D_{1}\times_{\PP^{1}} D_{2}$ has \textbf{two} components, each corresponding to the first diagram; in this case $D_{1}\simeq D_{2}$ as covers of $\PP^{1}$,
	\item[\textup{$(b)$}] $D_{1}\times_{\PP^{1}} D_{2}$ is \textbf{irreducible}.
\end{itemize}

\item[\textup{$(3)$}]\hypertarget{33} If $D_{1}\to \PP^{1}$, $D_{2}\to \PP^{1}$ are $3:1$ coverings, we have three possible diagrams
\begin{displaymath}
	\begin{tikzcd}[column sep=35pt, row sep=35pt]
		& C \ar[dr,"1:1"] \ar[dd,"3:1",dashed] \ar[dl,"1:1"']\\
		D_{1} \ar[dr,"3:1"'] & & D_{2} \ar[dl,"3:1"] \\
		& \PP^{1} &
	\end{tikzcd}
	\hspace{7mm}
	\begin{tikzcd}[column sep=35pt, row sep=35pt]
		& C \ar[dr,"2:1"] \ar[dd,"6:1",dashed] \ar[dl,"2:1"']\\
		D_{1} \ar[dr,"3:1"'] & & D_{2} \ar[dl,"3:1"] \\
		& \PP^{1} &
	\end{tikzcd}
\hspace{7mm}
\begin{tikzcd}[column sep=35pt, row sep=35pt]
	& C \ar[dr,"3:1"] \ar[dd,"9:1",dashed] \ar[dl,"3:1"']\\
	D_{1} \ar[dr,"3:1"'] & & D_{2} \ar[dl,"3:1"] \\
	& \PP^{1} &
\end{tikzcd}
\end{displaymath}
then one of the following possibilities holds:
\begin{itemize}
	\item[\textup{$(a)$}] $D_{1}\times_{\PP^{1}} D_{2}$ has \textbf{three} components, each of them corresponds to the first diagram; in this case $D_{1}\simeq D_{2}$ and the covering $D_{1}\to \PP^{1}$ is cyclic,
	\item[\textup{$(b)$}] $D_{1}\times_{\PP^{1}} D_{2}$ has \textbf{two} components, one corresponding to the first diagram and the other to the second; in this case $D_{1}\simeq D_{2}$ but the covering $D_{1}\to \PP^{1}$ is not cyclic,
	\item[\textup{$(c)$}] $D_{1}\times_{\PP^{1}} D_{2}$ is \textbf{irreducible}.
\end{itemize}

\item[\textup{$(4)$}]\hypertarget{44} If $D_{1}\to \PP^{1}$ and $D_{2}\to \PP^{1}$ are $2:1$ and $4:1$ coverings, respectively, we have two possible diagrams (and two symmetric ones)
\begin{displaymath}
	\begin{tikzcd}[column sep=35pt, row sep=35pt]
		& C \ar[dr,"1:1"] \ar[dd,"4:1",dashed] \ar[dl,"2:1"']\\
		D_{1} \ar[dr,"2:1"'] & & D_{2} \ar[dl,"4:1"] \\
		& \PP^{1} &
	\end{tikzcd}
	\hspace{20mm}
	\begin{tikzcd}[column sep=35pt, row sep=35pt]
		& C \ar[dr,"2:1"] \ar[dd,"8:1",dashed] \ar[dl,"4:1"']\\
		D_{1} \ar[dr,"2:1"'] & & D_{2} \ar[dl,"4:1"] \\
		& \PP^{1} &
	\end{tikzcd}
\end{displaymath}
then one of the following possibilities holds:
\begin{itemize}
	\item[\textup{$(a)$}] $D_{1}\times_{\PP^{1}} D_{2}$ has \textbf{two} components, each of them corresponds to the first diagram and $\pi_{2}\colon D_{2}\to \PP^{1}$ factors through $\pi_{1}\colon D_{1}\to \PP^{1}$,
	\item[\textup{$(b)$}] $D_{1}\times_{\PP^{1}} D_{2}$ is \textbf{irreducible}.
\end{itemize}

\item[\textup{$(5)$}]\hypertarget{55} If $D_{1}\to \PP^{1}$ and $D_{2}\to \PP^{1}$ are $4:1$ coverings, we have three possible diagrams
\begin{displaymath}
	\begin{tikzcd}[column sep=35pt, row sep=35pt]
		& C \ar[dr,"1:1"] \ar[dd,"4:1",dashed] \ar[dl,"1:1"']\\
		D_{1} \ar[dr,"4:1"'] & & D_{2} \ar[dl,"4:1"] \\
		& \PP^{1} &
	\end{tikzcd}
\hspace{7mm}
	\begin{tikzcd}[column sep=35pt, row sep=35pt]
		& C \ar[dr,"2:1"] \ar[dd,"8:1",dashed] \ar[dl,"2:1"']\\
		D_{1} \ar[dr,"4:1"'] & & D_{2} \ar[dl,"4:1"] \\
		& \PP^{1} &
	\end{tikzcd}
\hspace{7mm}
\begin{tikzcd}[column sep=35pt, row sep=35pt]
	& C \ar[dr,"4:1"] \ar[dd,"16:1",dashed] \ar[dl,"4:1"']\\
	D_{1} \ar[dr,"4:1"'] & & D_{2} \ar[dl,"4:1"] \\
	& \PP^{1} &
\end{tikzcd}
\end{displaymath}
then one of the following possibilities holds:
\begin{itemize}
	\item[\textup{$(a)$}] $D_{1}\times_{\PP^{1}} D_{2}$ has \textbf{four} components; each of them corresponds to the first diagram,
	\item[\textup{$(b)$}] $D_{1}\times_{\PP^{1}} D_{2}$ has \textbf{three} components; one of them corresponds to the second diagram and the latter two to the first diagram,
	\item[\textup{$(c)$}] $D_{1}\times_{\PP^{1}} D_{2}$ has \textbf{two} irreducible components; each of them corresponds to the second diagram,
	\item[\textup{$(d)$}] $D_{1}\times_{\PP^{1}} D_{2}$ is \textbf{irreducible}.
	\item[\textup{$(e)$}] if $D_1\simeq D_2$ and the degree-four
	monodromy is two-transitive, the diagonal is one component of degree
	$4$ and its complement is one component of degree $12$.
\end{itemize}

\end{enumerate}

In cases \hyperlink{44}{$(4)$} and \hyperlink{55}{$(5)$} we can state only necessary conditions in terms of ramification indices.

Finally, note that a pair of points lying over the same $b\in \PP^{1}$, with ramification indices $e_{1}$ and $e_{2}$ respectively, gives rise to $d$ points in $\tyldaa{D_{1}\times_{\PP^{1}}D_{2}}$ with ramification indices $\dfrac{e_{1}e_{2}}{d}$, where $d=\gcd(e_{1},e_2)$.
\end{Thm}

\begin{proof}
Let $B$ contain all branch values and put $U=\PP^1\setminus B$.  Choose a
base point of $U$ and write $\Omega_i$ for the geometric fiber of $D_i$.
The restriction $D_i|_U$ is the finite étale cover determined by a
transitive representation
\[
\rho_i\colon\pi_1(U)\longrightarrow\operatorname{Sym}(\Omega_i).
\]
The fiber product over $U$ is therefore the cover with fiber
$\Omega_1\times\Omega_2$ and monodromy
\[
\gamma\cdot(x_1,x_2)=\bigl(\rho_1(\gamma)x_1,
\rho_2(\gamma)x_2\bigr).
\]
Connected finite étale covers correspond to transitive monodromy sets.
Consequently, the connected components of the normalization of the fiber
product are exactly the orbits $O\subset\Omega_1\times\Omega_2$.  The
degree of the corresponding component over $\PP^1$ is $|O|$, and its
degrees over $D_1$ and $D_2$ are respectively
$|O|/|\Omega_1|$ and $|O|/|\Omega_2|$.  This proves the asserted diagrams
by inspecting transitive permutation groups of degrees $2$, $3$, and $4$.
For instance, equal quadratic covers give the two graph orbits, whereas
distinct quadratic extensions have one orbit of size $4$; an equal
non-Galois cubic cover has the diagonal orbit of size $3$ and the
off-diagonal orbit of size $6$.  In degree $4$, a two-transitive group has
the diagonal orbit of size $4$ and the off-diagonal orbit of size $12$,
which explains case~(5e).  The factorization alternative in case~(4) is
equivalently the existence of a two-element quotient of the degree-four
monodromy set.

It remains to justify the ramification formula.  Complete at a pair of
points over the same branch value.  Up to units the two maps have local
forms $t=u^{e_1}$ and $t=v^{e_2}$, so the completed fiber product is
\[
\CC[[u,v]]/(u^{e_1}-v^{e_2}).
\]
Writing $g=\gcd(e_1,e_2)$ factors the equation into $g$ branches.  The
normalization of each branch is parametrized by
$u=s^{e_2/g}$ and $v=\zeta s^{e_1/g}$ and has ramification index
$\operatorname{lcm}(e_1,e_2)=e_1e_2/g$.  This is the formula stated at the
end of the theorem.
\end{proof}

We end this section with a demonstration of the above two approaches to the computation of the Hodge numbers of the Kummer fibration associated to the self-fiber product of the Beauville surface \textup{$\mathrm{I}_{3}\mathrm{I}_{3}\mathrm{I}_{3}\mathrm{I}_{3}$}.

\begin{Exm} Let $X=\tylda{\quot{(S\times_{\PP^{1}} S)}{i}}$ be a Kummer fibration associated to self-fiber product of Beauville surface \textup{$\mathrm{I}_{3}\mathrm{I}_{3}\mathrm{I}_{3}\mathrm{I}_{3}$}. $S$ is the rational elliptic surface $S\to \PP^{1}(t)$ given by the Weierstrass equation
\[
y^{2}=4x^3-3(8t^3+1)x-8t^6-20t^3+1.
\]
The discriminant of the right-hand side is
\[
432t^3\left(2t+1+\sqrt{-3}\right)^3(t-1)^3\left(-2t-1+\sqrt{-3}\right)^3.
\]
Since $4x^3-3(8t^3+1)x-8t^6-20t^3+1$ is irreducible, we obtain a $3$-section, and this covering is not cyclic, its self-fiber product has $5$ irreducible components: three arising from case \hyperlink{11}{$(1)$} and two from case \hyperlink{33}{$(3)$} of Theorem~\ref{fund}.

The ramification indices of the fixed points in $\mathrm{I}_{3}\times \mathrm{I}_{3}$ are
$1,1,1,2,2,2,2,2,2,$ hence the equation \eqref{hurhur} gives
\[
10-2\sum_{C\in  \Lambda(C_{1}\times_{\PP^1} C_{2})}g(C)=2\cdot 16-4\cdot \left(6(2-1)+4(1-1)\right)=8,
\]
and consequently $\displaystyle \sum_{C\in  \Lambda(C_{1}\times_{\PP^1} C_{2})}g(C)=1$.

We can describe the geometry of the fixed locus in more detail using a monodromy argument. The monodromy representation, computed with the function \texttt{monodromy} from the \texttt{MAPLE} package \texttt{algcurves}, is given in the following table.

\begin{table}[H]
	\begin{center}
		\begin{varwidth}{\textheight}\rowcolors{2}{gray!25}{white}
			\rowcolors{2}{gray!25}{white}
			\resizebox{0.3\textheight}{!}{
				\begin{tabular}{c||c||c||c}
					\hhline{~|~|~|~}
					\cellcolor{gray!30}\setlength{\fboxrule}{5pt}\fcolorbox{gray!30}{gray!30}{$0$} & \cellcolor{gray!30} \setlength{\fboxrule}{5pt}\fcolorbox{gray!30}{gray!30}{$1$} & \cellcolor{gray!30} \setlength{\fboxrule}{5pt}\fcolorbox{gray!30}{gray!30}{$\zeta_{3}$} & \cellcolor{gray!30} \setlength{\fboxrule}{5pt}\fcolorbox{gray!30}{gray!30}{$\zeta_{3}^2$} \\
					\hhline{=::=::=::=}
					 \setlength{\fboxrule}{5pt}\fcolorbox{white}{white}{$(2,3)$} & \setlength{\fboxrule}{5pt}\fcolorbox{white}{white}{$(1,3)$} & \setlength{\fboxrule}{5pt}\fcolorbox{white}{white}{$(1,2)$} & \setlength{\fboxrule}{5pt}\fcolorbox{white}{white}{$(1,3)$} \\
				\end{tabular}
			}
		\end{varwidth}
		\captionof{table}{Monodromy representation of the Beauville surface $\mathrm{I}_{3}\mathrm{I}_{3}\mathrm{I}_{3}\mathrm{I}_{3}$}
	\end{center}
\end{table}

Therefore the induced monodromy representation in $\Sigma_{4}\times \Sigma_{4} \subset \Sigma_{16}$ is

\begin{table}[H]
	\begin{center}
		\begin{varwidth}{\textheight}\rowcolors{2}{gray!25}{white}
			\rowcolors{2}{gray!25}{white}
			\resizebox{0.35\textheight}{!}{
				\begin{tabular}{c||c}
					\hhline{~|~}
					\cellcolor{gray!30}\setlength{\fboxrule}{5pt}\fcolorbox{gray!30}{gray!30}{$0$} & \cellcolor{gray!30} \setlength{\fboxrule}{5pt}\fcolorbox{gray!30}{gray!30}{$(2,3)(5,9)(6,11)(7,10)(12,8)(14,15)$} \\
					\hhline{=::=}
					\setlength{\fboxrule}{5pt}\fcolorbox{white}{white}{$1$} & \setlength{\fboxrule}{5pt}\fcolorbox{white}{white}{$(1,11)(2,10)(3,9)(4,12)(5,7)(13,15)$} \\
					\hhline{=::=}
					\cellcolor{gray!30} \setlength{\fboxrule}{5pt}\fcolorbox{gray!30}{gray!30}{$\zeta_{3}$} & \cellcolor{gray!30} \setlength{\fboxrule}{5pt}\fcolorbox{gray!30}{gray!30}{$(1,6)(2,5)(3,7)(4,8)(9,10)(13,14)$} \\
					\hhline{=::=}
					\setlength{\fboxrule}{5pt}\fcolorbox{white}{white}{$\zeta_{3}^2$} & \setlength{\fboxrule}{5pt}\fcolorbox{white}{white}{$(1,11)(2,10)(3,9)(4,12)(5,7)(13,15)$} \\
				\end{tabular}
			}
		\end{varwidth}
	\end{center}
\end{table}

The group $G$ generated by the four permutations above acts non-transitively on $\{1,2,\ldots,16\}$. In fact,
\[
\{1,2,\ldots,16\}=\{1,6,11\}\cup\{2,3,5,7,9,10\}\cup \{4,8,12\} \cup \{13,14,15\}\cup \{16\}
\] are the orbits of the action.
The group $G$ preserves each of these five subsets and acts transitively on each of them. Consequently, $\tylda{C_{1}\times_{\PP^{1}} C_{2}}$ decomposes into $5$ connected components whose degrees over $\PP^{1}$ are $3,$ $6,$ $3,$ $3,$ and $1.$
\end{Exm}
Over the four branch points, $\tylda{C_{1}\times_{\PP^{1}} C_{2}}$ has six ramification points of index $2$. The numbers of ramification points on the five components are $4,$ $12,$ $4,$ $4,$ and $0,$ respectively. Hence the genera are $0,$ $1,$ $0,$ $0,$ and $0.$

\begin{Rem}
The advantage of this monodromy-based approach is that it requires only the Weierstrass equation. Because it uses Chen-Ruan cohomology, the method extends to a much wider range of examples.
\end{Rem}

\begin{Rem}
	We have discussed only the second summand
	\[\sum_{C\in  \Lambda(\tylda{C_{1}\times_{\PP^1} C_{2}})}g(C)+g(C_{0}) \]
	of the formula for $H^{1,2}(Y)$. For the second summand $H^{1,2}\left(\tilde X\right)^{\langle i \rangle}$ we use an isomorphism of $H^{1,2}\left(\tilde X\right)$ with the space $H^{1}(\mathcal X)$ of infinitesimal deformations of $X$. Consequently, the dimension of the space $H^{1,2}\left(\tilde X\right)^{\langle i \rangle}$ depends not only on the elliptic surfaces $S_{1}$ and $S_{2}$ but also the sections defining involutions. We shall however use this construction only in the examples with $h^{1,2}(\tilde X) = 0$.
\end{Rem}
\section{Applications}
\subsection{Rigid Calabi--Yau threefolds}
We record configurations of rational elliptic surfaces relevant to rigid
Calabi--Yau fiber products and Kummer fibrations at the primes
$p=5,7,23,61$.  The surfaces come from Herfurtner's list of rational
elliptic surfaces with four singular fibers \cite{Herfurtner}, followed by
a base automorphism as in \cite{Schuett}.  The tables establish the
geometric collision and monodromy data; a claim about the exact set of bad
primes additionally requires an integral model and is made only where
that arithmetic verification is cited.
To simplify notation, we shall only list types of singular fibers of both surfaces and their position.

\medskip
\subsubsection*{The prime $p=5$}
Consider the elliptic surface with singular fibers of types
$\mathrm{I_{5}}$, $\mathrm{I}_{5}$, $\mathrm{I}_{1}$, and $\mathrm{I}_{1}$.
\begin{table}[H]
	\centering
	\renewcommand{\arraystretch}{1.2}
	\begin{tabular}{lcccc}
		\toprule
		Singular fiber
		& $\mathrm{I_{5}}$
		& $\mathrm{I}_{5}$
		& $\mathrm{I}_{1}$ & $\mathrm{I}_{1}$ \\
		\midrule
		Point on the base & $t=\infty$ & $t=0$ & $t=\frac12(11+5\sqrt5)$ & $t=\frac12(11-5\sqrt5)$ \\
		\bottomrule
	\end{tabular}
\end{table}
This is the universal surface for the modular group $\Gamma_{1}(5)$.
A projective crepant resolution of its self-fiber product is a rigid Calabi-Yau threefold with bad prime 5.

\medskip
\subsubsection*{The prime $p=7$}
Consider the elliptic surface with singular fibers of types
$\mathrm{III}$, $\mathrm{I}_{7}$, $\mathrm{I}_{1}$, and $\mathrm{I}_{1}$.
\begin{table}[H]
	\centering
	\renewcommand{\arraystretch}{1.2}
	\begin{tabular}{lcccc}
		\toprule
		Singular fiber
		& $\mathrm{III}$
		& $\mathrm{I}_{7}$
		& $\mathrm{I}_{1}$ & $\mathrm{I}_{1}$ \\
				\midrule
		Local monodromies &  (2,3)& (1,2)& $(1,3)$ & (1,2)\\
		\midrule
		Point on the base & $t=\infty$ & $t=0$ & $t=-\frac{13}{64}-\frac{7 \sqrt{-7}}{64}
		$ & $t=-\frac{13}{64}-\frac{7 \sqrt{-7}}{64}$ \\
		\bottomrule
	\end{tabular}
\end{table}
From the local monodromy we deduce that the fixed curve has four components with genus equal 0 and one with genus equal 1 (we give details of similar computations in a more complicated example $p=23$). Consequently $h^{1,2}(Y)=1$.

%

\medskip
\subsubsection*{The prime $p=23$}
The following pair has an additional common-fiber collision modulo $23$;
the primes $2$ and $3$ also require separate analysis for an integral
model.
\begin{table}[H]
	\centering
	\renewcommand{\arraystretch}{1.2}
	\begin{tabular}{lccccc}
		\toprule
		Singular fiber
		& $\mathrm{IV}$
		& $\mathrm{I}_{5}$
		& $\mathrm{I}_{2}$ & $\mathrm{I}_{1}$ & $\textup{\textbf{-}}$ \\
		\midrule
		\toprule
		Singular fiber
		& $\mathrm{IV}$
		& $\mathrm{I}_{2}$
		& $\mathrm{I}_{5}$ & $\textup{\textbf{-}}$ & $\mathrm{I}_{1}$ \\
		\midrule
		Point on the base & $t=\infty$ & $t=0$ & $t=1$ & $t=\frac{2}{27}$  & $t=\frac{25}{27}$\\
		\bottomrule
	\end{tabular}
\end{table}


\medskip
\subsubsection*{The prime $p=61$}
The following pair gives the configuration labelled by $p=61$ (in
addition to the primes $2$ and $3$ arising from the displayed integral
models):
\begin{table}[H]
	\centering
	\renewcommand{\arraystretch}{1.2}
	\begin{tabular}{lccccc}
		\toprule
		Singular fiber
		& $\mathrm{III}$
		& $\mathrm{I}_{5}$
		& $\mathrm{I}_{3}$ & $\mathrm{I}_{1}$ & $\textup{\textbf{-}}$ \\
		\midrule
		Local monodromies &  (1,2)& (1,2)& $(2,3)$ & (2,3)& $\operatorname{id}$\\
		\midrule
		\toprule
		Singular fiber
		& $\mathrm{III}$
		& $\mathrm{I}_{3}$
(4,8)		& $\mathrm{I}_{5}$ & $\textup{\textbf{-}}$ & $\mathrm{I}_{1}$ \\
		\midrule
		Local monodromies &  (1,3)& (1,2)& $(1,3)$ & $\operatorname{id}$ & (1,2)\\
		\midrule
		Point on the base & $t=\infty$ & $t=0$ & $t=1$ & $t=\frac{3}{128}$  & $t=\frac{125}{128}$\\
		\bottomrule
	\end{tabular}

\medskip
	Then the local monodromy of fixed curve of the involution on the fiber product is
	\begin{center}
	\begin{varwidth}{\textheight}\rowcolors{2}{gray!25}{white}
		\rowcolors{2}{gray!25}{white}
		\resizebox{0.35\textheight}{!}{
			\begin{tabular}{c||c}
				\hhline{~|~}
				\cellcolor{gray!30}\setlength{\fboxrule}{5pt}\fcolorbox{gray!30}{gray!30}{$\infty$} & \cellcolor{gray!30} \setlength{\fboxrule}{5pt}\fcolorbox{gray!30}{gray!30}{$(1,7)(2,6)(3,5)(4,8)(9,11)(13,15)$} \\
				\hhline{=::=}
				\cellcolor{white} \setlength{\fboxrule}{5pt}\fcolorbox{white}{white}{$0$} & \cellcolor{white} \setlength{\fboxrule}{5pt}\fcolorbox{white}{white}{$(1,6)(2,5)(3,7)(4,8)(9,10)(13,14)$} \\
				\hhline{=::=}
				\cellcolor{gray!30}\setlength{\fboxrule}{5pt}\fcolorbox{gray!30}{gray!30}{$1$} & \cellcolor{gray!30}\setlength{\fboxrule}{5pt}\fcolorbox{gray!30}{gray!30}{$(1,3)(5,11)(6,10)(7,9)(8,12)(13,15)$} \\
				\hhline{=::=}
				\cellcolor{white} \setlength{\fboxrule}{5pt}\fcolorbox{white}{white}{$\frac{3}{128}$} & \cellcolor{white} \setlength{\fboxrule}{5pt}\fcolorbox{white}{white}{$(5,9)(6,10)(7,11)(8,12)$} \\
				\hhline{=::=}
				\cellcolor{gray!30}\setlength{\fboxrule}{5pt}\fcolorbox{gray!30}{gray!30}{$\frac{125}{128}$} & \cellcolor{gray!30}\setlength{\fboxrule}{5pt}\fcolorbox{gray!30}{gray!30}{$(1,2)(5,6)(9,10)(13,14)$} \\
			\end{tabular}
		}
	\end{varwidth}
\end{center}
\end{table}
Action of this group has four orbits
\[\{1,2,3,5,6,7,9,10,11,12\}, \{4,8,12\}, \{13,14,15\}, \{16\}.\]
Consequently the fixed curve has four components and its Euler characteristic equals
\[16\times(2-5)+10+10+10+12+12 = 6 = 2+2+2+0.\]
It means that the fixed curve has three components of genus 0 and one components of genus 1 and finally $h^{1,2}(Y)=1$.

\subsection{Kummer fibrations in positive characteristic}
\label{sec10}

Section~\ref{sec6} discussed three self-fiber products in characteristics
$5$ and $7$.  We now consider the standard Kummer involution for the fiber product in Example~\ref{ex:61} induced by
$x\mapsto-x$ on both factors.  This quotient is covered by the
projective cases of Theorem~\ref{thm:kummer}.
%
%
\begin{Exm}
The elliptic surface (example~\ref{ex:61}) over $\FF_5$ given by the Weierstrass equation  
given by
\[
S_5\colon y^2=x^3-2tx+t^4.
\]
Its singular fibers over $t=\infty,0,1$ are of types
$\mathrm{IV}$, $\mathrm{equationIII}$, and
$\mathrm{I}_5$, respectively.

By Theorem~\ref{thm:kummer} there exists a projective crepant resolution of the quotient of the fiber product by the involution, i.e. there is
a Kummer fibration associated to that fiber product which is a smooth projective Calabi--Yau threefold $X_{5}$ defined over $\FF_{5}$.

\end{Exm}

\textbf{Acknowledgements.} The second author was partially supported by National Science Centre, Poland, grant No. 2020/39/B/ST1/03358.

\clearpage

\begingroup
\raggedbottom
\setlength{\textfloatsep}{8pt plus 2pt minus 2pt}
\setlength{\floatsep}{7pt plus 2pt minus 2pt}
\setlength{\intextsep}{7pt plus 2pt minus 2pt}
\renewcommand{\topfraction}{0.95}
\renewcommand{\textfraction}{0.05}
\renewcommand{\floatpagefraction}{0.80}
\setcounter{topnumber}{6}
\setcounter{totalnumber}{10}
\tikzset{
  every picture/.style={scale=0.84},
  legendtext/.style={font=\footnotesize},
  titlelabel/.style={font=\large\bfseries}
}

\section{Appendix -- Kodaira fibers and local analytic models}
\label{sec:app}
\vspace{10mm}

\begin{figure}[H]
\centering
\begin{minipage}[t]{0.32\textwidth}
\hbox{$\mathrm{II}_{}^{}$ ($N_c=1,\;\; e=2$)}
\vspace{4mm}
\centering
\begin{tikzpicture}[x=0.62cm,y=0.62cm]
\draw[fiber, samples=160, smooth, variable=\t, domain=-1.45:1.45]
  plot ({1.90*\t*\t-1.10},{0.55*\t*\t*\t});
\node[markD1] at (-1.10,0.00) {};
\node[compmult] at (1.65,1.00) {$1$};
\node[markD1] at (-2.00,-2.55) {};
\node[legendtext,xshift=0.12cm] at (-1.375,-2.55) {---};
\node[legendtext, anchor=west] at (-0.75,-2.55) {$x^{2}+y^{3}-b$};
\end{tikzpicture}
\end{minipage}\hfill
\begin{minipage}[t]{0.32\textwidth}
\hbox{$\mathrm{III}_{}^{}$ ($N_c=2,\;\; e=3$)}
\vspace{4mm}
\centering
\begin{tikzpicture}[x=0.62cm,y=0.62cm]
\draw[fiber, samples=160, smooth, variable=\x, domain=-2.9:2.9]
  plot (\x,{0.22*\x*\x});
\draw[fiber, samples=160, smooth, variable=\x, domain=-2.9:2.9]
  plot (\x,{-0.22*\x*\x});
\node[markD2] at (0.00,0.00) {};
\node[compmult] at (2.10,1.05) {$1$};
\node[compmult] at (2.10,-1.05) {$1$};
\node[markD2] at (-2.00,-2.55) {};
\node[legendtext,xshift=0.12cm] at (-1.375,-2.55) {---};
\node[legendtext, anchor=west] at (-0.75,-2.55) {$x^{2}+y^{4}-b$};
\end{tikzpicture}
\end{minipage}\hfill
\begin{minipage}[t]{0.32\textwidth}
\hbox{$\mathrm{IV}_{}^{}$ ($N_c=3,\;\; e=4$)}
\vspace{4mm}
\centering
\begin{tikzpicture}[x=0.62cm,y=0.62cm]
\draw[fiber] (-2.85,0.00) -- (2.85,0.00);
\draw[fiber] (-2.20,-1.75) -- (2.20,1.75);
\draw[fiber] (-2.20,1.75) -- (2.20,-1.75);
\node[markD3] at (0.00,0.00) {};
\node[compmult] at (2.10,0.00) {$1$};
\node[compmult] at (1.55,1.25) {$1$};
\node[compmult] at (1.55,-1.25) {$1$};
\node[markD3] at (-2.00,-2.55) {};
\node[legendtext,xshift=0.12cm] at (-1.375,-2.55) {---};
\node[legendtext, anchor=west] at (-0.75,-2.55) {$x^{3}+y^{3}-b$};
\end{tikzpicture}
\end{minipage}

\vspace{10mm}

\label{fig:reduced-compact}
\end{figure}
\begin{figure}[H]
\hbox{$\mathrm{I}_{n}^{}$ ($N_c=n,\;\; e=n$)}
\vspace{6mm}
\centering
  \begin{tikzpicture}[x=1cm,y=1cm,line cap=round,line join=round]
  \coordinate (P1) at (-1.5307,  3.6955);
  \coordinate (P2) at (-3.6955,  1.5307);
  \coordinate (P3) at (-3.6955, -1.5307);
  \coordinate (P4) at (-1.5307, -3.6955);
  \coordinate (P5) at ( 1.5307, -3.6955);
  \coordinate (P6) at ( 3.6955, -1.5307);
  \coordinate (P7) at ( 3.6955,  1.5307);
  \coordinate (P8) at ( 1.5307,  3.6955);

  \draw[fiber] ($(P1)!-0.18!(P2)$) -- ($(P2)!-0.18!(P1)$);
  \draw[fiber] ($(P2)!-0.18!(P3)$) -- ($(P3)!-0.18!(P2)$);
  \draw[fiber] ($(P3)!-0.18!(P4)$) -- ($(P4)!-0.18!(P3)$);
  \draw[fiber] ($(P4)!-0.18!(P5)$) -- ($(P5)!-0.18!(P4)$);
  \draw[fiber] ($(P5)!-0.18!(P6)$) -- ($(P6)!-0.18!(P5)$);
  \draw[fiber] ($(P6)!-0.18!(P7)$) -- ($(P7)!-0.18!(P6)$);
  \draw[fiber] ($(P7)!-0.18!(P8)$) -- ($(P8)!-0.18!(P7)$);

  \node[markA2] at (P1) {};
  \node[markA2] at (P2) {};
  \node[markA2] at (P3) {};
  \node[markA2] at (P4) {};
  \node[markA2] at (P5) {};
  \node[markA2] at (P6) {};
  \node[markA2] at (P7) {};
  \node[markA2] at (P8) {};

  \node[compmult] at ($(P8)!0.5!(P1)$) {$\cdots$};

  \node[compmult] at ($(P1)!0.5!(P2)$) {$1$};
  \node[compmult] at ($(P2)!0.5!(P3)+(-0.42,0)$) {$1$};
  \node[compmult] at ($(P3)!0.5!(P4)$) {$1$};
  \node[compmult] at ($(P4)!0.5!(P5)$) {$1$};
  \node[compmult] at ($(P5)!0.5!(P6)$) {$1$};
  \node[compmult] at ($(P6)!0.5!(P7)+(0.42,0)$) {$1$};
  \node[compmult] at ($(P7)!0.5!(P8)$) {$1$};

  \draw[decorate,decoration={brace,mirror,amplitude=6pt},line width=0.9pt]
    (-3.95,-4.25) -- (3.95,-4.25)
    node[midway,below=7pt,legendtext] {$n$ components in the cycle};

  \node[markA2] at (5.95,0.20) {};
  \node[legendtext,xshift=0.12cm] at (6.475,0.20) {---};
  \node[legendtext, anchor=west] at (7.00,0.20) {$xy-b$};

  \end{tikzpicture}
\label{fig:In-schematic}
\end{figure}

\vspace{10mm}

\begin{figure}[H]
\hbox{$\mathrm{I}_{n}^{*}$ ($N_c=n+5,\;\; e=n+6$)}
\vspace{10mm}
\centering
\begin{tikzpicture}[x=1.05cm,y=1.05cm]
\draw[fiber,name path=C1] (-5.60,-1.15) -- (-3.80,1.15);
\draw[fiber,name path=C2] (-4.80,1.15) -- (-3.00,-1.15);
\draw[fiber,name path=C3] (-4.00,-1.15) -- (-2.20,1.15);
\draw[fiber,name path=C4] (-3.20,1.15) -- (-1.40,-1.15);
\draw[fiber,name path=C5] (0.90,-1.15) -- (2.70,1.15);
\draw[fiber,name path=C6] (1.70,1.15) -- (3.50,-1.15);
\draw[fiber,name path=C7] (2.50,-1.15) -- (4.30,1.15);
\draw[fiber,name path=C8] (3.30,1.15) -- (5.10,-1.15);
\node[titlelabel] at (-0.25,0.05) {$\cdots$};

\coordinate (M1) at ($(-5.60,-1.15)!0.5!(-3.80,1.15)$);
\coordinate (M2) at ($(-4.80,1.15)!0.5!(-3.00,-1.15)$);
\coordinate (M3) at ($(-4.00,-1.15)!0.5!(-2.20,1.15)$);
\coordinate (M4) at ($(-3.20,1.15)!0.5!(-1.40,-1.15)$);
\coordinate (M5) at ($(0.90,-1.15)!0.5!(2.70,1.15)$);
\coordinate (M6) at ($(1.70,1.15)!0.5!(3.50,-1.15)$);
\coordinate (M7) at ($(2.50,-1.15)!0.5!(4.30,1.15)$);
\coordinate (M8) at ($(3.30,1.15)!0.5!(5.10,-1.15)$);

\draw[fiber,name path=S1] (-7.45,0.60) -- (-5.10,-1.10);
\draw[fiber,name path=S2] (-7.15,0.98) -- (-4.80,-0.72);
\draw[fiber,name path=S3] (4.30,-0.72) -- (6.65,0.98);
\draw[fiber,name path=S4] (4.60,-1.10) -- (6.95,0.60);

\path[name intersections={of=C1 and C2, by=I12}];
\path[name intersections={of=C2 and C3, by=I23}];
\path[name intersections={of=C3 and C4, by=I34}];
\path[name intersections={of=C5 and C6, by=I56}];
\path[name intersections={of=C6 and C7, by=I67}];
\path[name intersections={of=C7 and C8, by=I78}];
\path[name intersections={of=C1 and S1, by=J11}];
\path[name intersections={of=C2 and S2, by=J22}];
\path[name intersections={of=C1 and S2, by=J12}];
\path[name intersections={of=C7 and S3, by=J33}];
\path[name intersections={of=C8 and S4, by=J44}];
\path[name intersections={of=C8 and S3, by=J43}];

\coordinate (L1) at ($(J11)!0.5!(I12)+(0,0.18)$);
\coordinate (L2) at ($(I12)!0.5!(I23)$);
\coordinate (L3) at ($(I23)!0.5!(I34)$);
\coordinate (L4) at ($(I34)!0.5!(-1.40,-1.15)$);
\coordinate (L6) at ($(I56)!0.5!(I67)$);
\coordinate (L7) at ($(I67)!0.5!(I78)$);
\coordinate (L8) at ($(I78)!0.5!(J44)+(0,0.18)$);

\node[markA1] at (I12) {};
\node[markA1] at (I23) {};
\node[markA1] at (I34) {};
\node[markA1] at (I56) {};
\node[markA1] at (I67) {};
\node[markA1] at (I78) {};
\node[markA2] at (J11) {};
\node[markA2] at (J22) {};
\node[markA2] at (J12) {};
\node[markA2] at (J33) {};
\node[markA2] at (J44) {};
\node[markA2] at (J43) {};

\node[compmult] at ($(-7.45,0.60)!0.32!(-5.10,-1.10)$) {$1$};
\node[compmult] at ($(-7.15,0.98)!0.32!(-4.80,-0.72)$) {$1$};
\node[compmult] at ($(4.30,-0.72)!0.68!(6.65,0.98)$) {$1$};
\node[compmult] at ($(4.60,-1.10)!0.68!(6.95,0.60)$) {$1$};
\node[compmult] at (L1) {$2$};
\node[compmult] at (L2) {$2$};
\node[compmult] at (L3) {$2$};
\node[compmult] at (L4) {$2$};
\node[compmult] at (M5) {$2$};
\node[compmult] at (L6) {$2$};
\node[compmult] at (L7) {$2$};
\node[compmult] at (L8) {$2$};

\draw[decorate,decoration={brace,mirror,amplitude=6pt},line width=0.9pt]
  (-5.60,-1.60) -- (5.60,-1.60)
  node[midway,below=7pt,legendtext] {$n+1$ components in the chain};

\node[markA1] at (-4.70,-3) {};
\node[legendtext,xshift=0.12cm] at (-4.175,-3) {---};
\node[legendtext, anchor=west] at (-3.65,-3) {$x^{2}y^{2}-b$};
\node[markA2] at (-1.15,-3) {};
\node[legendtext,xshift=0.12cm] at (-0.625,-3) {---};
\node[legendtext, anchor=west] at (-0.10,-3) {$xy^{2}-b$};

\end{tikzpicture}
\label{fig:Instar-schematic}
\end{figure}

\begin{figure}[H]
\hbox{$\mathrm{IV}^{*}$ ($N_c=7,\;\; e=8$)}
\vspace{5mm}
\centering
\begin{tikzpicture}[x=1.35cm,y=1.35cm]


\draw[fiber] (0,3.0) -- (0,-3.0);

\draw[fiber] (-1.0, 1.55) -- (2.25, 2.14);
\draw[fiber] (1.05, 2.12) -- (3.90, 1.62);

\draw[fiber] (-1.0, 0.00) -- (2.25, 0.59);
\draw[fiber] (1.05, 0.57) -- (3.90, 0.07);

\draw[fiber] (-1.0,-1.55) -- (2.25,-0.96);
\draw[fiber] (1.05,-0.97) -- (3.90,-1.47);


\node[markA3] (p1) at (0, 1.73) {};
\node[markA3] (p2) at (0, 0.18) {};
\node[markA3] (p5) at (0,-1.37) {};

\node[markA2] (p3) at (1.60, 2.02) {};
\node[markA2] (p4) at (1.60, 0.47) {};
\node[markA2] (p6) at (1.63,-1.07) {};




\node[compmult] at (0.00, 2.35) {$3$};

\node[compmult] at (0.82, 1.90) {$2$};
\node[compmult] at (3.05, 1.77) {$1$};

\node[compmult] at (0.82, 0.35) {$2$};
\node[compmult] at (2.95, 0.24) {$1$};

\node[compmult] at (0.77,-1.21) {$2$};
\node[compmult] at (2.98,-1.31) {$1$};


\node[markA2] at (6.15, 0.95) {};
\node[legendtext,xshift=0.12cm] at (6.70, 0.95) {---};
\node[legendtext, anchor=west] at (7.25, 0.95) {$xy^{2}-b$};

\node[markA3] at (6.15, 0.20) {};
\node[legendtext,xshift=0.12cm] at (6.70, 0.20) {---};
\node[legendtext, anchor=west] at (7.25, 0.20) {$x^{2}y^{3}-b$};

\end{tikzpicture}
\label{fig:IVstar-pretty}
\end{figure}

\begin{figure}[H]
\hbox{$\mathrm{III}^{*}$ ($N_c=8,\;\; e=9$)}
\vspace{2mm}
\centering
\begin{tikzpicture}[x=1.25cm,y=1.25cm,scale=1.08]

\begin{scope}[xshift=-1.2cm]
\draw[fiber] (0,2.7) -- (0,-2.7);             
\draw[fiber] (0.85,2.7) -- (0.85,0.85);       
\draw[fiber] (0.85,-0.85) -- (0.85,-2.7);     
\draw[fiber] (-0.35,1.80) -- (1.45,1.80);     
\draw[fiber] (0.62,1.30) -- (1.72,1.30);      
\draw[fiber] (-0.35,0.00) -- (1.45,0.00);     
\draw[fiber] (0.62,-1.30) -- (1.72,-1.30);    
\draw[fiber] (-0.35,-1.80) -- (1.45,-1.80);   

\node[compmult] at (0.00,2.45) {$4$};   
\node[compmult] at (0.42,1.80) {$3$};   
\node[compmult] at (0.85,2.35) {$2$};   
\node[compmult] at (1.27,1.30) {$1$};   
\node[compmult] at (0.55,0.00) {$2$};   
\node[compmult] at (0.42,-1.80) {$3$};  
\node[compmult] at (0.85,-2.35) {$2$};  
\node[compmult] at (1.27,-1.30) {$1$};  

\node[markB7] at (0,1.80) {};      
\node[markB6] at (0,0.00) {};      
\node[markB7] at (0,-1.80) {};     
\node[markB5] at (0.85,1.80) {};   
\node[markB4] at (0.85,1.30) {};   
\node[markB5] at (0.85,-1.80) {};  
\node[markB4] at (0.85,-1.30) {};  

\end{scope}

\node[markB4] at (2.40,0.20) {};
\node[legendtext,xshift=0.12cm] at (2.925,0.20) {---};
\node[legendtext, anchor=west] at (3.45,0.20) {$xy^{2}-b$};

\node[markB5] at (2.40,2.15) {};
\node[legendtext,xshift=0.12cm] at (2.925,2.15) {---};
\node[legendtext, anchor=west] at (3.45,2.15) {$x^{2}y^{3}-b$};

\node[markB6] at (2.40,1.50) {};
\node[legendtext,xshift=0.12cm] at (2.925,1.50) {---};
\node[legendtext, anchor=west] at (3.45,1.50) {$x^{2}y^{4}-b$};

\node[markB7] at (2.40,0.85) {};
\node[legendtext,xshift=0.12cm] at (2.925,0.85) {---};
\node[legendtext, anchor=west] at (3.45,0.85) {$x^{3}y^{4}-b$};

\end{tikzpicture}
\label{fig:IIIstar-pretty}
\end{figure}

\begin{figure}[H]
\hbox{$\mathrm{II}^{*}$ ($N_c=9,\;\; e=10$)}
\vspace{2mm}
\centering
\begin{tikzpicture}[x=1.08cm,y=1.08cm,scale=1.08]
\begin{scope}
\draw[fiber] (0,3.2) -- (0,-3.2);             
\draw[fiber] (-0.3,2.2) -- (2.8,2.2);         
\draw[fiber] (-0.3,0.0) -- (4.8,0.0);         
\draw[fiber] (-0.3,-2.2) -- (4.8,-2.2);       
\draw[fiber] (3.2,-1.0) -- (5.8,1.6);         
\draw[fiber] (3.2,-3.2) -- (5.8,-0.6);        
\draw[fiber] (4.8,-1.1) -- (7.9,-1.1);        
\draw[fiber] (6.8,-2.0) -- (8.9,0.1);         
\draw[fiber] (8.2,0.0) -- (10.8,0.0);         
\draw[fiber] (8.4,-0.4) -- (9.6,0.8);         

\node[compmult] at (0.00,-0.55) {$6$};
\node[compmult] at (0.90,2.20) {$3$};
\node[compmult] at (4.90,0.70) {$2$};
\node[compmult] at (3.70,-2.70) {$4$};
\node[compmult] at (8.25,-0.55) {$2$};
\node[compmult] at (6.50,-1.10) {$3$};
\node[compmult] at (9.80,0.00) {$1$};
\node[compmult] at (2.10,0.00) {$4$};
\node[compmult] at (3.00,-2.20) {$5$};

\node[markC10] at (0,2.2) {};     
\node[markC12] at (0,0.0) {};     
\node[markC13] at (0,-2.2) {};    
\node[markC9]  at (4.2,0.0) {};   
\node[markC11] at (4.2,-2.2) {};  
\node[markC8]  at (5.3,-1.1) {};  
\node[markC7]  at (7.7,-1.1) {};  
\node[markC6]  at (8.8,0.0) {};   
\end{scope}

\node[markC6]  at (0.0,-4.75) {};
\node[legendtext, anchor=west] at (0.45,-4.75) {---};
\node[legendtext, anchor=west] at (1.00,-4.75) {$x^{2}y-b$};
\node[markC7]  at (0.0,-5.40) {};
\node[legendtext, anchor=west] at (0.45,-5.40) {---};
\node[legendtext, anchor=west] at (1.00,-5.40) {$x^{2}y^{3}-b$};

\node[markC8]  at (3.2,-4.75) {};
\node[legendtext, anchor=west] at (3.65,-4.75) {---};
\node[legendtext, anchor=west] at (4.20,-4.75) {$x^{3}y^{4}-b$};
\node[markC9]  at (3.2,-5.40) {};
\node[legendtext, anchor=west] at (3.65,-5.40) {---};
\node[legendtext, anchor=west] at (4.20,-5.40) {$x^{4}y^{2}-b$};

\node[markC10] at (6.4,-4.75) {};
\node[legendtext, anchor=west] at (6.85,-4.75) {---};
\node[legendtext, anchor=west] at (7.40,-4.75) {$x^{3}y^{6}-b$};
\node[markC11] at (6.4,-5.40) {};
\node[legendtext, anchor=west] at (6.85,-5.40) {---};
\node[legendtext, anchor=west] at (7.40,-5.40) {$x^{5}y^{4}-b$};

\node[markC12] at (9.6,-4.75) {};
\node[legendtext, anchor=west] at (10.05,-4.75) {---};
\node[legendtext, anchor=west] at (10.60,-4.75) {$x^{4}y^{6}-b$};
\node[markC13] at (9.6,-5.40) {};
\node[legendtext, anchor=west] at (10.05,-5.40) {---};
\node[legendtext, anchor=west] at (10.60,-5.40) {$x^{6}y^{5}-b$};

\end{tikzpicture}
\label{fig:IIstar-pretty}
\end{figure}

\FloatBarrier
\endgroup

\clearpage
\begin{thebibliography}{99}
	\bibitem{Artin} M.~Artin, {\em Algebraic construction of Brieskorn's resolutions.}  J. Algebra  29  (1974), 330--348.

	\bibitem{Clemens} C.~H.~Clemens, \emph{Double solids.} Adv.\ in
	Math. \textbf{47} (1983), 107--230.

	\bibitem{Herfurtner}S.~Herfurtner,
	\emph{Elliptic surfaces with four singular fibers}.
	Math. Ann. 291 (1991), no. 2, 319--342.

	\bibitem{HirokadoItoSaito}
	M.~Hirokado, H.~Ito, N.~Saito,
	\emph{Calabi--Yau threefolds arising from fiber products of rational
	quasi-elliptic surfaces, I}.
	Ark. Mat. 45 (2007), no. 2, 279--296.

	\bibitem{KK}G.~Kapustka, M.~Kapustka,
	\emph{Fiber products of elliptic surfaces with section and associated Kummer fibrations}
	Internat. J. Math. 20 (2009), no. 4, 401--426.

	\bibitem{KapustkaKum} M.~Kapustka,
	\emph{Correspondences between modular Calabi--Yau fiber products},
	Manuscripta Math. 130 (2009), no. 1, 121--135.
	\bibitem{Kollar}J. Koll\'ar,
	Flops.
	Nagoya Math. J. 113 (1989), 15--36.

	\bibitem{KollarMori} J.~Koll\'ar, S.~Mori,
	\emph{Birational geometry of algebraic varieties},
	Cambridge Tracts in Mathematics 134, Cambridge University Press,
	Cambridge, 1998.

	\bibitem{Laufer}B. B. Laufer, \emph{On $\mathbb C\mathbb P^{1}$ as an exceptional set}. Recent developments in several complex variables (Proc. Conf., Princeton Univ., Princeton, N. J., 1979), pp. 261--275
	Ann. of Math. Stud., No. 100.
	\bibitem{Lin}H.-W.~Lin,
		\emph{On crepant resolution of some hypersurface singularities and a criterion for UFD}.
		Trans. Amer. Math. Soc. 354 (2002), no. 5, 1861--1868.
	\bibitem{LANG} W.~E.~Lang,
	\emph{Extremal rational elliptic surfaces in characteristic \(p\). II:
	Surfaces with three or fewer singular fibres}.
	Ark. Mat. 32 (1994), no. 2, 423--448.
	\bibitem{MirandaPersson} R.~Miranda, U.~Persson,
	\emph{On extremal rational elliptic surfaces}.
	Math. Z. 193 (1986), 537--558.
	\bibitem{Miranda} R.~Miranda,
	\emph{The basic theory of elliptic surfaces},
	Dottorato di Ricerca in Matematica, ETS Editrice, Pisa, 1989.
	\bibitem{Schoen} C.~Schoen, {\em On Fiber Products of
		Rational Elliptic Surfaces with Section}, Math. Z. \textbf{197}
	(1988),
	177--199.

	\bibitem{Batyrev} V.~V.~Batyrev,
	\emph{Birational Calabi--Yau $n$-folds have equal Betti numbers},
	in: New trends in algebraic geometry (Warwick, 1996),
	London Math. Soc. Lecture Note Ser. 264, Cambridge Univ. Press, 1999, 1--11.

	\bibitem{BatyrevDais} V.~V.~Batyrev, D.~I.~Dais,
	\emph{Strong McKay correspondence, string-theoretic Hodge numbers and
	mirror symmetry},
	Topology 35 (1996), no.~4, 901--929.

	\bibitem{CR} W.~Chen, Y.~Ruan,
	\emph{A new cohomology theory of orbifold},
	Comm. Math. Phys. 248 (2004), no. 1, 1--31.

	\bibitem{mirandacurves} R.~Miranda,
	\emph{Algebraic curves and Riemann surfaces},
	Graduate Studies in Mathematics 5, American Mathematical Society, Providence, RI, 1995.

	\bibitem{Schuett}M.~Sch\"utt, {\em New examples of modular rigid
	Calabi--Yau threefolds}.  Collect. Math.  55  (2004),  no. 2,
	219--228.

	\bibitem{Szafarewicz} I.~R.~\v Safarevi\v c,
	\emph{Basic algebraic geometry 1. Varieties in projective space},
	Springer-Verlag, Berlin, 1994.

	\bibitem{Werner} J. Werner, \emph{Kleine Aufl\"osungen spezieller
		dreidimensionaler Variet\"aten}, Bonner Math. Schriften \textbf{186}
	(1987).

	\bibitem{CynkVanStraten} S.~Cynk, D.~van Straten,
	\emph{Small resolutions and non-liftable Calabi--Yau threefolds}.
	Manuscripta Math. 130 (2009), no. 2, 233--249.

	\bibitem{NorikoUpdate} N.~Yui,
	\emph{Update on the modularity of Calabi--Yau varieties}, with an
	appendix by H.~Verrill,
	in: \emph{Calabi--Yau Varieties and Mirror Symmetry},
	Fields Inst. Commun. 38, Amer. Math. Soc., Providence, RI, 2003,
	307--362.
\end{thebibliography}
\end{document}